\documentclass[11pt]{article}
\usepackage{amsmath}
\numberwithin{equation}{section}
\usepackage[english]{babel}
\usepackage[utf8]{inputenc}
\usepackage{amssymb,amsmath,amsfonts,amsthm,latexsym,enumerate,url,cases}
\usepackage{graphicx}
\usepackage[colorinlistoftodos]{todonotes}
\usepackage[affil-it]{authblk}
\usepackage{hyperref}
\usepackage{indentfirst}
\usepackage{hyperref}
\usepackage{graphicx}
\usepackage{hyperref}
\usepackage{fullpage}
\usepackage{cite}
\allowdisplaybreaks[4]
\newtheorem{theorem}{Theorem}[section] %
\newtheorem{lemma}{Lemma}[section] %
\newtheorem{corollary}{Corollary}[section] %
\newtheorem{definition}{Definition}[section] %
\newtheorem{remark}{Remark}[section] %

\usepackage{stmaryrd}
\SetSymbolFont{stmry}{bold}{U}{stmry}{m}{n}
\newcommand{\vertiii}[1]{{\left\vert\kern-0.25ex\left\vert\kern-0.25ex\left\vert #1
    \right\vert\kern-0.25ex\right\vert\kern-0.25ex\right\vert}}
\title{Wong-Zakai approximation for the dynamics of stochastic evolution equations driven by rough path with Hurst index $H\in(\frac{1}{3},\frac{1}{2}]$\thanks
{This work is supported in part by a NSFC Grant No. 12171084 and the fundamental Research Funds for the Central Universities.}}

\author{Qiyong Cao$^1$\thanks{E-mail addresses: xjlyysx@163.com (Q. Cao)}, Hongjun Gao$^2$\thanks{Corresponding author: hjgao@seu.edu.cn(H. Gao).}\\ {\small \it 1. School of Mathematical Sciences, Nanjing Normal University, Nanjing 210023, P. R. China}\\
 {\small \it $2$. School of Mathematics, Southeast University, Nanjing 211189, P. R. China}
}
\date{}
\begin{document}
\maketitle
\begin{abstract}
In this paper, we obtain the existence of random attractors for a class of evolution equations driven by a geometric
fractional Brownian rough path with Hurst index $H\in(\frac{1}{3},\frac{1}{2}]$ and establish the upper semi-continuity of random attractors $\mathcal{A}_{\eta}$  for the approximated systems of the evolution equations.
\end{abstract}


\section{Introduction}
In this paper, we consider the existence of random attractors for a class of evolution equations driven by a geometric
fractional Brownian rough path(GFBRP) with Hurst index $H\in(\frac{1}{3},\frac{1}{2}]$. There are some results about  the existence of random attractors for stochastic (partial)  differential equations driven by a fractional Brownian motion(fBm). For additive noise with Hurst index $H\in(0,1]$, the cohomology  method  transforms stochastic (partial) differential equations  to  random  differential equations, then deterministic method can be used to construct random attractors. We refer to\cite{MR3374736,MR3225217} and references therein.  For nonlinear multiplicative noise, Garrido-Atienza, Maslowski, Schmalfuss\cite{MR2738732} considered  the  existence of random attractors of stochastic differential equations(SDEs) driven by a fBm with Hurst index $H>\frac{1}{2}$, where the stochastic integral  can be defined by the  fractional integral. Gao, Garrido-Atienza, Schmalfuss\cite{MR3226746} extended the existence of random attractors of SDEs to a class of evolution equations driven by an infinite dimensional fBm with Hurst index $H>\frac{1}{2}$, they used the modified H\"{o}lder space $C^{\beta,\thicksim}([0,T],V)$ to overcome semigroup is not H\"{o}lder continuity at zero.  However, there are a few of works for random attractors of SDEs(or SPDEs)  driven by fBm with Hurst index $H\in (\frac{1}{3},\frac{1}{2}]$.   Duc\cite{MR4385780} proved the existence of random attractors for rough differential equations driven by a Gaussian rough path which corresponding to a Gaussian process with $(\frac{1}{3},\frac{1}{2})\ni\nu$-H\"{o}lder regularity a.s., and the drift term is locally Lipschitz and satisfies $\langle f(y),y \rangle\leq \|y\|(D_{1}-D_{2}\|y\|)$. In particular, the Doss-Sussmann technique was used to derive the uniform energy estimate, somehow it can be regard as   cohomology  method.

\smallskip

To the best of our knowledge, for stochastic partial differential equations(SPDEs) driven by multiplicative fBm with $H\in (\frac{1}{3},\frac{1}{2})$, there are only few works on the solutions of SPDEs can generate a random dynamical system. About the global well-posedness of the solution.  Garrido-Atienza, Lu, Schmalfuss\cite{MR2998821} considered a class of evolution equations driven by an infinite dimensional fBm with Hurst index $H\in (\frac{1}{3},\frac{1}{2}]$.  They constructed a second order process which depends on the semigroup $S$ by the fractional integral\cite{MR2998821}, and then the global solution can be obtained.  In addition, they used the fractional integral to construct rough paths and consider the global solutions of a class of evolution equations driven by fBm with Hurst index $H\in (\frac{1}{3},\frac{1}{2}]$  and its  random dynamical system in \cite{MR3479690}. Hesse and Neam\c{t}u \cite{MR4001062} constructed the local solutions for a class of rough evolution equation driven by a rough path via establishing rough integral, and then  obtained the global solutions\cite{MR4097587}. Compared with finite-dimensional case\cite{MR4385780,MR4174393}, the stochastic integrals in the above works have more complex structure and therefore it is vital to define the controlled rough paths  as finite-dimensional case. Thus,
Gerasimovi\v{c}s, Hocquet, Nilssen\cite{MR4299812} established the controlled rough paths over interpolation spaces as finite dimensional case. Using the approach of \cite{MR4299812}, Hesse and Neam\c{t}u\cite{hesse2021global} constructed the global solutions for SPDEs driven by the finite-dimensional rough paths. We will use this framework to construct the random attractors.  In addition, Kuehn and Neam\c{t}u\cite{kuehn2021center} also established center manifold for a class of SPDEs driven by a Gaussian rough path under this framework.

\smallskip

Our first aim  in this paper is to obtain the existence of random attractors for following rough SPDEs
\begin{equation}\label{1.1}
  dy_{t}=Ay_{t}dt+F(y_{t})dt+G(y_{t})d\boldsymbol{W}_{t},
\end{equation}
where $\boldsymbol{W}:[0,T]\times\Omega\rightarrow \boldsymbol{W}_{t}(\omega)$ is a  GFBRP, and let $\boldsymbol{W}(\omega)$   represent a geometric rough path lifted by the path $\omega\in\Omega$ and the Hurst index $H\in (\frac13, \frac12]$.  We will use the technique as \cite{MR2738732,MR3226746} to obtain the  random attractors of \eqref{1.1}. Note that the metric dynamical system are ergodic in \cite{MR2738732,MR3226746} for the case of fBm. There are two strategies to construct the metric dynamical system for GFBRP. One is that the orbits of a fBm generate a metric dynamical systems. Due to   the GFBRP is a geometric rough path, namely its second order process $\mathbb{W}(\omega)$ as the limit of canonical lift(see,\cite[page 16]{MR4174393}) a smooth path $W^{\eta}(\omega)$, and  $\theta_{\tau}\mathbb{W}^{\eta}(\omega)$ is also the canonical lift of  the path $\theta_{\tau}W^{\eta}(\omega),\tau\in \mathbb{R},\omega\in\Omega $.  Thus, we use the same symbol $\theta$ for the path component and the second order process in our paper.  So it is enough to define the classical Wiener shift,  and the metric dynamical system is ergodic\cite{MR2836654}.   The other one is that GFBRP can be regard as a new tensor-valued stochastic process and its distribution is the transformation of the law of fBm. In \cite{MR3624539}, the authors  introduced the new Wiener shift for tensor-valued paths and rough cocycle, and then to construct a metric dynamical system. Similar to \cite{MR2836654}, it is not difficult to prove that the metric dynamical system is also ergodic. In our paper, we adopt the first strategy, so we require that the path of sample space $\Omega$ has a canonical  lift.  
\smallskip

 In our paper, for the argument of the compactness of the absorbing set, we impose the condition that a scale of Banach spaces $\mathcal{B}_{\gamma},0<\gamma<(\alpha-\sigma)\wedge(1-\delta)$ compactly embedding into a separable Banach space   $\mathcal{B}$.  For a class of evolution equations driven by a fractional Brownian motion with Hurst index $H>\frac{1}{2}$, Gao et al.\cite{MR3226746} did not impose the smallness condition for the diffusion term $G$, but the covariance $q$  of the fBm is sufficiently small.  Due to the structure of the solution $(y,G(y))$ and the estimate of radius of the absorbing set require temperedness of the stopping times, we  impose  the condition ``$C_{G}\leq\mu$ ". But the case $C_{G}>\mu$ can be dealt as the case  $C_{G}\leq\mu$(see Remark \ref{remark 3.1*} ). For simplicity, we only consider the case $C_{G}\leq \mu$. So we construct a sequence of stopping times in the spirit of  \cite{MR3226746}. In addition, Duc\cite{MR4385780}  considered rough differential equations driven by a Gaussian rough path, and the  similar condition for the diffusion term $G$ was imposed. For the drift term $F$, the smallness condition is still necessary.
 \smallskip

 {\it Hence, under the smallness conditions for $F$ and the noise, we obtain the existence of random attractor of the system \eqref{1.1} with Hurst index $H\in(\frac13, \frac12]$. In addition, the existence of global random attractor of SPDEs with generally multiplicative noise is still open by now, our results can be regarded as  first attempt to solve this problem using rough path theory.}


\smallskip
Our second  aim is to establish the upper semi-continuity for $\mathcal{A}_{\eta}$, where  $\mathcal{A}_{\eta}$ are the random attractors of the above evolution equations driven by $\boldsymbol{W}^{\eta}(\omega)$  which is the smooth and stationary approximation of the GFBRP $\boldsymbol{W}(\omega)$ in Section \ref{Wong-Zakai}. For fixed $\eta\in(0,1]$, due to the smoothness of $W^{\eta}(\omega)$, the rough integral $$\int_{0}^{t}\!S(t-r)\!G(y^{\eta})d\boldsymbol{W}^{\eta}_{r}\!=\!\lim_{|\mathcal{P}(0,t)|\rightarrow 0}\sum_{[u_i,v_i]\in \mathcal{P}(0,t)}\!S(t-u_i)\!\left(G(y^{\eta}_{u_i})W^{\eta}_{u_i,v_i}\!+\!DG(y^{\eta}_{u_i})G(y^{\eta}_{u_i})\mathbb{W}^{\eta}_{u_{i},v_{i}}\right)$$
is just a Young integral
$$\int_{0}^{t}\!S(t-r)\!G(y^{\eta})d\boldsymbol{W}^{\eta}_{r}\!=\!\lim_{|\mathcal{P}(0,t)|\rightarrow 0}\!\!\sum_{[u_i,v_i]\in \mathcal{P}(0,t)}\!S(t-u_i)G(y^{\eta}_{u_i})W^{\eta}_{u_i,v_i}.$$
Then the approximate equation  can be regarded as a random partial differential equation. Since the approximate noise is not truly rough\cite[page 109]{MR4174393}, then the Gubinelli derivative is not unique. In order to obtain the  existence and uniqueness of the solutions of approximate equations, similar to controlled rough paths, we choose a proper Gubinelli derivative $G(y^{\eta})$, then we can get the existence of the approximate equations in rough paths sense and the uniqueness for $y^\eta$  can be easily obtained by Lipschitz conditions for coefficients. It means that there is a  unique solution for random partial differential equation. That is to say that the random dynamical system for an approximate equation in the frame work of rough path coincides with the random dynamical system generated by a random partial differential equation. Moreover, we concern the limit behavior of $\eta\rightarrow 0$ for approximate system, then  we  regard $\boldsymbol{W}^{\eta}(\omega)$ as a smooth rough path. Since all moments of the approximated noise $\boldsymbol{W}^{\eta}$ are uniform bounded, it guarantees that the sufficiently small and uniform covariance  $q$ for $W$ and $W^{\eta},\eta\in(0,1]$ can be chosen. Thus, the existence of the $\mathcal{A}_{\eta}$ can be obtained as $\mathcal{A}$. Finally, we establish  the convergence of the stopping times to get the convergence of absorbing sets, and then get the upper semi-continuity of $\mathcal{A}_{\eta}$ as\cite{MR3270946,MR4266114}.
\smallskip

 Compared with \cite{MR3270946,MR4266114}, due to uniform convegence of the approximated noise, the  random dynamical system $\varphi^{\eta}$ with respect to $y^{\eta}$ converge to $\varphi$ with respect to $y$  is uniform for any $t\in [0,\infty)$. For this  stationary  and smooth approximation,  our method can be used to rough noise not just Brownian motion.       

\smallskip

The paper is organized as follows: In Section 2, we give some basic knowledge on rough path theory and random dynamical system. In Section 3, we consider the dynamics for non-autonomous dynamical system $\Phi$  which is generated  by the solution   of \eqref{1.1} on stopping times $\{T_{i}(\boldsymbol{W}(\omega))\}_{i\in\mathbb{Z}}$  and $\varphi$  is generated  by  the solution of \eqref{1.1} on $\mathbb{R}$. Section 4 is devoted to the existence of random attractor under the random settings. In Section 5, we establish the upper semi-continuity for approximated attractors $\mathcal{A}_{\eta}$. In an Appendix,  we give some proofs of lemmas and theorems which are used in our paper.
\section{Preliminaries}
In this section, we will recall some basic concepts in rough path theory and random dynamical system.

\textbf{Notation:} Let $I$ be a compact interval on $\mathbb{R}$. Denote $C(I,V)$ and $C^{\alpha}(I,V)$ by the V-valued continuous function space and the space of V-valued $\alpha$-H\"{o}lder continuous function space respectively, where V is a Banach space and  $\alpha\in(0,1)$. We use the notation $y_{t}$ instead of $y(t)$, and $y_{s,t}=y_t-y_s$. Constant $C$ different from line to line and the notion $C(x)$ to indicate the constant $C$ depends on $x$. The simplex $\triangle^{n}_{[0,T]}:= \{0\leq t_{1}\leq t_{2}\leq \cdots\leq t_{n-1}\leq t_{n}\leq T\}$. $\delta\xi_{s,u,t}=\xi_{s,t}-\xi_{s,u}-\xi_{u,t}$ for $(s,u,t)\in\triangle^{3}$. The space $C_{2}^{\alpha}([0,T],\mathcal{B}_{\gamma})$ of all families such that its element $g:\triangle^{2}_{[0,T]}\rightarrow \mathcal{B}_{\gamma}$ satisfying  $\interleave g\interleave_{\alpha,\gamma}=\frac{\|g_{s,t}\|_{\gamma}}{|t-s|^{\alpha}}<\infty$ and the space $C_{3}^{\alpha_{1},\alpha_{2}}([0,T],\mathcal{B}_{\gamma})$ of all families such that its element $h:\triangle_{[0,T]}^{3}\rightarrow \mathcal{B}_{\gamma}$ satisfying  $\interleave h\interleave_{\alpha_{1},\alpha_{2},\gamma}=\frac{\|h_{s,u,t}\|_{\gamma}}{|t-u|^{\alpha_{2}}|u-s|^{\alpha_{1}}}<\infty$.

\subsection{Rough path and rough integral}

\begin{definition}[$\alpha$-H\"{o}lder rough paths]
Let $\alpha\in(\frac{1}{3},\frac{1}{2}]$, we call the couple $\boldsymbol{W}=(W,\mathbb{W})$ a $d$-dimensional $\alpha$-H\"{o}lder rough path on $I$, if $W\in C^{\alpha}(I,\mathbb{R}^{d})$ and $\mathbb{W}\in C^{2\alpha}(\Delta^2_I,\mathbb{R}^{d}\otimes\mathbb{R}^{d})$,  and satisfy the Chen's identity
$$\mathbb{W}_{s,t}-\mathbb{W}_{s,u}-\mathbb{W}_{u,t}=W_{s,u}\otimes W_{u,t},\quad s\leq u\leq t\in I.$$
Let $\mathcal{C}^{\alpha}(I,\mathbb{R}^{d})$ denote  the set of all $\alpha$-H\"{o}lder rough paths.  In addition, for any $W\in C^1(I,\mathbb{R}^d)$, there is a canonical lift $S(W):=(W,\mathbb{W})$ in $\mathcal{C}^{\alpha}(I,\mathbb{R}^{d})$  defined as
	\begin{equation*}
		\mathbb{W}_{s,t}^{k,l}=\int_{s}^{t}\int_{s}^{r}dW_{r^\prime}^{k}dW_r^{l},\quad s<t\in I~ \text{and} ~ k,l\in\{1,\cdots,d\}.
	\end{equation*}
We denote the geometric rough space by $\mathcal{C}_{g}^{\alpha}(I,\mathbb{R}^{d})$, i.e. the closure of  the  canonical lift $S(W),W\in C^{1}(I,\mathbb{R}^d)$.
\end{definition}
\begin{remark}
The first component $W$ of $\alpha$-H\"{o}lder rough path $\boldsymbol{W}$ is a path on interval $I$ and it has an $\alpha$-H\"{o}lder regularity, and the second component is called second order process or L\'{e}vy area, and it satisfies an analytic condition, namely $\sup_{(s,t)\in \triangle^2_I}\frac{|\mathbb{W}_{s,t}|}{|t-s|^{2\alpha}}<\infty$. More details on rough path theory we refer to \cite{MR4174393}.
\end{remark}

For two different $\alpha$-H\"{o}lder rough paths we can compare their distance by defining a metric on rough path space $\mathcal{C}^{\alpha}(I,\mathbb{R}^{d})$ as follows.
\begin{definition}
Let $I\subset\mathbb{R}$ be a compact set and $\boldsymbol{W}, \tilde{\boldsymbol{W}}\in\mathcal{C}^{\alpha}(I,\mathbb{R}^{d}),\alpha\in(\frac{1}{3},\frac{1}{2}]$, we give an (inhomogeneous)$\alpha$-H\"{o}lder rough path metric
$$d_{\alpha,I}(\boldsymbol{W},\tilde{\boldsymbol{W}})=\sup_{(s,t)\in \triangle^2_I}\frac{|W_{s,t}-\tilde{W}_{s,t}|}{|t-s|^{\alpha}}+\sup_{(s,t)\in \triangle^2_I}\frac{|\mathbb{W}_{s,t}-\tilde{\mathbb{W}}_{s,t}|}{|t-s|^{2\alpha}}.$$
\end{definition}
\begin{remark}
The rough path metric induces an inhomogeneous norm of rough paths,
 $$\|\boldsymbol{W}\|_{\alpha,I}=\interleave W\interleave_{\alpha,I}+\interleave \mathbb{W}\interleave_{2\alpha,\Delta_I^{2}}.$$
In order to  keep with \cite{kuehn2021center,hesse2021global}, we sometimes use the notation $\rho_{\alpha,I}(\boldsymbol{X})$ instead of $\|\boldsymbol{X}\|_{\alpha,I}$  and more results on this topic can be found in \cite{MR4174393}.
\end{remark}
In order to introduce the  controlled rough path for a class of evolution equations, we need the following notion of interpolation spaces.
\begin{definition}[A monotone family of interpolation spaces]\label{def 2.3}
We call a family of separable Banach spaces $(\mathcal{B}_{\gamma},\|\cdot\|_{\gamma})_{\gamma\in\mathbb{R}}$ are monotone interpolation spaces if for $\gamma_{1}\leq \gamma_{2}$, the space $\mathcal{B}_{\gamma_{2}}\subset\mathcal{B}_{\gamma_{1}}$ with dense and continuous embedding and have the following interpolation inequality holds for $\gamma_{1}\leq \gamma_{2}\leq \gamma_{3}$ and $x\in \mathcal{B}_{\gamma_{3}}$:
$$\|x\|^{\gamma_{3}-\gamma_{1}}_{\gamma_{2}}\leq C\|x\|_{\gamma_{1}}^{\gamma_{3}-\gamma_{2}}\|x\|_{\gamma_{3}}^{\gamma_{2}-\gamma_{1}}.$$
\end{definition}
\begin{remark}
The above monotone family of interpolation spaces have the following nice properties:
If semigroup $S:[0,T]\rightarrow \mathcal{L}(\mathcal{B}_{\gamma},\mathcal{B}_{\gamma+1})$ is such that for each $x\in\mathcal{B}_{\gamma+1}$ and $t\in [0,T]$ we have $|(S(t)-Id)x|_{\gamma}\leq Ct\|x\|_{\gamma+1}$ and $|S(t)x|_{\gamma+1}\leq Ct^{-1}\|x\|_{\gamma}$, then for all $\sigma\in [0, 1]$,  we  get $S(t)\in \mathcal{L}(\mathcal{B}_{\gamma+\sigma},\mathcal{B}_{\gamma+\sigma})$ and have the following estimates
\begin{align}
\|(S(t)-Id)x\|_{\gamma}&\leq Ct^{\sigma}\|x\|_{\gamma+\sigma},\label{2.1}\\
\|S(t)x\|_{\gamma+\sigma}&\leq Ct^{-\sigma}\|x\|_{\gamma}.\label{2.2}
\end{align}
An example for these interpolation spaces is fractional Sobolev spaces(Bessel potential spaces)$H^{s},s\in\mathbb{R}$.  The further results on interpolation spaces we refer to \cite{MR3012216}. In the following deliberation, we stress that $\alpha\in (\frac{1}{3},\frac{1}{2})$ represents the time regularity and $\gamma$ is the spatial regularity in $\mathcal{B}_{\gamma}$.
\end{remark}
In  the following section, we will consider the mild solution of \eqref{1.1}
\begin{align}
y_{t}=S(t)y_{0}+\int_{0}^{t}S(t-r)F(y_{r})dr+\int_{0}^{t}S(t-r)G(y_{r})d\boldsymbol{W}_{r}.
\end{align}
It is vital to understand the last integral in the mild solution. To this end, we need to give a proper definition of controlled rough paths. Based on the monotone family of interpolation spaces, we introduce the concept of controlled rough paths which are similar to  finite-dimensional case,  and it reflects the interaction effect of regularity in time and space.
\begin{definition}[Controlled rough paths]\label{def 2.4}
We call the pair $(y,y^{\prime})$ is a controlled rough path if
\begin{itemize}
                                                                      \item $(y,y^{\prime})\in C([0,T],\mathcal{B}_{\gamma})\times(C([0,T],\mathcal{B}_{\gamma-\alpha})\bigcap C^{\alpha}([0,T],\mathcal{B}_{\gamma-2\alpha}))$, the component $y^{\prime}$ is also called Gubinelli derivative as finite-dimensional case.
                                                                      \item $(y,y^{\prime})$ satisfies the following identity
\begin{align}\label{2.4}
R^{y}_{s,t}=y_{s,t}-y^{\prime}_{s}W_{s,t},\end{align}
where $R^{y}_{s,t}$ is called the remainder and it belongs to {\small$C^{\alpha}(\triangle^2_{[0,T]},\mathcal{B}_{\gamma-\alpha})\cap C^{2\alpha}(\triangle^2_{[0,T]},\mathcal{B}_{\gamma-2\alpha})$}.

\end{itemize}
      We denote  the set of all controlled rough paths by $\mathcal{D}_{W,\gamma,[0,T]}^{2\alpha}:=\mathcal{D}_{W}^{2\alpha}([0,T],\mathcal{B}_{\gamma})$  and endowed with the norm as follows
      $$\|y,y^{\prime}\|_{W,2\alpha,\gamma}=\|y\|_{\infty,\gamma}+\|y^{\prime}\|_{\infty,\gamma-\alpha}+\interleave y^{\prime}\interleave_{\alpha,\gamma-2\alpha}+\interleave R^{y}\interleave_{\alpha,\gamma-\alpha}+\interleave R^{y}\interleave_{2\alpha,\gamma-2\alpha}.$$
\end{definition}
\begin{remark}
$D_{W,\gamma}^{2\alpha}$ is a Banach space under the above norm\cite{MR4299812}. Furthermore, the $\alpha$-H\"{o}lder semi-norm  for $y$ does not emerge here. Indeed, using \eqref{2.4} for $\theta\in \{\alpha,2\alpha\}$ we have that
\begin{align}
\interleave y\interleave_{\alpha,\gamma-\theta}\leq \|y^{\prime}\|_{\infty,\gamma-\theta}\interleave W\interleave_{\alpha}+\interleave R^{y} \interleave_{\alpha,\gamma-\theta}.
\end{align}
\end{remark}
The following  lemma gives the definition of rough integral and its' estimate.
\begin{lemma}[Lemma 4.5,\cite{MR4299812}]\label{Lemma 2.1}
Let $\boldsymbol{W}$ be an $\alpha$-H\"{o}lder rough path and $(y,y^{\prime})\in D_{W,\gamma}^{2\alpha}$, then the rough integral
$$\int_{0}^{t}S(t-r)y_{r}d\boldsymbol{W}_{r}:=\lim_{|\mathcal{P}(0,t)|\rightarrow0}\sum_{[u,v]\in\mathcal{P}(0,t)}S(t-u)\left[y_{u}W_{u,v}+y^{\prime}_{u}\mathbb{W}_{u,v}\right]$$
exists in $\mathcal{B}_{\gamma-2\alpha}$, where $\mathcal{P}(0,t)$ is a partition of the interval $[0,t]$ and the limit is independent of the choice of the specific partition  $\mathcal{P}(0,t)$. In addition, there is an estimate  as follows
\begin{align}\label{2.6}
\bigg|\int_{s}^{t}S(t-r)y_{r}d\boldsymbol{W}_{r}-S(t-s)y_{s}W_{s,t}-&S(t-s)y^{\prime}_{s}\mathbb{W}_{s,t}\bigg|_{\gamma-2\alpha+\beta}\nonumber\\
\leq& C\rho_{\alpha,[s,t]}(\boldsymbol{W})\|y,y^{\prime}\|_{X,2\alpha,\gamma,[s,t]}(t-s)^{3\alpha-\beta},
\end{align}
for all $0\leq \beta<3\alpha$ and $0\leq s< t \leq T$.
\end{lemma}
\subsection{Non-autonomous and random dynamical system}
The rough path theory in previous subsection provides a framework of pathwise.  In this subsection, we give some basic concepts on non-autonomous and random dynamical systems.

Let $\mathbb{T}=\mathbb{R}$ or $\mathbb{Z}$, a flow of non-autonomous perturbations $(\Omega,\theta)$ as follows
$$\theta: \mathbb{T}\times \Omega\rightarrow \Omega,$$
 $\theta_{t_1}\circ\theta_{t_2}=\theta_{t_{1}+t_{2}}$ for $t_{1},t_{2}\in \mathbb{T}$ and $\theta_{0}=Id_{\Omega}$. Furthermore, we call
a mapping $\varphi: \mathbb{T}^{+}\times \Omega\times V\rightarrow V$ is a non-autonomous dynamical system, if
$$\varphi(t+\tau,\omega,u_{0}) = \varphi(t,\theta_{\tau}\omega,\cdot)\circ\varphi(\tau,\omega,u_{0}),$$
for $t,\tau\in \mathbb{T}^{+},\omega\in\Omega, u_{0}\in V$, where $V$ is a separable Banach space. For the random settings, we have the following definition.
\begin{definition}[Metric dynamical system]
Let $(\Omega,\mathcal{F},\mathbb{P})$ be a probability space, and the flow $\theta: \mathbb{T}\times\Omega\rightarrow\Omega$ satisfies the following conditions:\begin{enumerate}
            \item $\theta_{0}=Id_{\Omega}$;
            \item $\theta_{t+s}=\theta_{t}\circ\theta_{s}$;
            \item the mapping $(t,\omega)\mapsto \theta_{t}\omega$ is $(\mathcal{B}(\mathbb{T})\times\mathcal{F},\mathcal{F})$-measurable, where $\mathcal{B}(\mathbb{T})$ and $\mathcal{F}$ are Borel $\sigma$-algebra generated by $\mathbb{T}$ and $\Omega$, respectively;
            \item $\theta_{t}:\Omega\rightarrow \Omega$ are $\mathbb{P}$-preserving transformations, namely $\theta_{t}\mathbb{P}=\mathbb{P}$.

          \end{enumerate}
Then the quadruple $(\Omega,\mathcal{F},\mathbb{P},\{\theta_{t}\}_{t\in \mathbb{T}})$ is called a metric dynamical system.
\end{definition}

In following consideration, the Wiener shift $\theta$ acts on a rough path, we need to extend the action of the $\theta$ from the path $W$ to rough path $\boldsymbol{W}$, namely, for any $\alpha$-H\"{o}lder rough path $\boldsymbol{W}=(W, \mathbb{W})$ and $\tau,s,t\in\mathbb{R}$, the Wiener shift can be changed to $\theta_{\tau}\boldsymbol{W}=(\theta_{\tau}W,\theta_{\tau}\mathbb{W})$ by
\begin{equation*}
	\begin{aligned}
\theta_{\tau}W_{t}&=W_{t+\tau}-W_{\tau},\\
\theta_{\tau}\mathbb{W}_{s,t}&=\mathbb{W}_{s+\tau,t+\tau}.
    \end{aligned}
\end{equation*}
In particular, let $W_{t}(\omega)=\omega_{t}$. Since fBm $W_t$ has a canonical lift almost surely,   then
\begin{equation*}
	\begin{aligned}
\theta_{\tau}W_t(\omega)&=W_{t}(\theta_\tau\omega)=W_{t+\tau}(\omega)-W_{\tau}(\omega),\\
\theta_{\tau}\mathbb{W}_{s,t}(\omega)&=\mathbb{W}_{s,t}(\theta_\tau\omega)=\mathbb{W}_{s+\tau,t+\tau}(\omega),
    \end{aligned}
\end{equation*}	
 it means that $\theta_{\tau}\boldsymbol{W}(\omega)=\boldsymbol{W}(\theta_\tau\omega)$. According to the definition of the rough path, the new Wiener shift $\theta$ act on a rough path is also  a rough path, we give it in a lemma.
\begin{lemma}[Lemma 32,\cite{MR4097587}]
Let $T_{1},T_{2},\tau\in\mathbb{R}$ and $\boldsymbol{W}=(W,  \mathbb{W})$ be an $\alpha$-H\"{o}lder rough path on $[T_{1},T_{2}]$ for $\alpha\in(\frac{1}{3},\frac{1}{2})$, then $\theta_{\tau}\boldsymbol{W}$ is an  $\alpha$-H\"{o}lder rough path on $[T_{1}-\tau,T_{2}-\tau]$.
\end{lemma}
 In our paper, the GFBRP is a geometric rough path, it has a stationary and smooth approximation. So compared with Bailleul et al.\cite{MR3624539},  our sample space is an $\alpha$-H\"{o}lder path space rather than an $\alpha$-H\"{o}lder rough path space.
\begin{definition}
Let $(\Omega,\mathcal{F},\mathbb{P},\{\theta_{t}\}_{t\in\mathbb{T}})$ be a metric dynamical system and $V$ be a separable Banach space. We call the mapping $\varphi$  continuous random dynamical system on $V$ if $\varphi$ has the following properties:
\begin{enumerate}
  \item $\varphi(0,\omega,\cdot)=Id_{V}$ for all $\omega\in\Omega;$
  \item $\varphi(t+\tau,\omega,x)=\varphi(t,\theta_{\tau}\omega,\varphi(\tau,\omega,x))$ for all $t,\tau\in \mathbb{T}^{+}$ and $x\in V;$
  \item $\varphi(t,\omega,\cdot): V\rightarrow V$  is continuous for all $t\in \mathbb{T}^{+}$ and $\omega\in\Omega;$
  \item the mapping $\varphi:\mathbb{T^{+}}\times\Omega\times V\rightarrow V$ is $(\mathcal{B}(\mathbb{T}^{+})\otimes\mathcal{F}\otimes\mathcal{B}(V),\mathcal{B}(V))$-measurable.
\end{enumerate}
\end{definition}

In order to describe the attractors, we  introduce some concepts.
\begin{definition}
Let $\mathcal{D}$ be consisted of a  family of nonempty sets $(D(\omega))_{\omega\in\Omega}$, it has a property $\mathcal{P}$  and if $D=(D(\omega))_{\omega\in\Omega}\in \mathcal{D}$ and  $\emptyset\neq D^{\prime}=(D^{\prime}(\omega))_{\omega\in\Omega}\subset D$, then $D^{\prime}\in \mathcal{D}$.
The property $\mathcal{P}$: For a given a constant $\nu>0$, we call  $\mathcal{D}$  backward exponential growth if there exist a mapping $\omega\rightarrow r(\omega)\in \mathbb{R}^{+}$ such that $\emptyset\neq D(\omega)\in \mathcal{D}_{V}^{\nu}$ is contained in a ball $B_{V}(0,r(\omega))$ in space $V$ with center $0$ and radius $r(\omega)$, and
$$\lim_{\mathbb{T}\ni t\rightarrow -\infty}\frac{\log^{+}r(\theta_{t}\omega)}{|t|}<\nu,\quad \forall \omega\in\Omega.$$
 We use  notations $\mathcal{D}_{\mathbb{R},V}^{\nu},\mathcal{D}_{\mathbb{Z},V}^{\nu}$ to stress  different time sets $\mathbb{R}$ and $\mathbb{Z}$, respectively.

In random settings, $\emptyset\neq D=D(\omega)\subset V$ equipped with the measurability, namely $dist(v,D(\omega))$ is a random variable for $\omega\in\Omega, v\in V$, then $D$ is called a random set.  In addition, for  a  random set $D$, if
 there exists a random variable $r(\omega)\in\mathbb{R}^{+}$ such that $D(\omega)\in B_{V}(0,r(\omega))$ and
$$\lim_{\mathbb{T}\ni t\rightarrow \pm\infty}\frac{\log^{+}r(\theta_{t}\omega)}{|t|}=0,\quad \forall \omega\in \Omega,$$
then we call the random set $D$  tempered or subexponential growth. We denote
the set  of random tempered sets by $\hat{\mathcal{D}}$.
\end{definition}
\begin{definition}
A family  $B=(B(\omega))_{\omega}\subset V$  is called pullback absorbing set for $\mathcal{D}$ if
$$\varphi(t,\theta_{-t}\omega,D(\theta_{-t}\omega))\subset B(\omega),\quad t\geq T(D,\omega)$$
for any $D\in \mathcal{D}$ and $\omega\in\Omega$, where $T(D,\omega)$ is called absorption time.
\end{definition}
\begin{definition}
Let $\mathcal{A}=\{A(\omega)\}$ be a random attractor for random dynamical system $\varphi$ if
\begin{itemize}
  \item $\mathcal{A}(\omega)$ is a compact set for all $\omega\in\Omega.$
  \item for any $\omega\in\Omega$ and $t\in \mathbb{T}^{+}$,
  $$\varphi(t,\omega,\mathcal{A}(\omega))=\mathcal{A}(\theta_{t}\omega).$$
  \item for any $D\in\hat{\mathcal{D}},\omega\in\Omega$,
  $$\lim_{\mathbb{T}^{+}\ni t\rightarrow \infty}dist(\varphi(t,\theta_{-t}\omega,D(\theta_{-t}\omega)),\mathcal{A}(\omega))\rightarrow 0.$$
\end{itemize}
\end{definition}
We have the following lemma about the existence of pullback/random attractors(see \cite{MR1427258,MR1870217})
\begin{lemma}\label{lemma 2.3}
Let $\varphi$ be a continuous nonautonomous dynamical system. If $\varphi$ has a compact pullback absorbing set $B\in\mathcal{D}$. Then nonautonomous dynamical system $\varphi$ has a pullback attractor for $\mathcal{D}$. Furthermore, if $\varphi$ is a random dynamical system and has a compact pullback absorbing set $B$ in $\hat{\mathcal{D}}$. Then $\varphi$ has a unique random attractor $\mathcal{A}=\{\mathcal{A}(\omega)\}$ for $\hat{\mathcal{D}}$ as follows:
$$\mathcal{A}(\omega)=\bigcap_{s\geq T(B,\omega)\in \mathbb{T}^{+}}\overline{\bigcup_{t\geq s}\varphi(t,\theta_{-t}\omega,B(\theta_{-t}\omega))},\quad \forall \omega\in\Omega.$$
\end{lemma}
\section{Evolution equations driven by an \texorpdfstring{$\alpha$}{lg}-H\"{o}lder  geometric rough path}
In this section, we consider the following rough evolution equations which are defined on the interval $[0,T]$:
\begin{equation}\label{3.1}
dy_{t}=\left(Ay_{t}+F(y_{t})\right)dt+G(y_{t})d\boldsymbol{W}_{t},\quad y(0)=y_{0}\in \mathcal{B},
\end{equation}
where $\boldsymbol{W}\in \mathcal{C}_{g}^{\alpha}([0,T],\mathbb{R}^d)$ 
and $\mathcal{B}$ is a separable Banach space.  In what follows we give some assumptions on  equation \eqref{3.1}:
\begin{description}
  \item[(A)] $-A$ generates a $C_0$-analytical semigroup $S(t)$ on $\mathcal{B}$, and  the Banach spaces $\mathcal{B}_{\gamma}\subset\subset\mathcal{B},\gamma>0$.
  \item[(F)] The nonlinear drift term $F:\mathcal{B}_{\gamma}\rightarrow \mathcal{B}_{\gamma-\delta}$ for $\delta\in [0,1)$ and it is global Lipschitz. Furthermore,  $\|F(0)\|_{\gamma-\delta}$ and Lipschitz constant $l_{F}$ are less than a constant $\mu$ which is enough small and will be determined later.
  \item[(G)] Let $\theta\in\{0,\alpha,2\alpha\}$ and $0\leq \sigma<\alpha$. The nonlinear diffusion coefficient $G:\mathcal{B}_{\gamma-\theta}\rightarrow \mathcal{B}_{\gamma-\theta-\sigma}$ is three times Frèchet differentiable with bounded derivatives, i.e. $\left\|D^{k} G\right\|_{\mathcal{L}\left(\mathcal{B}_{\gamma-\theta}^{\otimes k}, \mathcal{B}_{\gamma-\theta-\sigma}\right)}<\infty$ for $k \in\{1,2,3\}$ and the derivative of $DG(\cdot)G(\cdot):\mathcal{B}_{\gamma-\alpha}\rightarrow \mathcal{B}_{\gamma-2\alpha-\sigma}$ is also bounded. In addition, we require  $C_{G}\leq \mu$, where  $C_{G}=\max\left\{\left\|D^{k} G\right\|_{\mathcal{L}\left(\mathcal{B}_{\gamma-\theta}^{\otimes k}, \mathcal{B}_{\gamma-\theta-\sigma}\right)},k=1,2,3,\left\|D(DG(\cdot)G(\cdot))\right\|_{\mathcal{L}\left(\mathcal{B}_{\gamma-\alpha}, \mathcal{B}_{\gamma-2\alpha-\sigma}\right)}\right\}$.
\end{description}
\begin{remark}\label{remark 3.1*}
To obtain the compactness of the absorbing set, the condition (\textbf{A}) is necessary. For example, we can consider the operator $A=\triangle-I$, $-A$ is a strictly positive and symmetric operator and it's  inverse operator is compact on $L^{2}$, and the spaces $\mathcal{B}_{\gamma}:=D\left((-A)^{\frac{\gamma}{2}}\right)$ with norm $\|\cdot\|_{\gamma}:=\|(-A)^{\frac{\gamma}{2}}\cdot\|$, it coincides with the fractional Sobolev space $H^{\gamma},\gamma\in\mathbb{R}$.   Then $\mathcal{B}_{\gamma}$ compactly embedded into $\mathcal{B}$ for $\gamma>0$.  Thus, let $-A$ be a strictly and symmetric  operator  with a compact inverse which is a generator of an analytic exponential decreasing semigroup $S(t),t\geq 0$ in our paper. Furthermore, the semigroup $S(t)$ has the following properties
                                \begin{equation}\label{2.1*}
                                \|S(t)\|_{\mathcal{L}(\mathcal{B}_{\gamma},\mathcal{B}_{\gamma+\sigma})}\leq \frac{C}{t^
                                \sigma}e^{-\lambda t} \quad \text{for} \quad \sigma\geq 0,\gamma\in\mathbb{R},\quad t\in[0,T],
                                \end{equation}
                                \begin{equation}\label{2.2*}
                                \|S(t)-Id\|_{\mathcal{L}(\mathcal{B}_{\gamma+\sigma},\mathcal{B}_{\gamma})}\leq Ct^\sigma, \quad \text{for} \quad 1\geq\sigma\geq 0,\gamma\in \mathbb{R}, \quad t\in[0,T],
                                \end{equation}
 where the constant $C>0$ depends on $S$ and $\lambda>0$ is the smallest eigenvalue of $-A$. The condition $C_{G}\leq \mu$ seems to be strong. 
However,  for the case $C_{G}>\mu$, we  transform the diffusion term part $G(y)d\boldsymbol{W}$ to  $\frac{\mu}{C_{G}} G(y)d\boldsymbol{W}^{\mu,G}$, $\boldsymbol{W}^{\mu,G}=\left(\frac{C_{G}}{\mu}W,\frac{C^2_{G}}{\mu^2}\mathbb{W}\right)$,  Lemma \ref{Lemma 2.1} shows that $\int_{s}^{t}S(t-r)G(y_r)d\boldsymbol{W}_r=\int_{s}^{t}S(t-r)\frac{\mu}{C_{G}} G(y)d\boldsymbol{W}^{\mu,G}_r, s,t\in[0,T]$. 
Finally, we can choose $G(y)=g(x)(-\triangle)^{\sigma}y$ as \cite{MR4299812,hesse2021global}, where $g$ is a smooth function.
\end{remark}
Based on the above assumptions, we have the following global existence of the rough evolution equations.
\begin{theorem}[Theorem 3.9,\cite{hesse2021global}]
Let $T>0$, under assumptions \textbf{(A)},\textbf{(F)},\textbf{(G)}, $\boldsymbol{W}=(W,\mathbb{W})\in \mathcal{C}^\alpha([0,T];\mathbb{R}^d),\alpha\in(\frac{1}{3},\frac{1}{2}]$  and $y_{0}\in \mathcal{B}_{\gamma},\gamma\geq 0$. Then there exists a unique global  solution $(y,G(y))\in \mathcal{D}_{W,\gamma}^{2\alpha}([0,T])$ of \eqref{3.1}.
\end{theorem}
\begin{remark}
\cite{hesse2021global} required $\delta\in [2\alpha,1)$, we don't need this condition. Indeed, for $\delta\in[0,1)$, the estimate in Lemma 3.3\cite{hesse2021global} as follows
\begin{equation*}
 \left\|\int_{0}^{\cdot}S(\cdot-s)F(y_s)ds,0\right\|_{W,2\alpha,\gamma}\lesssim T^{(1-\delta)\wedge(1-2\alpha)}(1+\|y\|_{\infty,\gamma}),
 \end{equation*}
  the Banach fixed point theorem and concatenation argument can be used to get the global solution as\cite{hesse2021global}. So we omit its' proof.
\end{remark}
We consider the mild solution for \eqref{3.1} as follows
\begin{align}\label{3.2}
y_{t}=S(t)y_{0}+\int_{0}^{t}S(t-r)F(y_{r})dr+\int_{0}^{t}S(t-r)G(y_{r})d\boldsymbol{W}_{r},\quad t\in[0,T].
\end{align}
\subsection{Non-autonomous dynamical system associated to \texorpdfstring{\eqref{3.2}}{mg}}
The system \eqref{3.2} can generate a continuous random dynamical system $\varphi$, and it can be found in \cite{kuehn2021center,hesse2021global}. Thus, it also generates non-autonomous continuous dynamical system for a given sample space $\Omega$ which satisfies every $\omega\in\Omega$ has canonical lift, namely $\omega$ can generate an  $\frac{1}{2}\geq H>\alpha^\prime$-H\"{o}lder  geometric rough path $\boldsymbol{W}(\omega)=(W(\omega),\mathbb{W}(\omega))$.
\begin{lemma}
Let $y_{0}\in\mathcal{B}_{\gamma},\gamma\geq 0$ and $\boldsymbol{W}(\omega)=(W(\omega),\mathbb{W}(\omega))$ be a geometric $\alpha^\prime$-rough path for every $\omega\in\Omega$. Then the system \eqref{3.1} can generate a non-autonomous dynamical system $\varphi$ on $\mathbb{R}^{+}$ with state space $\mathcal{B}_{\gamma}$:
\begin{align}
&\varphi:\mathbb{R}^{+}\times \Omega\times\mathcal{B}_{\gamma}\rightarrow \mathcal{B}_{\gamma},\nonumber\\
&\varphi(t,\omega,y_{0})=S(t)y_{0}+\int_{0}^{t}S(t-r)F(y_{r})dr+\int_{0}^{t}S(t-r)G(y_{r})d\boldsymbol{W}_r(\omega).\label{3.3}
\end{align}
\end{lemma}
In the following section, we shall derive another discrete non-autonomous dynamical system on $\mathbb{Z}^{+}$ and construct a priori estimate of the solution. To this end, we introduce a sequence of stopping times to make the norm of rough path $\boldsymbol{W}(\omega)$ sufficiently small on each stopping time interval.

Firstly,  for some $\mu\in (0,1),\alpha\in(0,\alpha^\prime)$ and every $\omega\in\Omega$, we define  the stopping times as follows
\begin{align}
T(\boldsymbol{W}(\omega))&=\inf\{\tau>0:\|\boldsymbol{W}(\omega)\|_{\alpha,[0,\tau]}+\mu\tau^{1-\alpha}\geq \mu\},\label{3.4}\\
\hat{T}(\boldsymbol{W}(\omega))&=\sup\{\tau<0:\|\boldsymbol{W}(\omega)\|_{\alpha,[\tau,0]}+\mu|\tau|^{1-\alpha}\geq \mu\}.\label{3.5}
\end{align}
\begin{lemma}\label{lemma 3.2}
For  all $\omega\in\Omega$, stopping times $T(\boldsymbol{W}(\omega))$ and $\hat{T}(\boldsymbol{W}(\omega))$ have the following properties:
\begin{description}
  \item[(i)] $T(\boldsymbol{W}(\omega)),\left|\hat{T}(\boldsymbol{W}(\omega))\right|\in (0,1]$;
  \item[(ii)] $T(\boldsymbol{W}(\omega))=-\hat{T}(\theta_{T(\boldsymbol{W}(\omega))}\boldsymbol{W}(\omega)),\quad \hat{T}(\boldsymbol{W}(\omega))=-T(\theta_{\hat{T}(\boldsymbol{W}(\omega))}\boldsymbol{W}(\omega)). $
\end{description}
\end{lemma}
\begin{proof}
It is clear that $T(\boldsymbol{W}(\omega)),\left|\hat{T}(\boldsymbol{W}(\omega))\right|\leq 1$, it can be easily obtained from the definition \eqref{3.4} and \eqref{3.5}. Since $\boldsymbol{W}(\omega)$ is an
$\alpha^{\prime}$-H\"{o}lder geometric rough path for $\omega\in\Omega$ and $\alpha<\alpha^{\prime}$, we have
\begin{equation*}
\|\boldsymbol{W}(\omega)\|_{\alpha,[0,\tau]}\leq \|\boldsymbol{W}(\omega)\|_{\alpha^\prime,[0,\tau]}(\tau^{\alpha^\prime-\alpha}+\tau^{2(\alpha^\prime-\alpha)})\rightarrow 0, \quad \text{as}\quad \tau\rightarrow 0
\end{equation*}
and
\begin{equation*}
\lim_{\tau\rightarrow\infty}\|\boldsymbol{W}(\omega)\|_{\alpha,[0,\tau]}+\mu\tau^{1-\alpha}=+\infty.
\end{equation*}
Hence, once proving the mapping $\tau\mapsto \|\boldsymbol{W}(\omega)\|_{\alpha,[0,\tau]}+\mu\tau^{1-\alpha}$ is continuous and strictly increasing, then   there is a unique $\hat{\tau}_{0}$ such that $\|\boldsymbol{W}(\omega)\|_{\alpha,[0,\hat{\tau}_{0}]}+\mu\hat{\tau}_{0}^{1-\alpha}=\mu$.

For fixed $\tau_{0}$, we define the following truncated function:
\begin{align*}
W^{\tau_{0}}_{s}(\omega):=\left\{\begin{array}{l}
W_{s}(\omega),\quad s\leq \tau_{0}, \\
W_{\tau_{0}}(\omega),\quad  s> \tau_{0}.
\end{array}\right.
\end{align*}
Then $W^{\tau_{0}}(\omega)$ has a canonical lift, since $W_{s}^{\tau_0}(\omega),s>\tau_0$ is smooth, using Chen's identity  we have that
\begin{align*}
\mathbb{W}^{\tau_{0}}_{s,t}(\omega):=\left\{\begin{array}{l}
\mathbb{W}_{s,t}(\omega),\quad s<t\leq \tau_{0}, \\
\mathbb{W}_{s,\tau_{0}}(\omega),\quad  s\leq \tau_{0}\leq t,\\
0,\quad \tau_{0}\leq s<t.
\end{array}\right.
\end{align*}
Thus, for $\tau\geq\tau_{0}$,
\begin{align*}
&\left|\|\boldsymbol{W}(\omega)\|_{\alpha,[0,\tau]}-\|\boldsymbol{W}(\omega)\|_{\alpha,[0,\tau_{0}]}\right|=\left|\|\boldsymbol{W}(\omega)\|_{\alpha,[0,\tau]}-\|\boldsymbol{W}^{\tau_{0}}(\omega)\|_{\alpha,[0,\tau]}\right|\\
&~~~\leq \left|\interleave W(\omega)\interleave_{\alpha,[0,\tau]}-\interleave W^{\tau_{0}}(\omega)\interleave_{\alpha,[0,\tau]}\right|+\left|\interleave \mathbb{W}(\omega)\interleave_{2\alpha,[0,\tau]}-\interleave \mathbb{W}^{\tau_{0}}(\omega)\interleave_{2\alpha,[0,\tau]}\right|\\
&~~~\leq \interleave W(\omega) \interleave_{\alpha,[\tau_{0},\tau]}+\interleave \mathbb{W}(\omega)-\mathbb{W}^{\tau_{0}}(\omega)\interleave_{2\alpha,[0,\tau]}\\
&~~~\leq (\tau-\tau_{0})^{\alpha^{\prime}-\alpha}\|W(\omega)\|_{\alpha^\prime,[\tau_{0},\tau]}+ \Bigg[(\tau-\tau_{0})^{(2\alpha^{\prime}-2\alpha)}\interleave \mathbb{W}(\omega)\interleave_{2\alpha^\prime,[\tau_{0},\tau]}\\
&~~~~~+\interleave W(\omega) \interleave_{\alpha,[0,\tau_{0}]}\interleave W(\omega) \interleave_{\alpha^{\prime},[\tau_{0},\tau]}(\tau-\tau_{0})^{\alpha^{\prime}-\alpha}\Bigg]\! \vee\! \Bigg[(\tau-\tau_{0})^{2\alpha^{\prime}-2\alpha}\interleave \mathbb{W}(\omega)\interleave_{2\alpha^\prime,[\tau_{0},\tau]}\Bigg]\\
&~~~\leq (\tau-\tau_{0})^{2\alpha^{\prime}-2\alpha}\interleave \mathbb{W}(\omega)\interleave_{2\alpha,[\tau_{0},\tau]}+\left(\interleave W(\omega) \interleave_{\alpha,[0,\tau_{0}]}+1\right)\interleave W(\omega) \interleave_{\alpha^{\prime},[\tau_{0},\tau]}(\tau-\tau_{0})^{\alpha^{\prime}-\alpha}.
\end{align*}
Therefore, we have $\lim_{\tau\downarrow\tau_{0}}\|\boldsymbol{W}(\omega)\|_{\alpha,[0,\tau]}=\|\boldsymbol{W}(\omega)\|_{\alpha,[0,\tau_{0}]}$. Furthermore, we can obtain similar result for $\tau\uparrow\tau_{0}$. Thus, the mapping $\tau\mapsto\|\boldsymbol{W}(\omega)\|_{\alpha,[0,\tau]}+\mu\tau^{1-\alpha}$ is strictly increasing and continuous with respect to $\tau$.  And  the mapping $\tau\mapsto\|\boldsymbol{W}(\omega)\|_{\alpha,[\tau,0]}+\mu|\tau|^{1-\alpha}$ is  strictly decreasing. Then we have that
\begin{align*}
\mu&=\|\boldsymbol{W}(\omega)\|_{\alpha,[0,T(\boldsymbol{W}(\omega))]}+\mu\left(T(\boldsymbol{W}(\omega))\right)^{1-\alpha}\\
&=\|\theta_{T(\boldsymbol{W}(\omega))}\boldsymbol{W}(\omega)\|_{\alpha,[\hat{T}(\theta_{T(\boldsymbol{W}(\omega))}\boldsymbol{W}(\omega)),0]}+\mu(-\hat{T}(\theta_{T(\boldsymbol{W}(\omega))}\boldsymbol{W}(\omega)))^{1-\alpha}
\end{align*}
if and only if $T(\boldsymbol{W}(\omega))=-\hat{T}(\theta_{T(\boldsymbol{W}(\omega))}\boldsymbol{W}(\omega)))$. it is not difficult  to check that the similar property holds for $\hat{T}(\boldsymbol{W}(\omega))$.
\end{proof}
Furthermore, the stopping times  have the following monotonicity.
\begin{lemma}\label{lemma 3.3}
For any $t_{1}\leq t_{2}$, we have
\begin{align*}
t_{1}+\hat{T}(\theta_{t_{1}}\boldsymbol{W}(\omega))\leq t_{2}+\hat{T}(\theta_{t_{2}}\boldsymbol{W}(\omega)).
\end{align*}
\end{lemma}
\begin{proof}
It's proof is similar to\cite[Lemma 3.5]{MR3226746}, we omit it.
\end{proof}
\begin{remark}\label{remark 3.1}
In particular,  if the $t_{2}+\hat{T}(\theta_{t_{2}}\boldsymbol{W}(\omega))\leq t_{1}\leq t_{2}$. Then repeatedly use the formula in  Lemma \ref{lemma 3.3}, we obtain the following property of stoping times.
\begin{align*}
t_{2}\geq t_{1}\geq t_{2}+\hat{T}(\theta_{t_{2}}\boldsymbol{W}(\omega))&\geq t_{1}+\hat{T}(\theta_{t_{1}}\boldsymbol{W}(\omega))\\
&\geq t_{2}+\hat{T}(\theta_{t_{2}}\boldsymbol{W}(\omega))+\hat{T}(\theta_{t_{2}+\hat{T}(\theta_{t_{2}}\boldsymbol{W}(\omega))}\boldsymbol{W}(\omega))\\
&\geq t_{1}+\hat{T}(\theta_{t_{1}}\boldsymbol{W}(\omega))+\hat{T}(\theta_{t_{1}+\hat{T}(\theta_{t_{1}}\boldsymbol{W}(\omega))}\boldsymbol{W}(\omega))\\
&\cdots.
\end{align*}
\end{remark}
\subsection{Global attractor for the non-autonomous dynamical system associated to \texorpdfstring{\eqref{3.2}}{pg}}
 We first establish a priori estimate of the solution to \eqref{3.2}. To this end, for fixed $\omega\in\Omega$, we define a sequence of stopping times $(T_{i}(\boldsymbol{W}(\omega)))_{i\in\mathbb{Z}}$  as follows

\begin{equation}\label{3.6*}
\begin{aligned}
T_{i}(\boldsymbol{W}(\omega))=\left\{\begin{array}{l}
0,\quad i=0, \\
T_{i-1}(\boldsymbol{W}(\omega))+T(\theta_{T_{i-1}(\boldsymbol{W}(\omega))}\boldsymbol{W}(\omega)),\quad i\in \mathbb{N}^{+},\\
T_{i+1}(\boldsymbol{W}(\omega))+\hat{T}(\theta_{T_{i+1}(\boldsymbol{W}(\omega))}\boldsymbol{W}(\omega)),\quad i\in \mathbb{N}^{-}.
\end{array}\right.
\end{aligned}
\end{equation}
According to the definition of stopping times, we have the following cocycle property:
\begin{align}
T_{0}(\boldsymbol{W}(\omega))=0,\quad T_{j}(\boldsymbol{W}(\omega))+T_{i}(\theta_{T_{j}(\boldsymbol{W}(\omega))}\boldsymbol{W}(\omega))=T_{i+j}(\boldsymbol{W}(\omega)),\forall i,j\in\mathbb{Z}.
\end{align}
Furthermore, $T(\boldsymbol{W}(\omega))=T_{1}(\boldsymbol{W}(\omega))$ and $\hat{T}(\boldsymbol{W}(\omega))=T_{-1}(\boldsymbol{W}(\omega))$.
\begin{lemma}\label{lemma 3.4}
For all $\omega\in\Omega$, let $(y,y^{\prime})$ be the solution of \eqref{3.1} with respect to $y_{0}\in\mathcal{B}_{\gamma},\gamma\geq 0$ and  a sequence of stopping times $(T_{i}(\theta_{T_{j}(\boldsymbol{W}(\omega))}\boldsymbol{W}(\omega)))_{i\in\mathbb{Z}}$ which are defined by \eqref{3.6*}. Then we have
\begin{align*}
&\|y,y^{\prime}\|_{W,2\alpha,\gamma,[T_{i-1}(\theta_{T_{j}}\boldsymbol{W}(\omega)),T_{i}(\theta_{T_{j}}\boldsymbol{W}(\omega))]}\\
&\leq C\mu\sum_{m=1}^{i-1}(1+\|y,y^{\prime}\|_{W,2\alpha,\gamma,[T_{m-1}(\theta_{T_{j}}\boldsymbol{W}),T_{m}(\theta_{T_{j}}\boldsymbol{W}(\omega))]})e^{-\lambda(T_{i-1}(\theta_{T_{j}}\boldsymbol{W}(\omega))-T_{m}(\theta_{T_{j}}\boldsymbol{W}(\omega)))}\nonumber\\
&+C\mu(1+\|y,y^{\prime}\|_{W,2\alpha,\gamma,[T_{i-1}(\theta_{T_{j}}\boldsymbol{W}(\omega)),T_{i}(\theta_{T_{j}}\boldsymbol{W}(\omega))]})+Ce^{-\lambda T_{i-1}(\theta_{T_{j}}\boldsymbol{W}(\omega))}\|y_{0}\|_{\gamma},
\end{align*}
where $T_{j}=T_{j}(\boldsymbol{W}(\omega))$, and $\lambda>0$ is the smallest eigenvalue of $-A$, the constant C depends on $S,\sigma,\alpha$ .
\end{lemma}
The proof of this lemma can be found in the Appendix \ref{appendix A}.
We can define a discrete non-autonomous dynamical system based on the cocycle property of the stopping times $(T_{i}(\boldsymbol{W}(\omega)))_{i\in\mathbb{Z}}$ 
with index set $\mathbb{Z}^{+}$ for $\omega\in\Omega$, where $\omega$ as a parameter.  We introduce a new shift as follows
\begin{align*}
&\tilde{\theta}:\mathbb{Z}\times\mathbb{Z}\rightarrow\mathbb{Z},\\
&\tilde{\theta}_{i}j=i+j\quad i,j\in\mathbb{Z}.
\end{align*}
Then we define
\begin{align*}
&\Phi:\mathbb{Z}^{+}\times\mathbb{Z}\times\Omega\times\mathcal{B}\rightarrow\mathcal{B},\\
&\Phi(i,j,\omega,y_{0})=S(T_{i}(\theta_{T_{j}(\boldsymbol{W}(\omega))}\boldsymbol{W}(\omega)))y_{0}\\
&~~~~+\int_{0}^{T_{i}(\theta_{T_{j}(\boldsymbol{W}(\omega))}\boldsymbol{W}(\omega))}
S(T_{i}(\theta_{T_{j}(\boldsymbol{W}(\omega))}\boldsymbol{W}(\omega))-r)F(y_{r})dr\nonumber\\
&~~~~+\int_{0}^{T_{i}(\theta_{T_{j}(\boldsymbol{W}(\omega))}\boldsymbol{W}(\omega))}S(T_{i}(\theta_{T_{j}(\boldsymbol{W}(\omega))}\boldsymbol{W}(\omega))-r)G(y_{r})d\theta_{T_{j}(\boldsymbol{W}(\omega))}\boldsymbol{W}_{r}(\omega)\\
&~~=\varphi(T_{i}(\theta_{T_{j}(\boldsymbol{W}(\omega))}\boldsymbol{W}(\omega)),\theta_{T_{j}(\boldsymbol{W}(\omega))}\omega,y_{0}).
\end{align*}
Note that $\Phi(i,j,\omega,y_{0})$ is the solution of \eqref{3.2} driven by rough noise $\theta_{T_{j}(\boldsymbol{W}(\omega))}\boldsymbol{W}(\omega)$  at time $T_{i}(\theta_{T_{j}(\boldsymbol{W}(\omega))}\boldsymbol{W}(\omega))$. In order to construct the absorbing set of the discrete non-autonomous dynamical system $\Phi$, we need a
discrete Gr\"{o}nwall lemma, its proof can be found in \cite{MR3226746}.
\begin{lemma}[Lemma 3.6,\cite{MR3226746}]
Assume $\lambda^{\ast},v_{0},k_{0},k_{1}<1,k_{2}$ are positive constants and $\{t_{i}\}_{i\in\mathbb{Z}^{+}}$ is a sequence of positive constants with $t_{0}=0$, satisfy
$$t_{i-1}-t_{i-2}\leq -\frac{2}{\lambda^{\ast}}\log k_{1}, \quad \forall i\geq2. $$
Suppose that the sequence of positive constants $\{U_{i}\}_{i\in\mathbb{N}}$ have the following relation:
\begin{align}\label{3.53}
U_{i}&\leq k_{0}v_{0}e^{-\lambda t_{i-1}}+\sum_{m=1}^{i-1}k_{1}U_{m}e^{-\lambda(t_{i-1}-t_{m})}\nonumber\\
&~~+\sum_{m=1}^{i-1}k_{2}e^{-\lambda(t_{i-1}-t_{m})}+k_{2},\quad i=1,2,3,\cdots.
\end{align}
Then $U_{i}$ have the following estimates
\begin{align}
U_{i}&\leq (k_{0}v_{0}+k_{2})(1+k_{1})^{i-1}e^{-\frac{\lambda^{\ast}}{2}t_{i-1}}\nonumber\\
&~~+\sum_{m=1}^{i-1}2k_{2}(1+k_{1})^{i-1-m}e^{-\frac{\lambda^{\ast}}{2}(t_{i-1}-t_{m})},\quad \forall i=1,2,3,\cdots.
\end{align}
\end{lemma}
We can specify coefficients which emerge in  \eqref{3.53}. Let $\mu>0$ be a sufficiently small constant such that
$k_{1}(\mu)=\frac{C\mu}{1-C\mu}<1$ and
\begin{align}\label{3.55}
-\frac{2}{\lambda^{\ast}}\log k_{1}(\mu)>1,
\end{align}
 where constant C emerges in the inequality of Lemma \ref{lemma 3.4}. Furthermore, choosing
 $$\lambda^{\ast}=\lambda,k_{0}=\frac{C}{1-C\mu},k_{2}=k_{1}=k_{1}(\mu),t_{i}=T_{i}(\theta_{T_{j}\boldsymbol{W}(\omega)}\boldsymbol{W}(\omega)).$$
 Based on the above lemmas we have the following corollary.
 \begin{corollary}\label{corollary 3.1}
 Let $y_{0}\in \mathcal{B}$ and  the constant $\mu$ satisfies \eqref{3.55}. Then the discrete non-autonomous dynamical system $\Phi$ has the following estimate
 \begin{align*}
 \|\Phi(i,j,\omega,y_{0})\|&\leq (1+k_{1})^{i-1}e^{-\frac{\lambda}{2}T_{i-1}(\theta_{T_{j}(\boldsymbol{W}(\omega))}\boldsymbol{W}(\omega))}(k_{0}\|y_{0}\|+k_{2})\nonumber\\
 &+\sum_{m=1}^{i-1}2k_{2}(1+k_{1})^{i-1-m}e^{-\frac{\lambda}{2}(T_{i-1}(\theta_{T_{j}(\boldsymbol{W}(\omega))}\boldsymbol{W}(\omega))-T_{m}(\theta_{T_{j}(\boldsymbol{W}(\omega))}\boldsymbol{W}(\omega)))}.
 \end{align*}
 \end{corollary}
 In order to construct the absorbing set of the non-autonomous dynamical system $\Phi$, we  impose the smallness condition for $\omega\in\Omega$. 
 Firstly, we assume that the stopping times satisfy
 \begin{align}\label{3.58}
 1>\liminf_{i\rightarrow-\infty}\frac{|T_{i}(\boldsymbol{W}(\omega))|}{|i|}\geq d_1> \frac{2(\log(1+k_{1}(\mu))+\nu)}{\lambda},
 \end{align}
 where $\nu$ is parameter of  $\mathcal{D}_{\mathbb{Z},\mathcal{B}}^{\nu}$ and $\nu\in [0,\frac{d\lambda}{2})$.
 We  also need
 \begin{align}\label{3.59}
 \nu+d_1>1.
 \end{align}
 Finally, the subexponential growth condition is imposed on the sequence $\left\{|T(\theta_{T_{i}(\boldsymbol{W}(\omega))}\boldsymbol{W}(\omega))|^{-\gamma}\right\}_{i\in\mathbb{Z}}$ for $\omega\in\Omega,\gamma\in(0,(\alpha-\sigma)\wedge(1-\delta)]$, namely
 \begin{align}\label{3.60}
 \lim_{i\rightarrow -\infty}\frac{\log^{+}(|T(\theta_{T_{i}(\boldsymbol{W}(\omega))}\boldsymbol{W}(\omega))|^{-\gamma})}{|i|}=0.
 \end{align}
 In the concrete example(fractional Brownian rough path), we shall check  these conditions.

For the discrete non-autonomous dynamical system $\Phi$, we consider  $\mathcal{D}_{\mathbb{Z},\mathcal{B}}^{\nu}$ which element $\{D(i)\}_{i\in\mathbb{Z}}\subset\mathcal{B}$ included in a ball with center $0$ and radius $r(i)$ and satisfy
 $$\limsup_{i\rightarrow -\infty}\frac{\log^{+}r(i)}{|i|}<\nu.$$
 Next, we shall construct an absorbing set $B\in \mathcal{D}_{\mathbb{Z},\mathcal{B}}^{\nu}$. Thus, our first step is to determine the absorbing set $B$ with respect to $\Phi$.
 \begin{lemma}\label{lemma 3.6}
 Assume that \eqref{3.55}, \eqref{3.58}, \eqref{3.59}, \eqref{3.60} hold. 
 Then there exists a $\mathcal{D}_{\mathbb{Z},\mathcal{B}}^{\nu}-$absorbing set $B$ with center 0 and radius
 \begin{align}\label{3.61}
 R(i,\omega)=2\sum_{m=-\infty}^{0}2k_{2}(1+k_{1})^{-m}e^{\frac{\lambda}{2}T_{m}(\theta_{T_{i-1}(\boldsymbol{W}(\omega))}\boldsymbol{W}(\omega))}.
 \end{align}
  \end{lemma}
 \begin{proof}
For any $ D \in \mathcal{D}_{\mathbb{Z}, \mathcal{B}}^{\nu}$, namely the set $D(j,\omega)\in \mathcal{D}_{\mathbb{Z}, \mathcal{B}}^{\nu}$ and there is a sequence $\{B_{\mathcal{B}}(0,\rho(i,\omega))\}_{i\in\mathbb{Z}}\in \mathcal{D}_{\mathbb{Z}, \mathcal{B}}^{\nu}$ such that  $D(i,\omega)\subset B_{\mathcal{B}}(0,\rho(i,\omega))$. 
In terms of Corollary \ref{corollary 3.1}, we know that
\begin{align}
\sup_{y_{0}\in D(j-i,\omega)}&\|\Phi(i,j-i,\omega,y_{0})\|\leq (1+k_{1})^{i-1}e^{-\frac{\lambda}{2}T_{i-1}(\theta_{T_{j-i}(\boldsymbol{W}(\omega))}\boldsymbol{W}(\omega))}(k_{0}\rho(j-i,\omega)+k_{2})\nonumber\\
 +&\sum_{m=1}^{i-1}2k_{2}(1+k_{1})^{i-1-m}e^{-\frac{\lambda}{2}(T_{i-1}(\theta_{T_{j-i}(\boldsymbol{W}(\omega))}\boldsymbol{W}(\omega))-T_{m}(\theta_{T_{j-i}(\boldsymbol{W}(\omega))}\boldsymbol{W}(\omega)))}.
\end{align}
Using the cocycle property of stopping times
\begin{align*}
T_{i-1}(\theta_{T_{j-i}(\boldsymbol{W}(\omega))}\boldsymbol{W}(\omega))-T_{m}(\theta_{T_{j-i}(\boldsymbol{W}(\omega))}\boldsymbol{W}(\omega))&=T_{i-m-1}(\theta_{T_{j-i+m}(\boldsymbol{W}(\omega))}\boldsymbol{W}(\omega))\\
&=-T_{-i+m+1}(\theta_{T_{j-1}(\boldsymbol{W}(\omega))}\boldsymbol{W}(\omega)),
\end{align*}
then we have
\begin{align}\label{3.73}
\sup_{y_{0}\in D(j-i,\omega)}\|\Phi(i,j-i,\omega,y_{0})\|&\leq (1+k_{1})^{i-1}e^{\frac{\lambda}{2}T_{-i+1}(\theta_{T_{j-1}(\boldsymbol{W}(\omega))}\boldsymbol{W}(\omega))}(k_{0}\rho(j-i,\omega)+k_{2})\nonumber\\
&~~+\sum_{m=2-i}^{0}2k_{2}(1+k_{1})^{-m}e^{\frac{\lambda}{2}T_{m}(\theta_{T_{j-1}(\boldsymbol{W}(\omega))}\boldsymbol{W}(\omega))}.
\end{align}
The second inequality of  \eqref{3.58} means that for any $\epsilon>0$ there exists a number $m(\epsilon)>0$ such that for $|m|>m(\epsilon)$,  we have $|T_{m}(\theta_{T_{j-1}(\boldsymbol{W}(\omega))}\boldsymbol{W}(\omega))|>(d_1-\epsilon)|m|$. Thus, by the third inequality of \eqref{3.58} and the property of the backward $\nu$-exponentially growing of $\rho(j-i,\omega)$, then
\begin{align*}
&(1+k_{1})^{i-1}e^{\frac{\lambda}{2}T_{-i+1}(\theta_{T_{j-1}(\boldsymbol{W}(\omega))}\boldsymbol{W}(\omega))}(k_{0}\rho(j-i,\omega)+k_{2})\\
&~~~~~~~~~~~~~\leq C(k_0,k_2) e^{\nu|j-1|}e^{(i-1)(\log(1+k_1)-\frac{\lambda}{2}(d_1-\epsilon)+\nu)}.
\end{align*}
Hence, the first term of \eqref{3.73} converges to zero as $i\rightarrow \infty$. Similarly, for the second term  of \eqref{3.73}, choosing $\epsilon$ sufficiently small, we have
\begin{align*}
(1+k_{1})^{-m}e^{\frac{\lambda}{2}(T_{m}(\theta_{T_{j-1}(\boldsymbol{W}(\omega))}\boldsymbol{W}(\omega))}&=e^{-m\log(1+k_{1})+\frac{\lambda}{2}T_{m}(\theta_{T_{j-1}(\boldsymbol{W}(\omega))}\boldsymbol{W}(\omega))}\\
&\leq e^{m(-\log(1+k_{1})+\frac{\lambda}{2}(d_1-\epsilon))}\leq e^{m\nu}, \quad m<0.
\end{align*}
Then the second term  of \eqref{3.73} converges to $\frac{R(j)}{2}$ as $i\rightarrow \infty$. The attraction of the $B$ can be easily obtain by the cocycle property of the $\Phi$. Indeed,
\begin{align*}
\|\Phi(i+1,j-i-1,\omega,y_0)\|=\|\Phi(i,j-i,\omega,\Phi(1,j-i-1,\omega,y_0))\|.
\end{align*}
Thus,
\begin{align*}
\sup_{y_{0}\in D(j-i-1)}\|\Phi(i+1,j-i-1,\omega,y_0)\|\leq \sup_{y_{0}^\prime\in D(j-i)}\|\Phi(i,j-i,\omega,y_{0}^{\prime})\|<R(j)
\end{align*}
for sufficiently enough $i\in\mathbb{Z}^{+}$.
\end{proof}
  Lemma \ref{lemma 3.6} shows that a absorbing set $B$ in $\mathcal{B}$. In order to construct pullback attractor for $\Phi$, we need to illustrate $B\in \mathcal{D}_{\mathbb{Z},\mathcal{B}}^{\nu}$. We give this fact by the following lemma and its proof can be found in \cite[Lemma 3.10]{MR3226746}.
\begin{lemma}
The absorbing sets $B(i,\omega)$ which are constructed in Lemma \ref{lemma 3.6} belong to $\mathcal{D}_{\mathbb{Z},\mathcal{B}}^{\nu}$.
\end{lemma}
Based on the above considerations, we construct the pullback attractor of the discrete non-autonomous dynamical system $\Phi$.
\begin{theorem}\label{theorem 3.2}
Assume $\omega\in\Omega$ which satisfy the assumptions \eqref{3.55},\eqref{3.58},\eqref{3.59},\eqref{3.60}. Then the discrete non-autonomous dynamical system $\Phi(\cdot,\omega)$has a pullback attractor $\{\mathcal{A}(i,\omega)\}_{i\in\mathbb{Z}}$ on $\mathcal{D}_{\mathbb{Z},\mathcal{B}}^{\nu}$.
\end{theorem}
\begin{proof}
We give the outline of the proof, according to Lemma \ref{lemma 2.3}, we need to illustrate $\Phi$ for initial data has continuously dependence(Section 5),  and we need to check the existence of the compact pullback absorbing set. In terms of the absorbing set $B(j,\omega)\in \mathcal{D}_{\mathbb{Z},\mathcal{B}}^{\nu}$, let $T^{\ast}(j)$ is absorbing time of $B(j,\omega)$, then $\Phi(T^{\ast},-T^{\ast}+j-1,\omega, B(-T^{\ast}+j-1,\omega))$ is also an absorbing set.  In order to get the compactness of  absorbing sets.  Firstly, for any $y_0\in D(j-1,\omega)\in \mathcal{D}_{\mathbb{Z},\mathcal{B}}^{\nu}$,  the solutions  $y_{T_{1}(\theta_{T_{j-1}(\boldsymbol{W}(\omega))}\boldsymbol{W}(\omega)}=\Phi(1,j-1,\omega, y_0)$ have  more nicely regularity. Indeed, similar to the computation in Appendix A,  we have the following estimate
\begin{align*}
	\left\|\Phi(1,j-1,\omega, y_0)\right\|_{\gamma}&\leq \frac{C}{|T_{1}(\theta_{T_{j-1}(\boldsymbol{W}(\omega))}\boldsymbol{W}(\omega))|^{\gamma}}\|y_0\|\\
&~~+C\mu(1+\|y,y^\prime\|_{W,2\alpha,0,[0,T_{1}(\theta_{T_{j-1}(\boldsymbol{W}(\omega))}\boldsymbol{W}(\omega))]}),
\end{align*}
where we have to add the condition $\gamma<(1-\delta)\wedge (\alpha-\sigma)$ to  guarantee that above inequality holds.  In addition, let $i=1,\gamma=0$ in Lemma \ref{lemma 3.4}, then
\begin{align*}
\|y,y^\prime\|_{W,2\alpha,0,[0,T_{1}(\theta_{T_{j-1}(\boldsymbol{W}(\omega))}\boldsymbol{W}(\omega))]}\leq \frac{C}{1-C\mu}\|y_0\|+\frac{C\mu}{1-C\mu}.
\end{align*}
Thus, \eqref{3.60} shows that $\Phi(1,j-1,\omega, D(j-1,\omega))\in \mathcal{D}_{\mathbb{Z},\mathcal{B}_{\gamma}}^{\nu}$. Then we can construct a compact pullback absorbing set $C(j,\omega)$ as follows
$$C(j,\omega):=\overline{\Phi(1,j-1,\omega,\Phi(T^{\ast},-T^{\ast}+j-1,\omega, B(-T^{\ast}+j-1,\omega)))}.$$
Note that the compactness of $C(j,\omega)$ can be get from $\mathcal{B}_{\gamma}$ compactly embedded into $\mathcal{B}$  and $C(j,\omega)$ is bounded in $\mathcal{B}_{\gamma}$. So
$$\mathcal{A}(i,\omega)=\cap_{i^\prime\in\mathbb{Z}^{+}}\overline{\cup_{j\geq i^\prime}\Phi(j,i-j,\omega,C(i-j,\omega))},~~i\in\mathbb{Z}.$$
\end{proof}
\subsection{The absorbing set for non-autonomous dynamical system \texorpdfstring{$\varphi$}{qg} }
We shall use the above results to study  the dynamical behavior of the non-autonomous dynamical system $\varphi$ in this subsection. Different from the previous subsection, we consider the sample `$\omega$' instead of `$i$' for the continuous dynamical system $\varphi$. 
Thus, we will try to prove that $\mathcal{A}(\omega)=\mathcal{A}(0,\omega)$ is an invariant attracting set  in $\mathcal{D}^{0}_{\mathbb{R},\mathcal{B}}$ with respect to $\varphi$, where we say $D(\omega)\in\mathcal{D}^{0}_{\mathbb{R},\mathcal{B}}$ for each $\omega\in\Omega$, if there exists a radius $r(\omega)$ such that $D(\omega)\subset B_{\mathcal{B}}(0,r(\omega))$ and
$$\limsup_{\mathbb{R}\ni t\rightarrow-\infty}\frac{\log^{+}r(\theta_{t}\omega)}{|t|}=0.$$
We construct the following sets $$G(i,\omega)=\overline{\bigcup_{t\in(T_{i-1}(\boldsymbol{W}(\omega),T_{i}(\boldsymbol{W}(\omega))]}D(\theta_{t}\omega)}$$
for $D\in \mathcal{D}^{0}_{\mathbb{R},\mathcal{B}}$ and $i\in \mathbb{Z},\omega\in\Omega$, then $G(i,\omega)\in \mathcal{D}^{0}_{\mathbb{Z},\mathcal{B}}$. Indeed, if not, then there exists a $\omega\in\Omega$ and subsequence $i^{\prime}:=\{i^{\prime}(i,\omega)\}_{i\in\mathbb{Z}^{-}}$,$i^{\prime}\in\mathbb{Z}^{-}$, $t_{i^{\prime}}\in(T_{i^{\prime}-1}(\boldsymbol{W}(\omega)),T_{i^{\prime}}(\boldsymbol{W}(\omega))]$ such that
$$\limsup_{i^{\prime}\rightarrow -\infty}\frac{\log^{+}\|y_{i^{\prime}}\|}{|i^{\prime}|}>0,$$
for any $y_{i^{\prime}}\in D(\theta_{t_{i^{\prime}}}\omega)$. However, we have
\begin{align*}
\limsup_{i^{\prime}\rightarrow -\infty}\frac{\log^{+}\|y_{i^{\prime}}\|}{|i^{\prime}|}&\leq\limsup_{i^{\prime}\rightarrow -\infty}\frac{\log^{+}\|y_{i^{\prime}}\|}{|T_{i^{\prime}-1}|}\limsup_{i^{\prime}\rightarrow -\infty}\frac{|T_{i^{\prime}-1}|}{|i^{\prime}|}\\
&\leq \limsup_{i^{\prime}\rightarrow -\infty}\frac{\log^{+}\|y_{i^{\prime}}\|}{|t_{i^{\prime}}|}\limsup_{i^{\prime}\rightarrow -\infty}\frac{|T_{i^{\prime}-1}|}{|i^{\prime}|}=0,
\end{align*}
where the last equality we use the fact $y_{i^{\prime}}\in D(\theta_{t_{i^{\prime}}}\omega)$ and $\frac{|T_{i^{\prime}-1}|}{|i^{\prime}|}\leq \frac{|i^{\prime}-1|}{|i^{\prime}|}< 2$.

Before proving  that $\mathcal{A}(\omega)$ is an invariant attracting set of continuous dynamical system $\varphi$, we need the following lemma.
\begin{lemma}\label{lemma 3.8}
Let $D\in \mathcal{D}^{0}_{R,\mathcal{B}}$, then the set
$$E_{D}(i,\omega)=\overline{\bigcup_{-t\in[\hat{T}(\theta_{T_{i}(\boldsymbol{W}(\omega))}\boldsymbol{W}(\omega)),0]}\bigcup_{y_{0}\in D(\theta_{-t}\theta_{T_{i}(\boldsymbol{W}(\omega))}\boldsymbol{W}(\omega))}\varphi(t,\theta_{-t}\theta_{T_{i}(\boldsymbol{W}(\omega))}\omega,y_{0})}$$
belongs to $\mathcal{D}^{0}_{\mathbb{Z},\mathcal{B}}$. Furthermore, we can define a set
$$H_{D}(i,\omega)=\overline{\bigcup_{t\in[0,T(\theta_{T_{i}(\boldsymbol{W}(\omega))}\boldsymbol{W}(\omega))]}\bigcup_{y_{0}\in D(\theta_{T_{i}(\boldsymbol{W}(\omega))}\boldsymbol{W}(\omega))}\varphi(t,\theta_{T_{i}(\boldsymbol{W}(\omega))}\omega,y_{0})}$$
in $\mathcal{D}^{\nu}_{\mathbb{Z},\mathcal{B}}$ for $D\in \mathcal{D}^{\nu}_{\mathbb{R},\mathcal{B}}$.
\end{lemma}
\begin{proof}
We first check $E_{D}(i,\omega)\in \mathcal{D}^{0}_{\mathbb{Z},\mathcal{B}}$ for $D\in \mathcal{D}^{0}_{\mathbb{R},\mathcal{B}}$. For any $-t\in[\hat{T}(\theta_{T_{i}(\boldsymbol{W}(\omega))}\boldsymbol{W}(\omega)),0]$ and $y_{0}\in D(\theta_{-t}\theta_{T_{i}(\boldsymbol{W}(\omega))}\omega)$, by the fact $$\| \theta_{-t}\theta_{T_{i}(\boldsymbol{W}(\omega))}\boldsymbol{W}(\omega)\|_{\alpha,[0,t]}\leq \|\theta_{T_{i}(\boldsymbol{W}(\omega))}\boldsymbol{W}(\omega)\|_{\alpha,[\hat{T}(\theta_{T_{i}(\boldsymbol{W}(\omega))}\boldsymbol{W}(\omega)),0]}\leq\mu,$$ similar to Lemma \ref{lemma A.1}-\ref{lemma A.2} and \ref{lemma A.4}-\ref{lemma A.5} for $\gamma=0$, we have
\begin{align*}
\|\varphi(t,\theta_{-t}\theta_{T_{i}(\boldsymbol{W}(\omega))}\omega,y_{0})\|\leq C\|y_{0}\|+C\mu(1+\|y,y^{\prime}\|_{W,2\alpha,0,[0,T_{1}(\theta_{T_{i-1}(\boldsymbol{W}(\omega))}\boldsymbol{W}(\omega)]}).
\end{align*}
Furthermore,  let $i=1, \gamma=0$ in Lemma \ref{lemma 3.4} we get
\begin{align}\label{3.74}
\|y,y^{\prime}\|_{W,2\alpha,0,[0,T_{1}(\theta_{T_{i-1}(\boldsymbol{W}(\omega))}\boldsymbol{W}(\omega)]}\leq \frac{C}{1-C\mu}\|y_{0}\|+\frac{C\mu}{1-C\mu}.
\end{align}
Thus, we have
\begin{align}\label{3.75}
\|\varphi(t,\theta_{-t}\theta_{T_{i}(\boldsymbol{W}(\omega))}\omega,y_{0})\|\leq C\|y_{0}\|+C\mu\left(\frac{C}{1-C\mu}\|y_{0}\|+\frac{C\mu}{1-C\mu}+1\right).
\end{align}
In view of $y_{0}\in D(\theta_{-t}\theta_{T_{i}(\boldsymbol{W}(\omega))}\omega)$, hence we obtain
$$\|\varphi(t,\theta_{-t}\theta_{T_{i}(\boldsymbol{W}(\omega))}\omega,y_{0})\|\leq \sup_{y_{0}\in G(i,\omega)}C\|y_{0}\|\left(1+\frac{C\mu}{1-C\mu}\right)+C\mu\left(1+\frac{C\mu}{1-C\mu}\right).$$
Then $E_{D}(i,\omega)\in \mathcal{D}^{0}_{\mathbb{Z},\mathcal{B}}$. Similarly, we can prove that $H_{D}(i,\omega)\in\mathcal{D}^{\nu}_{\mathbb{R},\mathcal{B}}$.
\end{proof}
\begin{lemma}\label{lemma 3.9}
If conditions in Theorem \ref{theorem 3.2} hold and the family of sets $\mathcal{A}(\omega):=\mathcal{A}(0,\omega)$ which is defined in Theorem \ref{theorem 3.2}. Then $\mathcal{A}$  is attracting and invariant with respect to $\varphi$ in $\mathcal{D}^{0}_{\mathbb{R},\mathcal{B}}$.
\end{lemma}
\begin{proof}
Its' proof is same as \cite[Lemma 3.13]{MR3226746}, we omit it here.
\end{proof}
The next lemma shows that the absorbing set $\mathcal{A}$ is not a attractor, since it is not an absorbing set  on stopping times $\{T_{i}(\boldsymbol{W}(\omega))\}_{i\in\mathbb{Z}}$.
\begin{lemma}\label{lemma 3.10*}
For each $\omega\in\Omega$, let
\begin{align*}
\hat{d}:=\limsup_{i\rightarrow -\infty}\frac{|T_{i}(\boldsymbol{W}(\omega))|}{|i|}\leq 1,\quad \check{d}:=\liminf_{i\rightarrow -\infty}\frac{|T_{i}(\boldsymbol{W}(\omega))|}{|i|}\geq d_1>0
\end{align*}
hold, then we have
\begin{itemize}
  \item $\mathcal{A}\in\mathcal{D}^{\nu/\check{d}}_{\mathbb{R},\mathcal{B}}$;
  \item If $\hat{d}>\check{d}$ and $\nu>0$, then there exists a sample $\omega\in\Omega$ and $D\in \mathcal{D}^{\nu/\check{d}}_{\mathbb{R},\mathcal{B}}$ such that the  range of  mapping $t\rightarrow D(\theta_{t}\omega)$ for $t=T_{i}(\boldsymbol{W}(\omega)),i\in \mathbb{Z}$ in $\mathcal{D}_{\mathbb{Z},\mathcal{B}}^{\nu\hat{d}/\check{d}}$ but not in $\mathcal{D}_{\mathbb{Z},\mathcal{B}}^{\nu}$;
  \item Similarly, if $D\in \mathcal{D}_{\mathbb{R},\mathcal{B}}^{0}$, then the range of mapping $t\rightarrow D(\theta_{t}\omega)$ for $t=T_{i}(\boldsymbol{W}(\omega)),i\in \mathbb{Z}$ in $\mathcal{D}_{\mathbb{Z},\mathcal{B}}^{0}$.
\end{itemize}
\end{lemma}
\begin{proof}
For the first conclusion, in view of $\mathcal{A}(i,\omega)\in \mathcal{D}^{\nu}_{\mathbb{Z},\mathcal{B}}$. Then we have
$$\limsup_{i\rightarrow -\infty}\frac{\log^{+}(\sup_{y\in\mathcal{A}(i,\omega)}|y|)}{|T_{i}(\boldsymbol{W}(\omega))|}\leq \limsup_{i\rightarrow -\infty}\frac{\log^{+}(\sup_{y\in\mathcal{A}(i,\omega)}|y|)}{|i|}\limsup_{i\rightarrow -\infty}\frac{i}{|T_{i}(\boldsymbol{W}(\omega))|}\leq \nu/\check{d}.$$
Furthermore, for any $t<0$, there exists a $i(\omega)\in\mathbb{Z}^{-}$ such that $t\in (T_{i-1}(\boldsymbol{W}(\omega)),T_{i}(\boldsymbol{W}(\omega))]$ and $t=\tau+T_{i-1}(\boldsymbol{W}(\omega))$. Then we have
\begin{align*}
&\limsup_{t\rightarrow -\infty}\frac{\log^{+}\sup_{y\in \mathcal{A}(i-1,\omega)}\|\varphi(\tau,\theta_{T_{i-1}(\boldsymbol{W}(\omega))}\omega,y)\|}{|t|}\\
\leq&\limsup_{i\rightarrow -\infty}\frac{\log^{+}\sup_{y\in \mathcal{A}(i-1,\omega)}\|\varphi(\tau,\theta_{T_{i-1}(\boldsymbol{W})}\omega,y)\|}{|T_{i-1}(\boldsymbol{W}(\omega)) |}\leq \frac{\nu}{\check{d}},
\end{align*}
where we use the fact $\|\varphi(\tau,\theta_{T_{i-1}(\boldsymbol{W}(\omega))}\omega,y)\|\leq\sup_{y\in \mathcal{A}(i-1,\omega)}C\|y\|(1+\frac{C\mu}{1+C\mu})+C\mu(1+\frac{C\mu}{1+C\mu})$. Hence, $\mathcal{A}\in\mathcal{D}^{\nu/\check{d}}_{\mathbb{R},\mathcal{B}}$. The proof of the second property and the third property are similar, we only give the proof for the second conclusion. We can construct $D\in\mathcal{D}^{\frac{\nu}{\check{d}}}_{\mathbb{R},\mathcal{B}}$ such that $D(\theta_t\omega)=\{0\}$ for $t\neq T_{i}(\boldsymbol{W}(\omega))$, furthermore, for $t=T_{i^\prime}(\boldsymbol{W}(\omega))$, choosing $D(\theta_{T_{i^\prime}(\boldsymbol{W}(\omega)}\omega)=\{y_{i^{\prime}(i,\omega)}\}_{i\in\mathbb{Z}}$, where $i^{\prime}$ is a subsequence of $i$ such that $$\limsup_{i\rightarrow -\infty}\frac{|T_{i}(\boldsymbol{W}(\omega)|)}{|i|}=\lim_{i^\prime\rightarrow -\infty}\frac{|T_{i^{\prime}}(\boldsymbol{W}(\omega)|)}{|i^{\prime}|},$$
and for a sufficiently small $\epsilon>0$,
$$\lim_{i^{\prime}\rightarrow -\infty}\frac{|y_{i^{\prime}}|}{|T_{i^{\prime}}(\boldsymbol{W}(\omega))|}=\frac{\nu}{\check{d}}-\epsilon.$$
Then, it is easy to check
$$\limsup_{i\rightarrow -\infty}\frac{\log^{+}(\sup_{y\in D(i,\omega)}|y|)}{|i|}=(\frac{\nu}{\check{d}}-\epsilon)\hat{d}\leq \nu\hat{d}/\check{d}.$$
\end{proof}
\section{Random attractors for RPDEs driven by GFBRP}
In this section, we consider the random settings for RPDEs, namely the sample space $\Omega$ equipped with a measurable structure. Firstly, we introduce a metric dynamical system to describe the evolution of the noise. In addition, we shall prove that the non-autonomous dynamical system $\varphi$  is a random dynamical system  under the random settings and it has a random attractor. To this end, we first define the metric dynamical system. fBm has the following form:  we call a stochastic process $W(t)$ a $d$-dimensional fBm with Hurst index $H\in(0,1)$, if each component $W^{i}(t),i=1,\cdots,d$ is an independent, identically distributed, centered Gaussian process and has the following covariance
 $$R(W^{i}(t),W^{i}(s)):=E(W^{i}(t)W^{i}(s))=\frac{q^{2}}{2}(|t|^{2H}+|s|^{2H}-|t-s|^{2H})$$
 for $t,s\in \mathbb{R}$, and we require $q>0$ sufficiently small in following settings. Firstly, we define a quadruple $(C_{0}(\mathbb{R},\mathbb{R}^{d}),\mathcal{B}(C_{0}(\mathbb{R},\mathbb{R}^{d})),\mathbb{P},\theta)$, where $C_{0}(\mathbb{R},\mathbb{R}^{d})$ denotes the set of all $\mathbb{R}^{d}$-values continuous functions which elements are zero at zero, it is endowed with the compact open topology and  $\mathcal{B}(C_{0}(\mathbb{R},\mathbb{R}^{d}))$ is a Borel $\sigma$-algebra which is generated by $C_{0}(\mathbb{R},\mathbb{R}^{d})$. $\mathbb{P}$ is the law of fBm and $\theta$ is the Wiener shift which is introduced in previous section. Then the classical results\cite{MR2095071,MR2836654} tell us that $(C_{0}(\mathbb{R},\mathbb{R}^{d}),\mathcal{B}(C_{0}(\mathbb{R},\mathbb{R}^{d})),\mathbb{P},\theta)$ is an ergodic metric dynamical system.

 Secondly, according to Kolmogorov's continuous criterion, there exists a full measure set $\Omega_1\subset C_{0}(\mathbb{R},\mathbb{R}^d)$  such that its' elements have  $\alpha^{\prime}$-H\"{o}lder continuous paths for every  $\frac{1}{3}<\alpha<\alpha^{\prime}<H\leq \frac{1}{2}$ on any interval $[-T,T]$. In particular, $\Omega_1\in\mathcal{B}(C_{0}(\mathbb{R},\mathbb{R}^d))$ and it is $(\theta)_{t\in\mathbb{R}}$-invariant(see \cite{MR3072986}). Furthermore,  by Corollary \ref{corollary 5.1*}(Section 5), there exists a $(\theta_t)_{t\in\mathbb{R}}$-invariant full set $\Omega_{2}\in\mathcal{B}(C_{0}(\mathbb{R},\mathbb{R}^{d}))$ such that each $\omega\in\Omega_{2}$ has the canonical lift, namely it's second order process  as the limit of  the canonical lift  $C^1$-smooth paths,  we call $S$  a canonical lift mapping, if $S: C^{1}(\mathbb{R},\mathbb{R}^{d})\mapsto G^{2}(\mathbb{R}^{d})$, $S(x)=\left(1,x_{s,t},\int_{s}^{t}x_{s,r}dx_{r}\right)$, where $G^{2}(R^{d})$ is the free step-2 nilpotent group over $\mathbb{R}^{d}$(see\cite[Page 23]{MR4174393}). 

  Thirdly,  according to the  denseness property and countability of rational numbers, the continuity of  $\|\boldsymbol{W}^\eta(\omega)\|_{\alpha^\prime}$ for $\eta$ and ergodic theorem\cite[p538]{MR1723992}, there exists a $(\theta_{t})_{t\in R}$-invariant full set $\Omega_{3}\in \mathcal{B}(C_0(\mathbb{R},\mathbb{R}^d)$  such that
\begin{align}\label{4.1**}
&\lim_{t\rightarrow \pm \infty}\frac{1}{t}\int_{0}^{t}\left(\frac{\sup_{r\in[0,1]}\|\theta_{r+s}\boldsymbol{W}(\omega)\|_{\alpha^{\prime},[-1,0]}+\mu}{\mu}\right)^{\frac{1}{\alpha^{\prime}-\alpha}}ds \nonumber\\
&=E_{\mathbb{P}}\left(\frac{\sup_{r\in[0,1]}\|\theta_{r}\boldsymbol{W}\|_{\alpha^{\prime},[-1,0]}+\mu}{\mu}\right)^{\frac{1}{\alpha^{\prime}-\alpha}}:=d_1^{-1}<\infty,
\end{align}
and for any $\eta\in (0,1]$, there exists a  $d^\eta$ such that \eqref{4.1**} holds for $\boldsymbol{W}^\eta$ which is defined in Section 5, it approximates to $\boldsymbol{W}$. Moreover, due to \cite[Lemma 3.3]{caogao}, $d^\eta>0$ have a uniform lower bound for given $q$.
And under the assumption $q$ sufficiently small, $d_1$ and $d^\eta$ shall close sufficiently  to $1$. 

 Finally, since  GFBRP $\boldsymbol{W}$ and $\boldsymbol{W}^\eta$  can be extended to any finite interval \cite{MR4266219,caogao}, based on the  \cite[Theorem 3.1]{caogao}, we  know that $ \|\boldsymbol{W}\|_{\alpha^\prime,[-1-|t|,1+|t|]}$ and $ \|\boldsymbol{W}^\eta\|_{\alpha^\prime,[-1-|t|,1+|t|]}$  have
  any finite  order moment, then by Proposition 4.13\cite{MR1723992} and  the similar method in step $3$, we have a $(\theta_{t})_{t\in R}$-invariant set $\Omega_{4}\in \mathcal{B}(C_0(\mathbb{R},\mathbb{R}^d) $ of full measure such that the mapping $t\rightarrow \sup_{r\in[0,1]}\|\theta_{t+r}\boldsymbol{W}(\omega)\|_{\alpha^\prime,[-1,0]}$ and $ \sup_{r\in[0,1]}\|\theta_{t+r}\boldsymbol{W}^\eta(\omega)\|_{\alpha^\prime,[-1,0]}$ are sublinear growth. Hence, let $\Omega=\Omega_1\cap\Omega_2\cap\Omega_3\cap\Omega_4$, $\mathcal{F}$ is  a $\sigma$-algebra of $\mathcal{B}(C_0(\mathbb{R},\mathbb{R}^d)$ with respect to $\Omega$, the measure $\mathbb{P}$  is the restriction of  the distribution  of  fBm over $\mathcal{F}$. Hence,  we  construct a new ergodic metric dynamical system $(\Omega,\mathcal{F},\mathbb{P},\theta_{t})$.

 Now, we shall to illustrate that the non-autonomous dynamical system $\varphi$ is a random dynamical system in random settings.
 \begin{lemma}
 The non-autonomous dynamical system $\varphi$ is  a random dynamical system and the stopping times $T(\boldsymbol{W})$ and $\hat{T}(\boldsymbol{W})$ are measurable.
 \end{lemma}
 \begin{proof}
 $\varphi$ is random dynamical system, it has been proved in \cite{hesse2021global,kuehn2021center}. The measurability can be obtained easily, by the Wong-Zakai approximation of the noise in $\alpha^{\prime}$-H\"{o}lder rough path, we know that the mapping $\omega\mapsto \|\boldsymbol{W}(\omega)\|_{\alpha^{\prime},[0,\tau]}$ is continuous, thus it is measurable. Then $T(\boldsymbol{W})$ is measurable. Furthermore, the measurability of $\hat{T}(\boldsymbol{W})$ can be obtained in the same way.
 \end{proof}
 In order to construct the inequality of \eqref{3.58}, we need the following lemma.
 \begin{lemma}\label{lemma 4.2}
 Let random variable $N(\omega)\in \mathbb{N}$ is the number of stopping times $(T_{i}(\boldsymbol{W}(\omega)))_{i\in \mathbb{Z}}$ in $[-1,0]$. Then for every $\omega\in\Omega$, we have
 $$N(\omega)\leq \left(\frac{\|\boldsymbol{W}(\omega)\|_{\alpha^{\prime},[-1,0]}+\mu}{\mu}\right)^{\frac{1}{\alpha^{\prime}-\alpha}},\quad \text{for} \quad \frac{1}{3}<\alpha<\alpha^{\prime}<H\leq\frac{1}{2}.$$
 \end{lemma}
 \begin{proof}
 In view of $\boldsymbol{W}(\omega)\in\mathcal{C}_g^{\alpha^\prime}$ and $|t-s|\leq |\hat{T}(\boldsymbol{W}(\omega))|\leq 1$, we have
 \begin{align*}
 \mu&=\sup_{\hat{T}(\boldsymbol{W}(\omega))\leq s<t\leq 0}\left(\frac{|W_{t}(\omega)-W_{s}(\omega)|}{|t-s|^{\alpha}}+\frac{|\mathbb{W}_{s,t}(\omega)|}{|t-s|^{2\alpha}}\right)+\mu(-\hat{T}(\boldsymbol{W}(\omega)))^{1-\alpha}\\
 &\leq \left(\sup_{\hat{T}(\boldsymbol{X})\leq s<t\leq 0}\left(\frac{|W_{t}(\omega)-W_{s}(\omega)|}{|t-s|^{\alpha^{\prime}}} +\frac{|\mathbb{W}_{s,t}(\omega)|}{|t-s|^{2\alpha^{\prime}}}\right)+\mu(-\hat{T}(\boldsymbol{W}(\omega)))^{1-\alpha^{\prime}}\right)(-\hat{T}(\boldsymbol{W}(\omega)))^{\alpha^{\prime}-\alpha}\\
 &\leq (\|\boldsymbol{W}(\omega)\|_{\alpha^{\prime},[-1,0]}+\mu)(-\hat{T}(\boldsymbol{W}(\omega)))^{\alpha^{\prime}-\alpha}.
 \end{align*}
 Thus, we obtain
 $$|\hat{T}(\boldsymbol{W}(\omega))|\geq \left(\frac{\mu}{\|\boldsymbol{W}(\omega)\|_{\alpha^{\prime},[-1,0]}+\mu}\right)^{\frac{1}{\alpha^{\prime}-\alpha}}.$$
 Let $i>0$ be the largest integer such that $T_{-i-1}(\boldsymbol{W}(\omega))\leq -1$. Then, let $N(\omega)=i+1$ and using the property of cocycle for the stopping times and the above estimate, we can get
 \begin{align*}
1\geq |T_{-N(\omega)}(\boldsymbol{W}(\omega))|=\sum^{N(\omega)-1}_{j=0}|\hat{T}(\theta_{T_{-j}(\boldsymbol{W}(\omega))}\boldsymbol{W}(\omega))|\geq N(\omega)\left(\frac{\mu}{\|\boldsymbol{W}(\omega)\|_{\alpha^{\prime},[-1,0]}+\mu}\right)^{\frac{1}{\alpha^{\prime}-\alpha}}.
\end{align*}
Hence, we complete the proof of this lemma.
 \end{proof}
 \begin{lemma}
 We have the following relation holds
 $$\liminf_{i\rightarrow -\infty}\frac{|T_{i}(\boldsymbol{W}(\omega))|}{|i|}=\liminf_{i\rightarrow -\infty}\frac{T_{i}(\boldsymbol{W}(\omega))}{i}\geq d_1>0$$
 for $\omega\in\Omega$. Furthermore, $|T(\theta_{T_{i}(\boldsymbol{W}(\omega))}\boldsymbol{W}(\omega))|^{-\gamma},\omega\in\Omega,\gamma\in(0,(\alpha-\sigma)\wedge(1-\delta)]$ is subexponentially growing on $\Omega$ for $i\rightarrow -\infty$.
 \end{lemma}
 \begin{proof}
 The proof of this lemma can be found in \cite[Lemma 4.5]{MR3226746}.
 \end{proof}
The constant $\mu$ and $\nu$ can be chosen as we determine $d_1$.  Now we conclude the random attractor to finish this section, and its proof is same as \cite[Theorem 4.6]{MR3226746}.
 \begin{theorem}
 The pullback attractor $\mathcal{A}$ which is stated in Lemma \ref{lemma 3.9} for  non-autonomous system $\varphi$ is random attractor and it attracts random tempered sets in $\hat{\mathcal{D}}$.

 \end{theorem}
\begin{remark}\label{remark 4.1*}
	According to \cite[Lemma 2.1 ]{MR3226746}, $\nu$-exponentially growing random  sets is also subexponentially growing. Thus,  similar to the first item of Lemma \ref{lemma 3.10*} , $B(\omega)=B(0,\omega)$ is subexponentially growing on $\mathbb{R}$. Furthermore,
	for any $t>0$, we can find a $i^{*}(\omega)\in \mathbb{Z}^{-}$ such that
	$-t\in (T_{i^{*}-1}(\boldsymbol{W}(\omega)),T_{i^{*}}(\boldsymbol{W}(\omega))]$, then for any $D\in\mathcal{D}^{0}_{\mathbb{R},\mathcal{B}}$, by the property of  cocycle and stopping times  we have
	\begin{align*}
	&\varphi(t,\theta_{-t}\boldsymbol{X},D(\theta_{-t}\omega))\\
	&=\varphi(T_{-i^{*}}(\theta_{T_{i^{*}}}\boldsymbol{W}(\omega)),\theta_{T_{i^{*}}}\omega,\varphi(t-T_{-i^{*}}(\theta_{T_{i^{*}}}\boldsymbol{W}(\omega)),\theta_{-t}\omega,D(\theta_{-t}\omega)))\\
	&=\varphi(T_{-i^{*}}(\theta_{T_{i^{*}}}\boldsymbol{W}(\omega)),\theta_{T_{i^{*}}}\omega,\varphi(t-T_{-i^{*}}(\theta_{T_{i^{*}}}\boldsymbol{W}(\omega)),\theta_{-t-T_{i^{*}}}\theta_{T_{i^{*}}}\omega,D(\theta_{-t-T_{i^{*}}}\theta_{T_{i^{*}}}\omega)))\\
	&=\Phi(-i^{*},i^{*},\omega,\varphi(t-T_{-i^{*}}(\theta_{T_{i^{*}}}\boldsymbol{W}(\omega)),\theta_{-t-T_{i^{*}}}\theta_{T_{i^{*}}}\omega,D(\theta_{-t-T_{i^{*}}}\theta_{T_{i^{*}}}\omega)))\\
	&\subset \Phi(-i^{*},i^{*},\omega,E_{D}(i^{*},\omega))\subset B(\omega).
	\end{align*}
In addition, similar to Theorem \ref{theorem 3.2}, there exists a compact absorbing set
$$
C(\omega):=\overline{\Phi\left(1,-1, \omega, \Phi\left(T,-T-1, \omega, B\left(\theta_{-T-1} \omega\right)\right)\right)} \subset B(\omega),
$$
where $T$ is the absorption time corresponding to $B$.  Random attractors $\mathcal{A}(\omega)$ are also given by
$$
\mathcal{A}(\omega)=\bigcap_{s \geq 0} \overline{\bigcup_{t \geq s} \varphi\left(t, \theta_{-t} \omega, C\left(\theta_{-t} \omega\right)\right)}=\bigcap_{s \geq 0} \overline{\bigcup_{t \geq s} \varphi\left(t, \theta_{-t} \omega, B\left(\theta_{-t} \omega\right)\right)}.
$$
\end{remark}
\section{Wong-Zakai approximation for evolution equation driven by GFBRP}\label{Wong-Zakai}
\subsection{Wong-Zakai approximation for GFBRP}
In this subsection, we introduce a  smooth  and stationary approximation scheme. Let $(\Omega,\mathcal{F},\mathbb{P})$ be probability space which is defined in previous section. $(\theta_{t})_{t\in R}$ is a Wiener shift and $\theta_{t}W_{\cdot}(\omega)=W_{\cdot+t}(\omega)-W_t(\omega)=\omega_{\cdot+t}-\omega_t$ for any $\omega\in\Omega$. As \cite{MR3680943} and \cite{MR3912728} did, for any $\eta\in (0,1)$, we define a random variable  $\mathcal{G}_{\eta}: \Omega \rightarrow \mathbb{R}^{d}$
$$
\mathcal{G}_{\eta}(\omega)=\frac{1}{\eta} \omega_{\eta}.
$$
Then we have
$$
\mathcal{G}_{\eta}\left(\theta_{t} \omega\right)=\frac{1}{\eta}(\omega_{t+\eta}-\omega_t).
$$
By the properties of fBm,  $\mathcal{G}_{\eta}(\theta_{t}\omega)$ is a stochastic process with  normal distribution and stationary increments. Let
\begin{equation}\nonumber
W^{\eta}(t,\omega):=\int_{0}^{t}\mathcal{G}_{\eta}(\theta_{s}\omega)ds.
\end{equation}
$W^{\eta}(t, \omega)$ can be viewed as an approximation of the fractional Brownian motion. Furthermore, we denote by
\begin{equation}\nonumber
\mathbb{W}^{\eta}(\omega)_{s,t}:=\int_{s}^{t}W^{\eta}(\cdot,\omega)_{s,r}\otimes dW^{\eta}(r,\omega)
\end{equation}
 as a smooth second order process, it is well defined as Riemann-Stieljes integral. In addition, we constructed the convergence between  $\boldsymbol{W}^{\eta}(\omega)=(W^{\eta}(\cdot,\omega),\mathbb{W}^{\eta}(\omega))$ and $\boldsymbol{W}(\omega)=(W(\omega),\mathbb{W}(\omega))$ in \cite{caogao}.
 \begin{lemma}[Theorem 3.2,\cite{caogao}]\label{lemma 5.1}
Let $\boldsymbol{W}=(W,\mathbb{W})$ be the canonical lift of the fractional Brownian motion and $\boldsymbol{W}^{\eta}=\left(W^{\eta},\mathbb{W}^{\eta}\right)$ as the approximation of $\boldsymbol{W}$. Then we have
$$\rho_{\alpha^{\prime},[-T,T]}(\boldsymbol{W},\boldsymbol{W}^{\eta})\rightarrow 0,\quad \text{as}\quad \eta\rightarrow 0$$
for any  $T>0$, $\alpha^{\prime}\in(\frac{1}{3},\frac{1}{2})$.  Furthermore, the convergence takes place for all $\omega$ in an $\theta$-invariant set $\Omega^{\prime}$ of full measure.
 \end{lemma}

\begin{corollary}\label{corollary 5.1*}
Let $\boldsymbol{W}=(W,\mathbb{W})$ be the canonical lift of the fractional Brownian motion and $\boldsymbol{W}^{\eta}=\left(W^{\eta},\mathbb{W}^{\eta}\right)$ as the approximation of $\boldsymbol{W}$. Then we have
$$\rho_{\alpha^{\prime}, [T_j(\boldsymbol{W}),T_{j+1}(\boldsymbol{W})]}(\boldsymbol{W},\boldsymbol{W}^{\eta})\rightarrow 0,\quad \text{as}\quad \eta\rightarrow 0$$
for any  $j\in\mathbb{Z}$, $\alpha^{\prime}\in(\frac{1}{3},\frac{1}{2})$.  Furthermore, the convergence takes place for all $\omega$ in an $\theta$-invariant set $\Omega^{\prime\prime}$ of full measure and the convergence is uniform for $j\in\mathbb{Z}$ .
 \end{corollary}
 \begin{proof}
 Firstly, similar to \cite[Theorem 3.1]{caogao} or \cite[Theorem 4.5]{MR4266219}, we can prove that $\left|\rho_{\alpha^\prime,[s,t]}(\boldsymbol{W},\boldsymbol{W}^{\eta})\right|_{L^{\frac{q^\prime}{2}}}\leq C(q^\prime,\alpha^{\prime\prime},H)\eta^{H-\alpha^{\prime\prime}}$ for any $[s,t]\subset[n,n+2]$, $n\in\mathbb{Z}$, where $\alpha^{\prime\prime}>\alpha^{\prime}$ and $q^\prime\geq 2$ such  that $\alpha^{\prime\prime}-\frac{1}{q^\prime}>\frac{1}{3}$. In particular,  compared with \cite{caogao}  and \cite{MR4266219}, the  constant $C(q^\prime,\alpha^{\prime\prime},H)$ is uniform  for  each $n\in\mathbb{Z}$. Secondly, similar to \cite[Theorem 3.2]{caogao} or \cite[Theorem 4.6]{MR4266219},  we have a sequence $(\eta_i)_{i\in\mathbb{N}}$ that converges sufficiently fast to zero as $i\rightarrow\infty$ such  that $\lim_{i\rightarrow\infty}\rho_{\alpha^\prime,[n,n+2]}(\boldsymbol{W},\boldsymbol{W}^{\eta_i})=0$  for every $n\in\mathbb{Z}$ and $\alpha^\prime\in\left(\frac{1}{3},\frac{1}{2}\right)$ and these convergence take place almost surely, in an $\theta$-invariant set $\Omega^{\prime\prime}$ of full measure.  Furthermore, since we obtain the uniform estimate of  $\left|\rho_{\alpha^\prime,[s,t]}(\boldsymbol{W},\boldsymbol{W}^{\eta})\right|_{L^{\frac{q^\prime}{2}}}$, then  the convergence is uniform for $n\in\mathbb{Z}$. Finally,  for any $\omega\in\Omega^{\prime\prime}$ and $[T_j(\boldsymbol{W}(\omega)),T_{j+1}(\boldsymbol{W}(\omega))],j\in\mathbb{Z}$, there exists a $n_j\in\mathbb{Z}$ such that $[T_j(\boldsymbol{W}(\omega)),T_{j+1}(\boldsymbol{W}(\omega))]\subset [n(j),n(j)+2]$,  similar to the pathwise analysis of \cite[Theorem 3.2]{caogao}, we can prove that $\lim_{\eta\rightarrow 0}\rho_{\alpha^\prime,[T_j(\boldsymbol{W}(\omega)),T_{j+1}(\boldsymbol{W}(\omega))]}(\boldsymbol{W}^\eta,\boldsymbol{W}^{\eta_i})=0$,  and  based on  the  property of stopping times, i.e. $|T_{j+1}(\boldsymbol{W}(\omega))-T_j(\boldsymbol{W}(\omega))|\leq 1$ and $\|\boldsymbol{W}(\omega)\|_{\alpha^\prime,[T_j(\boldsymbol{W}(\omega)),T_{j+1}(\boldsymbol{W}(\omega))]}\leq\mu\leq 1$,  this convergence is also uniform for $j\in\mathbb{Z}$.  Thus, for any $\omega\in\Omega^{\prime\prime} $, there exists a  subsequence $\{\eta_i\}_{i\in\mathbb{N}}\subset(0,1)$ as before,  such that
 \begin{align*}
  \rho_{\alpha^\prime,[T_j(\boldsymbol{W}(\omega)),T_{j+1}(\boldsymbol{W}(\omega))]}(\boldsymbol{W}^\eta(\omega),\boldsymbol{W}(\omega))&\leq \rho_{\alpha^\prime,[T_j(\boldsymbol{W}(\omega)),T_{j+1}(\boldsymbol{W}(\omega))]}(\boldsymbol{W}^\eta(\omega),\boldsymbol{W}^{\eta_i}(\omega))\\
  &~~+\rho_{\alpha^\prime,[T_j(\boldsymbol{W}(\omega)),T_{j+1}(\boldsymbol{W}(\omega))]}(\boldsymbol{W}^{\eta_i}(\omega),\boldsymbol{W}(\omega)).
  \end{align*}
Due to $[T_j(\boldsymbol{W}(\omega)),T_{j+1}(\boldsymbol{W}(\omega))]\subset [n(j),n(j)+2]$,  we complete the proof.
 \end{proof}
 \subsection{Wong-Zakai approximation of the stopping times}
 \begin{lemma}\label{lemma 5.2}
 Assume  the following conditions hold
 \begin{align*}
 \|\boldsymbol{W}^{\eta}(\omega)\|_{\alpha,[0,T(\boldsymbol{W}^{\eta}(\omega))]}+\mu(T(\boldsymbol{W}^{\eta}(\omega)))^{1-\alpha}&=\mu,\\
 \|\boldsymbol{W}^{\eta}(\omega)\|_{\alpha,[\hat{T}(\boldsymbol{W}^{\eta}(\omega)),0]}+\mu|\hat{T}(\boldsymbol{W}^{\eta}(\omega))|^{1-\alpha}&=\mu
 \end{align*}
 for all $\omega\in\Omega$, then we have
 \begin{align*}
 \lim_{\eta\rightarrow 0}T(\boldsymbol{W}^{\eta}(\omega))=T(\boldsymbol{W}(\omega)),\quad \lim_{\eta\rightarrow 0}\hat{T}(\boldsymbol{W}^{\eta}(\omega))=\hat{T}(\boldsymbol{W}(\omega)).
 \end{align*}
 \end{lemma}
 \begin{proof}
 Suppose  $ \lim_{\eta\rightarrow 0}T(\boldsymbol{W}^{\eta}(\omega))\neq T(\boldsymbol{W}(\omega)),\omega\in\Omega$. In view of $|T(\boldsymbol{W}^{\eta}(\omega))|\leq 1$, then there exists a sequence $\{\eta_{n}\}_{n\in\mathbb{N}}$ such that $\eta_{n}\rightarrow 0,n\rightarrow \infty$ and $$\lim_{n\rightarrow \infty}T(\boldsymbol{W}^{\eta_n}(\omega))=\tau\neq T(\boldsymbol{W}(\omega)).$$ By Lemma \ref{lemma 5.1}, we have
 $$\lim_{n\rightarrow \infty}\|\boldsymbol{W}^{\eta_n}(\omega)\|_{\alpha,[0,\tau]}=\|\boldsymbol{W}(\omega)\|_{\alpha,[0,\tau]}. $$
 Then for any $\epsilon>0$, there is an integer $N_{1}(\tau,\epsilon)$ such that for $n>N_{1}(\tau,\epsilon)$ we have
 $$\left|\|\boldsymbol{W}^{\eta_{n}}(\omega)\|_{\alpha,[0,\tau]}-\|\boldsymbol{W}(\omega)\|_{\alpha,[0,\tau]}\right|\leq \frac{\epsilon}{2}.$$
  According to the proof of the Lemma \ref{lemma 3.2}, we know that the mapping $t\mapsto \|\cdot\|_{\alpha,[0,t]}$ is continuous, then we have
  $$\left|\|\boldsymbol{W}^{\eta_{n}}(\omega)\|_{\alpha,[0,T(\boldsymbol{W}^{\eta_n}(\omega))]}-\|\boldsymbol{W}^{\eta_{n}}(\omega)\|_{\alpha,[0,\tau]}\right|\leq C|T(\boldsymbol{W}^{\eta_{n}}(\omega))-\tau|.$$
  Then for any $\epsilon>0$, there is an integer $N_{2}(\tau,\epsilon)$ such that for $n>N_{2}(\tau,\epsilon)$ we have
  $$|T(\boldsymbol{W}^{\eta_n}(\omega))-\tau|\leq \frac{\epsilon}{2C}.$$
  Let $N=\max\{N_{1}(\tau,\epsilon),N_{2}(\tau,\epsilon)\}$, thus for $n>N$ we obtain
  \begin{align*}
  \left|\|\boldsymbol{W}^{\eta_{n}}(\omega)\|_{\alpha,[0,T(\boldsymbol{W}^{\eta}(\omega))]}-\|\boldsymbol{W}(\omega)\|_{\alpha,[0,\tau]}\right|&\leq \left|\|\boldsymbol{W}^{\eta_{n}}(\omega)\|_{\alpha,[0,T(\boldsymbol{W}^{\eta_n}(\omega))]}-\|\boldsymbol{W}^{\eta_{n}}(\omega)\|_{\alpha,[0,\tau]}\right|\\
  &~~+\left|\|\boldsymbol{W}^{\eta_{n}}(\omega)\|_{\alpha,[0,\tau]}-\|\boldsymbol{W}(\omega)\|_{\alpha,[0,\tau]}\right|\leq \epsilon.
  \end{align*}
  Hence,
  $$\mu=\lim_{n\rightarrow \infty}\left(\|\boldsymbol{W}^{\eta_{n}}(\omega)\|_{\alpha,0,T(\boldsymbol{W}^{\eta_{n}}(\omega))}+\mu(T(\boldsymbol{W}^{\eta_{n}}(\omega)))^{1-\alpha}\right)=\|\boldsymbol{W}(\omega)\|_{\alpha,[0,\tau]}+\mu\tau^{1-\alpha}.$$
 Since the function $t\mapsto \|\cdot\|_{\alpha,0,t}$ is continuously increasing, it  obviously contradicts the definition of stopping times.
 \end{proof}
 \begin{remark}\label{remark 5.1*}
 Similar to this lemma, we can prove
 $\lim_{\eta\rightarrow 0}T(\theta_{T(\boldsymbol{W}^{\eta}W(\omega))}\boldsymbol{W}^{\eta}(\omega))=T(\theta_{T(\boldsymbol{W}(\omega))}\boldsymbol{W}(\omega)),$ and based on this inequality and Lemma \ref{lemma 5.2}, by induction we obtain
 $$\lim_{\eta\rightarrow 0} T_{i}(\theta_{T_{j}(\boldsymbol{W}^{\eta}(\omega))}\boldsymbol{W}^{\eta}(\omega))=T_{i}(\theta_{T_{j}(\boldsymbol{W}(\omega))}\boldsymbol{W}(\omega)),~ i,j\in\mathbb{Z}.$$
 \end{remark}
 \subsection{Wong-Zakai approximation of the solution}
 As we described in the introduction,  it is not difficult to prove the existence and uniqueness of the approximate system
  \begin{equation}\label{5.2}
dy^{\eta}=Ay^{\eta}dt+F(y^{\eta})dt+G(y^{\eta})d\boldsymbol{W}^{\eta},\quad y^{\eta}(0)=y^{\eta}_{0}\in \mathcal{B}.
\end{equation}
Indeed, applying the method of reference \cite{hesse2021global} to construct the global solution for  a given  Gubinelli derivative $G(y^{\eta})$. We have  the following theorem:
\begin{theorem}
Let $T>0$, $\boldsymbol{W}^{\eta}$  be an approximation of $\alpha$-H\"{o}lder GFBRP $\boldsymbol{W}$(see Section 5.1)  and  $y^{\eta}_{0}\in \mathcal{B}_{\gamma},\gamma\geq 0$.  Under  the assumptions \textbf{(A)},\textbf{(F)},\textbf{(G)}, then there exists a unique global in time solution  $(y^{\eta},G(y^{\eta}))\in \mathcal{D}_{W_{\eta}(\cdot,X),\gamma}^{2\alpha}$ of \eqref{5.2} for a given Gubinelli derivative $G(y^{\eta})$.
\end{theorem}
In order to compare the distance of controlled rough paths which are controlled by different stochastic processes, we define the following metric.
\begin{definition}\label{definition 5.1}
Let $\frac{1}{2}>\alpha>\frac{1}{3}$ and  $\boldsymbol{W},\tilde{\boldsymbol{W}}\in\mathcal{C}^{\alpha}([0,T];\mathbb{R}^{d}), I\subset [0,T]$. For any $(y,y^{\prime})\in \mathcal{D}_{W,\gamma}^{2\alpha}([0,T])$ and  $(z,z^{\prime})\in \mathcal{D}_{\tilde{W},\gamma}^{2\alpha}([0,T])$, we define the distance $d_{2\alpha,\gamma,I}(y,z)$ as follows
\begin{align*}
d_{2\alpha,\gamma,I}(y,z)=&\|y-z\|_{\infty,\gamma,I}+\|y^{\prime}-z^{\prime}\|_{\infty,\gamma-\alpha,I}+\interleave y^{\prime}-z^{\prime}\interleave_{\alpha,\gamma-2\alpha,I}\\
&+\interleave R^{y}-R^{z}\interleave_{\alpha,\gamma-\alpha,\Delta_{I}}+\interleave R^{y}-R^{z}\interleave_{2\alpha,\gamma-2\alpha,\Delta_{I}}.
\end{align*}
\end{definition}
Our next work is to construct the Wong-Zakai approximation of the  solution on stopping time intervals. To this end, we need  some lemmas.
\begin{lemma}\label{lemma 5.3}
Let $\boldsymbol{W},\tilde{\boldsymbol{W}}\in \mathcal{C}^{\alpha}([0,T],\mathbb{R}^{d})$ for $\frac{1}{2}>\alpha>\frac{1}{3}$,  and $(y,y^{\prime})\in\mathcal{D}_{W,\gamma}^{2\alpha},(z,z^{\prime})\in\mathcal{D}_{\tilde{W},\gamma}^{2\alpha}$. Then
\begin{small}
$$\int_{0}^{t}S(t\!-\!s)y_{s}d\boldsymbol{W}_{s}\!-\!\int_{0}^{t}S(t\!-\!s)z_{s}d\tilde{\boldsymbol{W}}_{s}\!=\!\lim_{|\mathcal{P}(0,t)|\rightarrow 0}\sum_{[u,v]\in \mathcal{P}}\!S(t\!-\!u)(y_{u}W_{u,v}+y^{\prime}_{u}\mathbb{W}_{u,v}-z_{u}\tilde{W}_{u,v}-z^{\prime}_{u}\tilde{\mathbb{W}}_{u,v})$$
\end{small}
exists in $\mathcal{B}_{\gamma-2\alpha}$. Moreover, for $0\leq \beta<3\alpha$ the above integral has the following estimate
\begin{small}
\begin{align*}
&\left\|\int_{s}^{t}S(t-r)y_{r}d\boldsymbol{W}_{r}\!-\!\int_{s}^{t}S(t-r)z_{r}d\tilde{\boldsymbol{W}}_{r}\!-\!S(t-s)(y_{s}W_{s,t}\!+\!y^{\prime}_{s}\mathbb{W}_{s,t}\!-\!z_{s}\tilde{W}_{s,t}-z^{\prime}_{s}\tilde{\mathbb{W}}_{s,t})\right\|_{\gamma-2\alpha+\beta}\\
&~~~~~~~~~\leq \Bigg(d_{\alpha,[s,t]}(\boldsymbol{W},\tilde{\boldsymbol{W}})\|y,y^{\prime}\|_{W,2\alpha,\gamma,[s,t]}
+d_{2\alpha,\gamma,[s,t]}(y,z)d_{\alpha,[s,t]}(\tilde{\boldsymbol{W}},0)\Bigg)(t-s)^{3\alpha-\beta}.
\end{align*}
\end{small}
\end{lemma}
\begin{proof}
The basic idea of the proof  is multiplicative Sewing Lemma\cite[Theorem 4.1]{MR4299812}.  Let $\xi_{s,t}=y_{s}W_{s,t}+y^{\prime}_{s}\mathbb{W}_{s,t}-z_{s}\tilde{W}_{s,t}-z^{\prime}_{s}\tilde{\mathbb{W}}_{s,t}$, we need to show that $\xi\in \mathcal{Z}_{\gamma}^{\alpha}$, where the space $\mathcal{Z}_{\gamma}^{\alpha}$ consists of double index element $\xi=(\xi_{s,t})\in C_{2}^{\alpha}(\mathcal{B}_{\gamma})+C_{2}^{2\alpha}(\mathcal{B}_{\gamma-\alpha})$ with the property $\delta\xi\in C_{3}^{2\alpha,\alpha}(\mathcal{B}_{\gamma-2\alpha})+C_{3}^{\alpha,2\alpha}(\mathcal{B}_{\gamma-2\alpha})$. That is to say, there exists $\xi^{1},\xi^{2}$ and $h^{1},h^{2}$ with
\begin{align*}
\xi_{s,t}=\xi^{1}_{s,t}+\xi^{2}_{s,t},\quad (s,t)\in\triangle_{2},\\
\delta\xi_{s,u,t}=h^{1}_{s,u,t}+h^{2}_{s,u,t},\quad (s,u,t)\in \triangle_{3},
\end{align*}
such that $\interleave \xi^{1}\interleave_{\alpha,\gamma}+\interleave \xi^{2}\interleave_{2\alpha,\gamma-\alpha}+\interleave h^{1}\interleave_{2\alpha,\alpha,\gamma-2\alpha}+\interleave h^{2}\interleave_{\alpha,2\alpha,\gamma-2\alpha}<\infty$, and $\mathcal{Z}_{\gamma}^{\alpha}$ is equipped with the norm
$$\|\xi\|_{\mathcal{Z}_{\gamma}^{\alpha}} =\inf_{\xi^{1},\xi^{2},h^{1},h^{2}}(\interleave \xi^{1}\interleave_{\alpha,\gamma}+\interleave \xi^{2}\interleave_{2\alpha,\gamma-\alpha}+\interleave h^{1}\interleave_{2\alpha,\alpha,\gamma-2\alpha}+\interleave h^{2}\interleave_{\alpha,2\alpha,\gamma-2\alpha}).$$
It is clearly that $\xi\in C_{2}^{\alpha}(\mathcal{B}_{\gamma})+C_{2}^{2\alpha}(\mathcal{B}_{\gamma-\alpha})$ and
\begin{align*}\interleave \xi\interleave_{C_{2}^{\alpha}(\mathcal{B}_{\gamma})+C_{2}^{2\alpha}(\mathcal{B}_{\gamma-\alpha})}\leq &\left(\|y\|_{\infty,\gamma}\interleave W-\tilde{W}\interleave_{\alpha}+\|y-z\|_{\infty,\gamma}\interleave \tilde{W}\interleave_{\alpha}\right) \\
+&\left(\|y^{\prime}\|_{\infty,\gamma-\alpha}\interleave \mathbb{W}-\tilde{\mathbb{W}}\interleave_{2\alpha}+\|y^{\prime}-z^{\prime}\|_{\infty,\gamma-\alpha}\interleave \tilde{\mathbb{W}}\interleave_{2\alpha}  \right).
\end{align*}
Furthermore, In view of Chen's identity we have that
$$\delta\xi_{s,u,t}=\left(-y_{s,u}^{\prime}\mathbb{W}_{u,t}+z_{s,u}^{\prime}\tilde{\mathbb{W}}_{u,t}\right)+\left(-R^{y}_{s,u}W_{u,t}+R^{z}_{s,u}\tilde{W}_{u,t}\right).$$
Invoking this identity we obtain
\begin{small}
\begin{align*}
\|\delta\xi_{s,u,t}\|&\leq \interleave\mathbb{W}\!-\!\tilde{\mathbb{W}} \interleave_{2\alpha}\interleave y^{\prime}\interleave_{\alpha,\gamma-2\alpha}(t-u)^{2\alpha}(u-s)^{\alpha}+\interleave y^{\prime}-z^{\prime}\interleave_{\alpha,\gamma-2\alpha}\interleave\tilde{\mathbb{W}} \interleave_{2\alpha}(t\!-\!u)^{2\alpha}(u\!-\!s)^{\alpha} \\
&+\interleave R^{y}\interleave_{2\alpha}\interleave W-\tilde{W} \interleave_{\alpha}(u-s)^{2\alpha}(t-u)^{\alpha}\!+\!\interleave R^{y}\!-\!R^{z}\interleave_{2\alpha}\interleave \tilde{W}\interleave_{\alpha}(t\!-\!u)^{\alpha}(u\!-\!s)^{2\alpha}.
\end{align*}
\end{small}
Then $\delta\xi\in C_{2}^{\alpha,2\alpha}(\mathcal{B}_{\gamma-2\alpha})+C_{2}^{2\alpha,\alpha}(\mathcal{B}_{\gamma-2\alpha})$ and
\begin{small}
\begin{align*}
\interleave \delta\xi\interleave_{C_{2}^{\alpha,2\alpha}(\mathcal{B}_{\gamma-2\alpha})+C_{2}^{2\alpha,\alpha}(\mathcal{B}_{\gamma-2\alpha})}&\leq \interleave \mathbb{W}-\tilde{\mathbb{W}}\interleave_{2\alpha}\interleave y^{\prime}\interleave_{\alpha,\gamma-2\alpha}+\interleave y^{\prime}-z^{\prime} \interleave_{\alpha,\gamma-2\alpha}\interleave \tilde{\mathbb{W}}\interleave_{2\alpha}\\
&~~~+\interleave R^{y}\interleave_{2\alpha,\gamma-\alpha}\interleave W-\tilde{W}\interleave_{\alpha}+  \interleave R^{y}-R^{z}\interleave_{2\alpha,\gamma-2\alpha}
\interleave\tilde{W}\interleave_{\alpha}.
\end{align*}
\end{small}
Hence, we have
$$\|\xi\|_{\mathcal{Z}_{\gamma}^{\alpha}}\leq \|y,y^{\prime}\|_{W,2\alpha,\gamma}d_{\alpha}(\boldsymbol{W},\tilde{\boldsymbol{W}})+d_{2\alpha,\gamma}(y,z)d_{\alpha}(\tilde{\boldsymbol{W}},0).$$
Finally, the multiplicative Sewing Lemma\cite[Theorem 4.1]{MR4299812} shows that
\begin{small}
\begin{align*}
&\left\|\int_{s}^{t}S(t\!-\!s)y_{s}d\boldsymbol{W}_{s}\!-\!\int_{s}^{t}S(t\!-\!s)z_{s}d\tilde{\boldsymbol{W}}_{s}\!-\!S(t\!-\!s)(y_{s}W_{s,t}\!+\!y^{\prime}_{s}\mathbb{W}_{s,t}\!-\!z_{s}\tilde{W}_{s,t}\!+\!z^{\prime}_{s}\tilde{\mathbb{W}}_{s,t})\right\|_{\gamma-2\alpha+\beta}\\
&\leq \Bigg(d_{\alpha,[s,t]}(\boldsymbol{W},\tilde{\boldsymbol{W}})\|y,y^{\prime}\|_{W,2\alpha,\gamma,[s,t]}
+d_{2\alpha,\gamma,[s,t]}(y,z)d_{\alpha,[s,t]}(\tilde{\boldsymbol{W}},0)\Bigg)(t-s)^{3\alpha-\beta}.
\end{align*}
\end{small}
\end{proof}
\begin{lemma}\label{lemma 5.4}
Let $(y,y^{\prime})\in\mathcal{D}_{W,\gamma}^{2\alpha}([0,T]), (z,z^{\prime})\in\mathcal{D}_{\tilde{W},\gamma}^{2\alpha}([0,T])$ and $\boldsymbol{W},\tilde{\boldsymbol{W}}\in \mathcal{C}^{\alpha}([0,T],\mathbb{R}^{d})$ for $\frac{1}{2}>\alpha>\frac{1}{3}$. Then controlled rough paths $(G(y),DG(y)G(y))\in \mathcal{D}_{W,\gamma}^{2\alpha}([0,T])$ and $(G(z),DG(z)G(z))\in \mathcal{D}_{\tilde{W},\gamma}^{2\alpha}([0,T])$ on each stopping time interval $I_{m}:=[T_{m-1}(\boldsymbol{W}),T_{m}(\boldsymbol{W})]\\(1\leq m\leq i-1,i\geq 1)$  have the following estimate
\begin{small}
 \begin{align*}
d_{2\alpha,\gamma-\sigma,I_{m}}(G(y),G(z))&\leq C\mu d_{2\alpha,\gamma,I_{m}}(y,z)(\|z,z^{\prime}\|_{\tilde{W},2\alpha,\gamma,I_{m}}+\|y,y^{\prime}\|_{W,2\alpha,\gamma,I_{m}}+1)^{2}\nonumber\\
&~~~~~~~~~~~~~~~~~\times(\interleave W\interleave_{\alpha,I_{m}}+\interleave \tilde{W}\interleave_{\alpha,I_{m}}+1)\nonumber\\
&+C\mu(\|z,z^{\prime}\|_{\tilde{W},2\alpha,\gamma,I_{m}}+\|y,y^{\prime}\|_{W,2\alpha,\gamma,I_{m}}+1)^{2}d_{\alpha,I_{m}}(\boldsymbol{W},\tilde{\boldsymbol{W}})\nonumber\\
&~~~~~~~~~~~~~~~~~\times(\interleave W\interleave_{\alpha,I_{m}}+\interleave \tilde{W}\interleave_{\alpha,I_{m}}+1).
 \end{align*}
 \end{small}
\end{lemma}
The proof of this lemma is too lengthy, we omit it here, it can be found in the Appendix \ref{Appendix B}.
\begin{lemma}
Let $(y^{\eta},G(y^{\eta}))$ be the solution of \eqref{5.2}  driven by $\theta_{-T}\boldsymbol{W}^{\eta}$ with initial data $y_{0}^{\eta}\in\mathcal{B}$. Then  the norm of  controlled rough path
$(y^{\eta},G(y^{\eta}))$ on stopping times interval $I_{i}(T):=[T_{i-1}(\theta_{-T}\boldsymbol{W}),T_{i}(\theta_{-T}\boldsymbol{W})], i\in\mathbb{Z}^{+},T>0$ is bounded for sufficiently small $\eta$.
\end{lemma}
\begin{proof}
By  the Corollary \ref{corollary 5.1*} and \eqref{3.58},  for any $\left(\frac{1}{C}\bigwedge \frac{e^{\frac{d\lambda}{2}-\nu}-1}{Ce^{\frac{d\lambda}{2}-\nu}}\right)-\mu>\epsilon>0$,  there  exists  a $\eta^{\prime}(\epsilon)$ sufficiently small,  such that
\begin{align*}
\left|\interleave \theta_{-T}\boldsymbol{W}^{\eta} (\omega)\interleave_{\alpha, I_{i}(T)}-\interleave \theta_{-T}\boldsymbol{W}(\omega)\interleave_{\alpha, I_{i}(T)}\right|\leq \epsilon, \quad \eta<\eta^{\prime},
\end{align*}
then
\begin{align}
\interleave \theta_{-T}\boldsymbol{W}^{\eta}(\omega)\interleave_{\alpha, I_{i}(T)}\leq \mu+\epsilon.
\end{align}
Similar to Lemma \ref{lemma 3.4}, we have
\begin{align*}
&\|y^\eta,(y^{\eta})^{\prime}\|_{W,2\alpha,0,[T_{i-1}(\theta_{-T}\boldsymbol{W}(\omega)),T_{i}(\theta_{-T}\boldsymbol{W}(\omega))]}\\
&~~~~~~~~~~~\leq C(\mu+\epsilon)\sum_{m=1}^{i-1}\bigg(e^{-\lambda(T_{i-1}(\theta_{-T}\boldsymbol{W}(\omega))-T_{m}(\theta_{-T}\boldsymbol{W}(\omega)))}\\
&~~~~~~~~~~~~~\times(1+\|y^{\eta},(y^{\eta})^{\prime}\|_{W,2\alpha,0,[T_{m-1}(\theta_{T_{j}}\boldsymbol{W}(\omega)),T_{m}(\theta_{T_{j}}\boldsymbol{W}(\omega))]})\bigg)\nonumber\\
&~~~~~~~~~~~~~+C(\mu+\epsilon)\left(1+\|y,(y^{\eta})^{\prime}\|_{W,2\alpha,0,[T_{i-1}(\theta_{-T}\boldsymbol{W}(\omega)),T_{i}(\theta_{-T}\boldsymbol{W}(\omega))}]\right)\\
&~~~~~~~~~~~~~+Ce^{-\lambda T_{i-1}(\theta_{-T}\boldsymbol{W}(\omega))}\|y^{\eta}_{0}\|.
\end{align*}
Furthermore, analogy with Corollary \ref{corollary 3.1} and Lemma \ref{lemma 3.6}, we obtain
\begin{align}\label{5.3**}
\|y^{\eta},(y^{\eta})^{\prime}\|_{W,2\alpha,0,[T_{i-1}(\theta_{-T}\boldsymbol{W}(\omega)),T_{i}(\theta_{-T}\boldsymbol{W}(\omega))]}\leq R(0,\theta_{-T}\omega)+o(\epsilon).
\end{align}
Due to  $R(0,\theta_{-T}\omega)$ is bounded, so   $\|y^{\eta},(y^{\eta})^{\prime}\|_{W,2\alpha,\gamma,[T_{i-1}(\theta_{-T}\boldsymbol{W}(\omega)),T_{i}(\theta_{-T}\boldsymbol{W}(\omega))]}$ is bounded.
\end{proof}
Based on the above results, we could  construct the Wong-Zakai approximation of the solution on some stopping time intervals. Now, we consider the following mild solutions:
\begin{align}\label{5.22}
  y_{t}=S(t)y_{0}+\int_{0}^{t}S(t-r)F(y_{r})dr+\int_{0}^{t}S(t-r)G(y_{r})d\theta_{-T}\boldsymbol{W}_{r}
\end{align}
and
\begin{align}\label{5.23}
  y^{\eta}_{t}=S(t)y^{\eta}_{0}+\int_{0}^{t}S(t-r)F(y^{\eta}_{r})dr+\int_{0}^{t}S(t-r)G(y^{\eta}_{r})d\theta_{-T}\boldsymbol{W}^{\eta}_{r}.
\end{align}
\begin{theorem}\label{theorem 5.2}
Let $(y,y^{\prime})$ be the solution of equation \eqref{5.22} with initial data $y_{0}\in \mathcal{B}$ and $(y^{\eta},(y^{\eta})^{\prime})$ be the solution of equation \eqref{5.23} with initial data $y^{\eta}_{0}\in \mathcal{B}$, where  a sequence  of stopping times  $(T_{i}(\theta_{-T}\boldsymbol{W}(\omega)))$ is defined above. If $\|y_{0}^{\eta}-y_{0}\|\rightarrow 0$ as $\eta \rightarrow 0$, then
$$d_{2\alpha,0,I_{i}(T)}(y,y^{\eta})\rightarrow 0,\quad \forall i\in\mathbb{Z}^{+}, T>0.$$
Here, for fixed $\omega\in\Omega$, the convergence is uniform for each $i$.
\end{theorem}
Its proof can be found in Appendix \ref{Appendix C}.
\subsection{The upper semi-continuity of random attractor}
In this subsection, we shall construct the upper semi-continuity for the random attractor $\mathcal{A}_{\eta}$  of the approximate system $\varphi_\eta$ which is generated by the solutions $y^{\eta}$. Since the approximated noise can be regarded  as a smooth rough path. For the metric dynamical system in section 4, $\boldsymbol{W}^\eta$  has the same properties with $\boldsymbol{W}$, so those lemmas for $\boldsymbol{W}$  also hold for $\boldsymbol{W}^\eta$.    Then the existence of $\mathcal{A}_{\eta}$ can be obtained as  $\mathcal{A}$, we do not claim it any more. For the upper semicontinuity of attractor  $\mathcal{A}_{\eta}$, we refer to \cite{MR3270946,MR4266114}.
\begin{theorem}
For any $\omega\in\Omega$, we have
$$\lim_{\eta\rightarrow 0}dist(\mathcal{A}_{\eta},\mathcal{A})=0.$$
\end{theorem}
\begin{proof}
For any $\{x_{\eta_n}\}_{n\in\mathbb{Z}^{+}}\in\mathcal{B}$, and $x_{\eta_n}\rightarrow x\in\mathcal{B}, t\geq 0$, we have that
\begin{align*}
 &\|\varphi_{\eta_n}(t,\theta_{-t}\boldsymbol{W}^{\eta_n}(\omega),x_{\eta_n})-\varphi(t,\theta_{-t}\boldsymbol{W}(\omega),x)\|\nonumber\\
 &\leq \|S(t)(x-x_{\eta_n})\|+\left\|\int_{0}^{t}S(t-r)(F(y_{r})-F(y^{\eta_n}_{r}))dr\right\|\nonumber\\
 &~~+\left\|\int_{0}^{t}S(t-r)G(y_{r})d\theta_{-t}\boldsymbol{W}_{r}(\omega)-\int_{0}^{t}S(t-r)G(y^{\eta_n}_{r})d\theta_{-t}\boldsymbol{W}^{\eta_n}_{r}(\omega)\right\|\nonumber\\
 &\leq Ce^{-\lambda t}\|x-x^{\eta_n}\|+\left\|\int_{0}^{t}S(t-r)(F(y_{r})-F(y^{\eta_n}_{r}))dr\right\|\nonumber\\
 &~~+\left\|\int_{0}^{t}S(t-r)G(y_{r})d\theta_{-t}\boldsymbol{W}_{r}(\omega)-\int_{0}^{t}S(t-r)G(y^{\eta_{n}}_{r})d\theta_{-t}\boldsymbol{W}^{\eta_n}_{r}(\omega)\right\|.\nonumber
\end{align*}
We can always find $i\in \mathbb{Z}$ such that  $t\in I_{i}(t)$. Then by Theorem \ref{theorem 5.2} and Corollary \ref{corollary 5.1*} we have
\begin{align*}
&\|\varphi_{\eta}(t,\theta_{-t}\boldsymbol{W}^{\eta_n}(\omega),x_{\eta_n})-\varphi(t,\theta_{-t}\boldsymbol{W}(\omega),x)\|\\
&\leq Ce^{-\lambda t}\|x-x_{\eta_n}\|+\sum_{m=1}^{i-1}\left\|\int_{T_{m-1}(\theta_{-t}\boldsymbol{W}(\omega))}^{T_{m}(\theta_{-t}\boldsymbol{W}(\omega))}S(t-r)(F(y_{r})-F(y_{r}^{\eta_n}))d_{r}\right\|\nonumber\\
&+\left\|\int_{T_{i-1}(\theta_{-t}\boldsymbol{W}(\omega))}^{t}S(t-r)(F(y_{r})-F(y_{r}^{\delta_n}))d_{r}\right\|\\
&+\sum_{m=1}^{i-1}\left\|\int_{T_{m-1}(\theta_{-t}\boldsymbol{W}(\omega))}^{T_{m}(\theta_{-t}\boldsymbol{W}(\omega))}S(t-r)G(y_{r})d\theta_{-t}\boldsymbol{W}_{r}(\omega)-\int_{T_{m-1}(\theta_{-t}\boldsymbol{W}(\omega))}^{T_{m}(\theta_{-t}\boldsymbol{W}(\omega))}S(t-r)G(y^{\eta_n}_{r})d\theta_{-t}\boldsymbol{W}^{\eta_n}_{r}(\omega)\right\|\\
&+\left\|\int_{T_{m-1}(\theta_{-t}\boldsymbol{W}(\omega))}^{t}S(t-r)G(y_{r})d\theta_{-t}\boldsymbol{W}_{r}(\omega)-\int_{T_{m-1}(\theta_{-t}\boldsymbol{W}(\omega))}^{t}S(t-r)G(y^{\delta_n}_{r})d\theta_{-t}\boldsymbol{W}^{\eta_n}_{r}(\omega)\right\|\\
&\leq Ce^{-\lambda t}\|x-x_{\eta_n}\|+\sum_{m=1}^{i-1}C\mu e^{-\lambda(T_{i-1}-T_{m}(\theta_{-t}W(\omega)))}\bigg(d_{\alpha,I_{m}(t)}(\theta_{-t}\boldsymbol{W}(\omega),\theta_{-t}\boldsymbol{W}^{\eta_n}(\omega))\\
&~~~~~~~~~~~~~~~~~~~~~~~~+d_{2\alpha,0,I_{m}(t)}(y,y^{\eta_n})\bigg)\nonumber\\
& +C\mu \bigg(d_{\alpha,I_{i}(T)}(\theta_{-t}\boldsymbol{W}(\omega),\theta_{-t}\boldsymbol{W}^{\eta_n}(\omega))+d_{2\alpha,0,I_{i}(T)}(y,y^{\eta_n})\bigg)\rightarrow 0,\quad \eta_{n}\rightarrow 0.
\end{align*}
Thus, we obtain the convergence of  random dynamical systems $\varphi_{\eta}$.  We still need to illustrate that the random  absorbing sets $B_{\eta}(\omega)$  are uniform bounded and $\cup_{\eta\in(0,1]}\mathcal{A}_{\eta}$ is precompact.  Remark \ref{remark 4.1*} implies that $B_{\eta}(\omega)$ are also random  absorbing sets, and the version of  Lemma \ref{lemma 3.6} for $B_{\eta}$ and $d^{\eta}$ have a uniform lower bound  show that its radius is uniform bounded.  Due to  Lemma \ref{lemma 5.2} and Remark \ref{remark 5.1*}, we have
$$\lim_{\eta\rightarrow 0}R_{\eta}(0,\theta_{-t}\omega)=R(0,\theta_{-t}\omega).$$
Thus, $\cup_{\eta\in(0,1]}\mathcal{A}_{\eta}$ is precompact.
\end{proof}

\appendix
\renewcommand{\appendixname}{Appendix~\Alph{section}}
\section{Proof of Lemma 3.4}\label{appendix A}
\setcounter{equation}{0}

\renewcommand{\theequation}{A.\arabic{equation}}
In this subsection, for simplicity,  we omit $\omega$ in $W(\omega)$ and let $I_{i,j}=[T_{i-1}(\theta_{T_j}\boldsymbol{W}),\\T_{i}(\theta_{T_j}\boldsymbol{W})],i,j\in \mathbb{Z}$, we shall use the cocycle property of stopping times and the length $|I_{i,j}|$ of $I_{i,j},i,j\in \mathbb{Z}$ less than $1$ to prove the next lemmas.
\begin{lemma}\label{lemma A.1}
Let $y_0\in\mathcal{B}_\gamma,\gamma\geq 0$. Then $\left(S(\cdot)y_0,0\right)\in\mathcal{D}_{W,\gamma,I_{i,j}}^{2\alpha}$ and
\begin{equation}
\left\|S(\cdot)y_0,0\right\|_{W,2\alpha,\gamma,I_{i,j}}\leq Ce^{-\lambda T_{i-1}(\theta_{T_{i}}\boldsymbol{W})}\|y_{0}\|_{\gamma},
\end{equation}
where the  constant $C$ only depends on the semigroup  $S$.
\end{lemma}
\begin{proof}
Its proof is same as \cite[Lemma 3.2]{hesse2021global},  we can use the estimates \eqref{2.1*}-\eqref{2.2*} instead of \eqref{2.1}-\eqref{2.2} to prove this lemma. Thus, compared with \cite[Lemma 3.2]{hesse2021global}, $e^{-\lambda T_{i-1}(\theta_{T_{i}}\boldsymbol{W})}$ will emerge here. We omit the details here.
\end{proof}
\begin{lemma}\label{lemma A.2}
Let $\left(y,y^\prime\right)\in\mathcal{D}^{2\alpha}_{W,\gamma}$ solve the equation \eqref{3.2} on $[0,T]$. Then $\left(\int_{0}^{t}S(t-r)F(y_r)dr,0\right)_{t\in I_{i,j}}\\\in\mathcal{D}^{2\alpha}_{W,\gamma,I_{i,j}}$ and satisfies  the following bounds
\begin{align*}
\left\|\int_{0}^{\cdot}S(\cdot-r)F(y_r)dr,0\right\|_{W,2\alpha,\gamma,I_{i,j}}&\leq C\mu\sum_{m=1}^{i-1} e^{-\lambda(T_{i-1}(\theta_{T_{j}}\boldsymbol{W})-T_{m}(\theta_{T_{j}}\boldsymbol{W})) }(1+\|y\|_{\infty,\gamma,I_{m,j}})\\
 &~~~~+C\mu(1+\|y\|_{\infty,\gamma,I_{i,j}}).
\end{align*}
\end{lemma}
\begin{proof}
Since the Gubinelli derivative of the deterministic integral is zero, then
\begin{align*}
    \left\|\int_0^\cdot S(\cdot-r)F(y_r)dr,0\right\|_{W,2\alpha,\gamma}
   & = \left\|\int_0^\cdot S(\cdot-r)F(y_s)d r\right\|_{\infty,\gamma}
    + \left\| \int_0^\cdot S(\cdot-r)F(y_s)d r \right\|_{\alpha,\gamma-\alpha}\\
    &~~~~+ \left\| \int_0^\cdot S(\cdot-r)F(y_s)d r \right\|_{2\alpha,\gamma-2\alpha}.
\end{align*}
Thanks to the additivity of the deterministic integral, for $t\in I_{i,j}$ the following splitting holds true
\begin{align*}
\int_{0}^{t}S(t-r)F(y_r)dr=\sum_{m=1}^{i-1}\int_{I_{m,j}}S(t-r)F(y_r)dr+\int_{T_{i-1}(\theta_{T_{j}}\boldsymbol{W})}^tS(t-r)F(y_r)dr.
\end{align*}
Similar to \cite[Lemma 3.3]{hesse2021global}, by assumption \textbf{(F)}, $|I_{i,j}|\leq 1$ and \eqref{2.1*}-\eqref{2.2*} we have
\begin{equation}\label{A.3}
\left\|\int_{T_{i-1}(\theta_{T_{j}}\boldsymbol{W})}^{\cdot}S(\cdot-r)F(y_r)dr,0\right\|_{W,2\alpha,\gamma}\leq C\mu(1+\|y\|_{\infty}).
\end{equation}
For $s<t\in I_{i,j}$, let
\begin{align*}
 s^\prime:=s-T_{i-1}(\theta_{T_{j}}\boldsymbol{W}),\quad t^\prime:=t-T_{i-1}(\theta_{T_{j}}\boldsymbol{W}).
 \end{align*}
 Then, $s^\prime<t^\prime\in [0,T(\theta_{T_{i+j-1}}\boldsymbol{W})]$ and $T(\theta_{T_{i+j-1}}\boldsymbol{W})\leq 1$. Using the change of variable we obtain
 \begin{small}
 \begin{align*}
&\int_{I_{m,j}}S(t-r)F(y_r)dr=\int_{0}^{T(\theta_{T_{m+j-1}}\boldsymbol{W})}S(t-r-T_{m-1}(\theta_{T_j}\boldsymbol{W}))F(y_{r+T_{m-1}(\theta_{T_j}\boldsymbol{W})})dr\\
&=S(t^\prime+T_{i-1}(\theta_{T_j}\boldsymbol{W})-T_{m}(\theta_{T_j}\boldsymbol{W}))\int_{0}^{T(\theta_{T_{m+j-1}}\boldsymbol{W})}S(T(\theta_{T_{m+j-1}}\boldsymbol{W})-r)F(y_{r+T_{m-1}(\theta_{T_j}\boldsymbol{W})})dr.
 \end{align*}
 \end{small}
 Thus, by \eqref{2.1*}-\eqref{2.2*} and assumption \textbf{(F)} we obtain
\begin{align}\label{3.14}
&\sum_{m=1}^{i-1}\left\|\int_{T_{m-1}(\theta_{T_{j}}\boldsymbol{W})}^{T_{m}(\theta_{T_{j}}\boldsymbol{W})}S(t-r)F(y_{r})dr\right\|_{\gamma}\leq\sum_{m=1}^{i-1}\bigg(\left\|S( t^{\prime}\!+\!T_{i-1}(\theta_{T_{j}}\!\boldsymbol{W})\!-\!T_{m}(\theta_{T_{j}}\!\boldsymbol{W}))\right\|_{\mathcal{L}(\mathcal{B}_{\gamma})}\nonumber\\
&~~~~~~~~~~~~~~~~~~~~~ \times\left\|\int_{0}^{T(\theta_{T_{\!m\!+\!j\!-\!1}}\!\boldsymbol{W})}S(T(\theta_{T_{\!m\!+\!j\!-\!1}}\!\boldsymbol{W})\!-\!r)F(y_{\!r\!+\!T_{m\!-\!1}(\theta_{T_{j}}\!\boldsymbol{W})})dr \right\|_{\gamma}\bigg)\nonumber\\
&~~~~~~~~~~~~~~~~~~~~~\leq C\mu\sum_{m=1}^{i-1} e^{-\lambda(T_{i-1}(\theta_{T_{j}}\boldsymbol{W})-T_{m}(\theta_{T_{j}}\boldsymbol{W})) }(1+\|y\|_{\infty,\gamma,I_{m,j}}).
\end{align}
For $\theta\in\{\alpha,2\alpha\}, m\in \{1,\cdots,i-1\} $, by the change of variable, \eqref{2.1*}-\eqref{2.2*} and assumption \textbf{(F)} we have that
\begin{small}
\begin{align}\label{A.5}
&\left\|\int_{I_{m,j}}S(t-r)F(y_{r})dr-\int_{I_{m,j}}S(s-r)F(y_{r})dr\right\|_{\gamma-\theta}\nonumber\\
&~~=  \Bigg\|\left(S(t^{\prime}-s^{\prime})-Id\right)S(s^{\prime}+T_{i-1}(\theta_{T_{j}}\boldsymbol{W})-T_{m}(\theta_{T_{j}}\boldsymbol{W}))\nonumber\\
&~~~~~~~~~\times\int_{0}^{T(\theta_{T_{m+j-1}}\boldsymbol{W})}S(T(\theta_{T_{m+j-1}}\boldsymbol{W})-r)F(y_{r+T_{m-1}(\theta_{T_{j}}\boldsymbol{W})})dr \nonumber\Bigg\|_{\gamma-\theta}\nonumber\\
&~~\leq C|t-s|^{\theta}e^{-\lambda(T_{i-1}(\theta_{T_{j}}\boldsymbol{W})-T_{m}(\theta_{T_{j}}\boldsymbol{W}))}\nonumber\\
&~~~~~~~~~\times\Bigg\|
\int_{0}^{T(\theta_{T_{m+j-1}}\boldsymbol{W})}S(T(\theta_{T_{m+j-1}}\boldsymbol{W})-r)F(y_{r+T_{m-1}(\theta_{T_{j}}\boldsymbol{W})})dr \Bigg\|_{\gamma}\nonumber\\
&~~\leq C\mu|t-s|^{\theta}e^{-\lambda(T_{i-1}(\theta_{T_{j}}\boldsymbol{W})-T_{m}(\theta_{T_{j}}\boldsymbol{W}))}(1+\|y\|_{\infty,\gamma,[T_{m-1}(\theta_{T_{j}}\boldsymbol{W}),T_{m}(\theta_{T_{j}}\boldsymbol{W})]}).
\end{align}
\end{small}
Thus, \eqref{3.14}-\eqref{A.5} show that
\begin{align}\label{A.6}
&\left\|\int_{0}^{T_{i-1}(\theta_{T_j}(\boldsymbol{W})}S(t-r)F(y_{r})dr,0\right\|_{W,2\alpha,\gamma,I_{i,j}}\nonumber\\
&~~~~~~~~~~~~~~~~~~\leq C\mu\sum_{i=1}^{i-1}e^{-\lambda(T_{i-1}(\theta_{T_{j}}\boldsymbol{W})-T_{m}(\theta_{T_{j}}\boldsymbol{W})) }(1+\|y\|_{\infty,\gamma,I_{m,j}}).
\end{align}
Furthermore, combining with \eqref{A.3}, \eqref{A.6} we obtain
\begin{align}\label{3.15*}
 &\left\|\!\int_{0}^{\cdot}\!S(\cdot-r)F(y_{r})dr,0\right\|_{W,2\alpha,\gamma,I_{i,j}}\nonumber\\
 &~~\leq C\mu\sum_{m=1}^{i-1} e^{-\lambda(T_{i-1}(\theta_{T_{j}}\boldsymbol{W})-T_{m}(\theta_{T_{j}}\boldsymbol{W})) }(1+\|y\|_{\infty,\gamma,I_{m,j}})\nonumber\\
 &~~~~~~+C\mu(1+\|y\|_{\infty,\gamma,I_{i,j}}),
\end{align}
where the above constant $C>0$ only depends on $S$.
\end{proof}
\begin{remark}\label{R-A.1}
For  $\int_{T_{i-1}(\theta_{T_{j}}\boldsymbol{W})}^{\cdot}S(\cdot-r)F(y_r)dr$,  we do not require that the condition  $\delta\in [2\alpha,1)$, our estimate  is different from \cite{hesse2021global}. Indeed, for any $s<t\in I_{i,j}$ we have
\begin{align*}
&\int_{T_{i-1}(\theta_{T_{j}}\boldsymbol{W})}^{t}S(t-r)F(y_r)dr	-\int_{T_{i-1}(\theta_{T_{j}}\boldsymbol{W})}^{s}S(s-r)F(y_r)dr\nonumber\\
&~~~=\int_{s}^{t}S(t-r)F(y_r)dr+(S(t-s)-Id)\int_{T_{i-1}(\theta_{T_{j}}\boldsymbol{W})}^{s}S(s-r)F(y_r)dr,
\end{align*}
using this identity  for $\theta\in\{\alpha,2\alpha\}$, we obtain
\begin{align*}
	\left\|\int_{s}^{t}S(t-r)F(y_r)dr\right\|_{\gamma-\theta}&\leq C\mu\int_{s}^{t}(t-r)^{(\theta-\delta)\wedge 0}(1+\|y\|_{\infty,\gamma,I_{i,j}})dr\\
	&\leq C\mu (t-s)^{(1+\theta-\delta)\wedge 1}(1+\|y\|_{\infty,\gamma,I_{i,j}}),
\end{align*}
and
\begin{align*}
	&\left\|(S(t-s)-Id)\int_{T_{i-1}(\theta_{T_{j}}\boldsymbol{W})}^{s}S(s-r)F(y_r)dr\right\|_{\gamma-\theta}\\
	&~~~~\leq C(t-s)^{\theta}\left\|\int_{T_{i-1}(\theta_{T_{j}}\boldsymbol{W})}^{s}S(s-r)F(y_r)dr\right\|_{\gamma}\\
	&~~~~\leq C\mu(t-s)^{\theta}|I_{i,j}|^{1-\delta}(1+\|y\|_{\infty,\gamma,I_{i,j}})\leq  C\mu(t-s)^{\theta}(1+\|y\|_{\infty,\gamma,I_{i,j}}).
\end{align*}
Thus,
\begin{align*}
\left\|\int_{T_{i-1}(\theta_{T_{j}}\boldsymbol{W})}^{\cdot}S(\cdot-r)F(y_r)dr\right\|_{\gamma-\theta,I_{i,j}}&\leq C\mu(1+\|y\|_{\infty,\gamma,I_{i,j}})(1+|I_{i,j}|^{(1-\delta)\wedge(1-\theta)})\\
&\leq C\mu(1+\|y\|_{\infty,\gamma,I_{i,j}}).
\end{align*}
\end{remark}
\begin{lemma}\label{lemma A.3}
Let $G$ satisfy assumption $\textbf{(G)}$ and $(y,y^\prime)$ solves the equation \eqref{3.2} on $[0,T]$. Then $(G(y),DG(y)G(y))\in\mathcal{D}_{W,\gamma-\sigma, I_{i,j}}^{2\alpha}$ and have the following estimate
\begin{align*}
\|G(y),DG(y)G(y)\|_{W,2\alpha,\gamma-\sigma,I_{i,j}}&\leq C\mu(\rho_{\alpha,I_{i,j}}(\boldsymbol{W})+1)^{2}(\|y,y^\prime\|_{W,2\alpha,\gamma,I_{i,j}}+1)\\
&\leq C\mu(\|y,y^\prime\|_{W,2\alpha,\gamma,I_{i,j}}+1).
\end{align*}
\end{lemma}
\begin{proof}
Its proof is same as \cite[Lemma 3.5]{hesse2021global}, we need only to  explicitly write the noise $\boldsymbol{W}$ and $\mu$  into  the inequality in \cite[Lemma 3.5]{hesse2021global} and use  the assumption \textbf{(G)} and the norm of the rough noise less than $\mu\in(0,1)$ on stopping times interval $I_{i,j}$ to complete the proof.
\end{proof}
\begin{lemma}\label{lemma A.4}
For any $T>0$, let $(y,y^\prime)\in \mathcal{D}^{2\alpha}_{W,\gamma,[0,T]}$ and $z_\cdot=\int_{s}^{\cdot}S(\cdot-r)y_rd\boldsymbol{W}_r$, then
\begin{align*}
\left(\int_{0}^{\cdot}S(\cdot-r)y_rd\boldsymbol{W}_r,y\right)\in\mathcal{D}^{2\alpha}_{W,\gamma+\sigma,[0,T]}
\end{align*}
and have the following estimates
\begin{align*}
&\left|\int_{s}^{t}S(t-r)y_{r}d\boldsymbol{W}_r\right|_{\gamma+\sigma}\leq C(t-s)^{\alpha-\sigma}\|y,y^\prime\|_{W,2\alpha,\gamma,[s,t]} \rho_{\alpha,[s,t]}(\boldsymbol{W}),\\
\|z,z^\prime\|_{W,2\alpha,\gamma\!+\!\sigma,[s,t]}&\leq C(|y_s|_\gamma\!+\!|y^\prime_s|_{\gamma-\alpha}\rho_{\alpha,[s,t]}(\boldsymbol{W})\!+\!(t\!-\!s)^{\alpha\!-\!\sigma}(\rho_{\alpha,[s,t]}(\boldsymbol{W})\!+\!1)\|y,y^\prime\|_{W,2\alpha,\gamma,[s,t]})\\
&\leq C(1+(t-s)^{\alpha-\sigma})(\rho_{\alpha,[s,t]}(\boldsymbol{W})\!+\!1)\|y,y^\prime\|_{W,2\alpha,\gamma,[s,t]},
\end{align*}
for all $0\leq\sigma<\alpha$, $0\leq s<t\leq T$, where the above constant $C$ depends on $\alpha,\sigma$.
\end{lemma}
\begin{remark}
The proof of this lemma can be found in \cite[Lemma 3.4]{hesse2021global}, where the additivity of the rough integral has be used. In fact, since the semigroup $S$ is a continuous linear operator on $\mathcal{B}_{\gamma-2\alpha}$, according to  Lemma \ref{Lemma 2.1} we define
\begin{align*}
\int_{0}^{t_1}S(t-r)y_rd\boldsymbol{W}_r:&=\lim_{|\mathcal{P}(0,t_1)|\rightarrow0}\bigg[\sum_{[u,v]\in\mathcal{P}(0,t_1)}S(t-u)(y_uW_{u,v}+y^\prime_{u}\mathbb{W}_{u,v})\bigg]\\
&=S(t-t_1)\lim_{|\mathcal{P}(0,t_1)|\rightarrow0}\bigg[\sum_{[u,v]\in\mathcal{P}(0,t_1)}S(t_1-u)(y_uW_{u,v}+y^\prime_{u}\mathbb{W}_{u,v})\bigg]\\
&=S(t-t_1)\int_{0}^{t_1}S(t_1-r)y_rd\boldsymbol{W}_r
\end{align*}
for any $t_1\in(0,t)$.  Let $\mathcal{P}(0,t)=\mathcal{P}(0,t_1)\cup\mathcal{P}(t_1,t)$, then we have that
\begin{align*}
\int_{0}^{t}S(t-r)y_{r}d\boldsymbol{W}_r&=\lim_{|\mathcal{P}(0,t)|\rightarrow0}\sum_{[u,v]\in\mathcal{P}(0,t)}S(t-u)\left[y_{u}W_{u,v}+y^{\prime}_{u}\mathbb{W}_{u,v}\right]\\
&=\lim_{|\mathcal{P}(0,t_1)|\rightarrow0}\bigg[\sum_{[u,v]\in\mathcal{P}(0,t_1)}S(t-u)(y_uW_{u,v}+y^\prime_{u}\mathbb{W}_{u,v})\bigg]\\
&~~~~+\lim_{|\mathcal{P}(t_1,t)|\rightarrow0}\bigg[\sum_{[u,v]\in\mathcal{P}(t_1,t)}S(t-u)(y_uW_{u,v}+y^\prime_{u}\mathbb{W}_{u,v})\bigg]\\
&=\int_{0}^{t_1}S(t-r)y_rd\boldsymbol{W}_r+\int_{t_1}^{t}S(t-r)y_rd\boldsymbol{W}_r.
\end{align*}
\end{remark}
\begin{lemma}\label{lemma A.5}
Let $(y,y^\prime)$ solve the equation \eqref{3.2} on $[0,T]$. Then $\left(\int_{0}^{t}S(t-r)G(y_r)d\boldsymbol{W}_r,G(y)\right)_{t\in I_{i,j}}\\\in\mathcal{D}_{W,\gamma,I_{i,j}}^{2\alpha}$ and satisfy the following bound
\begin{align*}
\left\|\int_{0}^{\cdot}S(\cdot-r)G(y_r)d\theta_{T_j}\boldsymbol{W}_r,G(y)\right\|_{W,2\alpha,\gamma,I_{i,j}}&\!\leq\! C\mu\!\sum_{m=1}^{i-1}\! e^{-\lambda(T_{i-1}(\theta_{T_{j}}\boldsymbol{W})-T_{m}(\theta_{T_{j}}\boldsymbol{W})) }(1+\|y,y^\prime\|_{W,2\alpha,\gamma,I_{m,j}})\\
 &~~~~+C\mu(1+\|y,y^\prime\|_{W,2\alpha,\gamma,I_{i,j}}).
\end{align*}
\end{lemma}
\begin{proof}
Similar to Lemma \ref{lemma A.2}, for any $s<t\in I_{i,j}$, by the additivity of the rough integral, we have
\begin{small}
\begin{align}\label{A.7}
&\int_{0}^{t}S(t-r)G(y_{r})d\theta_{T_j}\boldsymbol{W}_{r}=\sum_{m=1}^{i-1}\int_{I_{m,j}}\!S(t\!-\!r)G(y_{r})d\theta_{T_j}\boldsymbol{W}_{r}+\int_{T_{i-1}(\theta_{T_{j}}\!\boldsymbol{W})}^{t}S(t\!-\!r)G(y_{r})d\theta_{T_j}\boldsymbol{W}_{r}\nonumber\\
&~~~~=\sum_{m=1}^{i-1}\int_{0}^{T(T_{m-1}(\theta_{T_{j}}\boldsymbol{W}))}S(t-r-T_{m-1}(\theta_{T_{j}}\boldsymbol{W}))G(y_{r+T_{m-1}(\theta_{T_{j}}\boldsymbol{W})})d\theta_{T_{m-1}(\boldsymbol{W})}\boldsymbol{W}_{r}\nonumber\\
&~~~~~~~~~~+\int_{T_{i-1}(\theta_{T_{j}}\boldsymbol{W})}^{t}S(t-r)G(y_{r})d\theta_{T_j}\boldsymbol{W}_{r}\nonumber\\
&~~~~=\sum_{m=1}^{i-1}\bigg[S(t^{\prime}\!+\!T_{i-1}(\theta_{T_{j}}\!\boldsymbol{W})\!-\!T_{m}(\theta_{T_{j}}\!\boldsymbol{W}))\int_{0}^{T(T_{m-1}(\theta_{T_{j}}\boldsymbol{W}))}S(T(T_{m-1}(\theta_{T_{j}}\boldsymbol{W}))-r)\nonumber\\
&~~~~~~~~~G(y_{r+T_{m-1}(\theta_{T_{j}}\boldsymbol{W})})d\theta_{T_{m-1}(\boldsymbol{W})}\boldsymbol{W}_{r}\bigg]
+\int_{T_{i-1}(\theta_{T_{j}}\boldsymbol{W})}^{t}S(t-r)G(y_{r})d\theta_{T_j}\boldsymbol{W}_{r}.
\end{align}
\end{small}
Using Lemma \ref{lemma A.3}, \ref{lemma A.4} and $|I_{i,j}|\leq 1, \rho_{\alpha,I_{i,0}}(\boldsymbol{W})\leq \mu\leq 1$, we obtain
\begin{align}\label{A.8}
&\left\|\int_{T_{i-1}(\theta_{T_{j}}\boldsymbol{W})}^{\cdot}S(\cdot-r)G(y_r)d\theta_{T_j}\boldsymbol{W}_r,G(y)\right\|_{W,2\alpha,\gamma,I_{i,j}}\nonumber\\
&~~~~~~~~~~~~~\leq C(1+|I_{i,j}|^{\alpha-\sigma})(\rho_{\alpha,I_{i,0}}(\boldsymbol{W})+1)\|G(y),DG(y)G(y)\|_{W,2\alpha,\gamma-\sigma,I_{i,j}}\nonumber\\
&~~~~~~~~~~~~~\leq C\mu(\|y,y^\prime\|_{W,2\alpha,\gamma,I_{i,j}}+1).
\end{align}
Note that the Gubnelli derivative of rough integral $\int_{t_1}^{t_2}S(\cdot-r)G(y_r)d\boldsymbol{W}_{r}$ on $I_{i,j}$ is zero. Indeed, for any interval $[t_1,t_2]$ and $t>s\geq t_2$,
\begin{align*}
   &\int_{t_1}^{t_2}S(t-r)G(y_r)d\theta_{T_j}\boldsymbol{W}_{r}-\int_{t_1}^{t_2}S(s-r)G(y_r)d\theta_{T_j}\boldsymbol{W}_{r}\\
   &~~~~~~=\left(S(t-t_2)-S(s-t_2)\right)\int_{t_1}^{t_2}S(t_2-r)G(y_r)d\theta_{T_j}\boldsymbol{W}_{r},
\end{align*}
and \eqref{2.1*}-\eqref{2.2*} show that  $\int_{t_1}^{t_2}S(\cdot-r)G(y_r)d\boldsymbol{W}_{r}\in C^{\alpha}(I_{i,j},\mathcal{B}_{\gamma-\alpha})\cap C^{2\alpha}(I_{i,j},\mathcal{B}_{\gamma-2\alpha})$. Thus,  by Definition \ref{def 2.4}, we know that $0$ is the  Gubinelli derivative of  the rough integral $\int_{t_1}^{t_2}S(\cdot-r)G(y_r)d\boldsymbol{W}_{r}$ on $I_{i,j}$. Then
\begin{align*}
\left\|\sum_{m=1}^{i-1}\int_{I_{m,j}}\!S(\cdot\!-\!r)G(y_{r})d\theta_{T_j}\boldsymbol{W}_{r},0\right\|_{W,2\alpha,\gamma,I_{i,j}}\leq &\left\|\sum_{m=1}^{i-1}\int_{I_{m,j}}\!S(\cdot\!-\!r)G(y_{r})d\theta_{T_j}\boldsymbol{W}_{r}\right\|_{\infty,\gamma,I_{i,j}}\\
&+\left\|\sum_{m=1}^{i-1}\int_{I_{m,j}}\!S(\cdot\!-\!r)G(y_{r})d\theta_{T_j}\boldsymbol{W}_{r}\right\|_{\alpha,\gamma-\alpha,I_{i,j}}\\
&+\left\|\sum_{m=1}^{i-1}\int_{I_{m,j}}\!S(\cdot\!-\!r)G(y_{r})d\theta_{T_j}\boldsymbol{W}_{r}\right\|_{2\alpha,\gamma-2\alpha,I_{i,j}}.
\end{align*}
By \eqref{A.7}, Lemma \ref{lemma A.3}, \ref{lemma A.4}, $|I_{m,j}|=T(T_{m-1}(\theta_{T_j}\boldsymbol{W}))\leq 1,$ $\rho_{\alpha,I_{m,0}}(\boldsymbol{W})\leq \mu\leq 1, m=1,\cdots, i-1$, we get
\begin{align}\label{A.9}
&\left\|\sum_{m=1}^{i-1}\int_{I_{m,j}}\!S(\cdot\!-\!r)G(y_{r})d\boldsymbol{W}_{r}\right\|_{\infty,\gamma,I_{i,j}}\nonumber \\
&~~\leq\sum_{i=1}^{i-1}C\mu e^{-\lambda(T_{i-1}(\theta_{T_{j}}\boldsymbol{W})-T_{m}(\theta_{T_{j}}\boldsymbol{W}))}\left(\|y,y^{\prime}\|_{W,2\alpha,\gamma,[T_{m-1}(\theta_{T_{j}}\boldsymbol{W}),T_{m}(\theta_{T_{j}}\boldsymbol{W})]}+1\right).
\end{align}
For a given $\theta\in\{\alpha,2\alpha\}$ and $s<t\in I_{i,j}$, similar to Lemma \ref{lemma A.2}, using the same computation as \eqref{A.9}, we have that
\begin{align*}
&\left\|\sum_{m=1}^{i-1}\left(\int_{I_{m,j}}\!S(t\!-\!r)G(y_{r})d\theta_{T_j}\boldsymbol{W}_{r}-\int_{I_{m,j}}\!S(s\!-\!r)G(y_{r})d\theta_{T_j}\boldsymbol{W}_{r}\right)\right\|_{\gamma-\theta,I_{i,j}}\nonumber\\
&\leq\sum_{m=1}^{i-1} \|S(t^\prime)-S(s^\prime)\|_{\mathcal{L}(\mathcal{B}_{\gamma},\mathcal{B}_{\gamma-\theta})}\|S(s^{\prime}+T_{i-1}(\theta_{T_{j}}\boldsymbol{W})-T_{m}(\theta_{T_{j}}\boldsymbol{W}))\|_{\mathcal{L}(\mathcal{B}_{\gamma})}\\
&\times \left\|\int_{0}^{T(T_{m-1}(\theta_{T_{j}}\boldsymbol{W}))}S(T(T_{m-1}(\theta_{T_{j}}\boldsymbol{W}))-r)G(y_{r+T_{m-1}(\theta_{T_{j}}\boldsymbol{W})})d\theta_{T_{m-1}(\boldsymbol{W})}\boldsymbol{W}_{r}\right\|_{\gamma}\\
&\leq \sum_{m=1}^{i-1}C\mu(t^\prime-s^\prime)^\theta e^{-\lambda(T_{i-1}(\theta_{T_{j}}\boldsymbol{W})-T_{m}(\theta_{T_{j}}\boldsymbol{W}))}(1+\|y,y^{\prime}\|_{W,2\alpha,\gamma,[T_{m-1}(\theta_{T_{j}}\boldsymbol{W}),T_{m}(\theta_{T_{j}}\boldsymbol{W})]}).
\end{align*}
Since $t^\prime-s^{\prime}=t-s$, thus,
\begin{align}\label{A.10}
&\left\|\sum_{m=1}^{i-1}\int_{I_{m,j}}\!S(\cdot\!-\!r)G(y_{r})d\theta_{T_j}\boldsymbol{W}_{r}\right\|_{\theta,\gamma-\theta,I_{i,j}}\nonumber\\
&~~\leq \sum_{m=1}^{i-1}C\mu e^{-\lambda(T_{i-1}(\theta_{T_{j}}\boldsymbol{W})-T_{m}(\theta_{T_{j}}\boldsymbol{W}))}(1+\|y,y^{\prime}\|_{W,2\alpha,\gamma,[T_{m-1}(\theta_{T_{j}}\boldsymbol{W}),T_{m}(\theta_{T_{j}}\boldsymbol{W})]}).
\end{align}
It follows from \eqref{A.8}-\eqref{A.10}, one completes the proof of this lemma.
\end{proof}
\begin{proof}
Now, we have all tools to prove Lemma \ref{lemma 3.4}.  Since $(y,y^\prime)$ solves \eqref{3.2}, then
\begin{align*}
    \|y,y^\prime\|_{W,2\alpha,\gamma,I_{i,j}}
    &\leq \|S(\cdot)y_0,0\|_{W,2\alpha,\gamma,I_{i,j}}
     + \left\|\int_0^\cdot S(\cdot-r) F(y_r)d r,0\right\|_{W,2\alpha,\gamma,I_{i,j}}\\
    &+ \left\|\int_0^\cdot S(\cdot-r) G(y_r) d\theta_{T_j}\boldsymbol{W}_r, G(y)\right\|_{W,2\alpha,\gamma,I_{i,j}}.
\end{align*}
Therefore, using Lemma \ref{lemma A.1}, \ref{lemma A.2}, \ref{lemma A.5}, we  complete the proof of Lemma \ref{lemma 3.4}.
\end{proof}
\section{Proof of Lemma 5.4}\label{Appendix B}
\setcounter{equation}{0}
\renewcommand{\theequation}{B.\arabic{equation}}
\begin{proof}
According to Definition \ref{definition 5.1}, we have
\begin{small}
\begin{align*}
d_{2\alpha,\gamma\!-\!\sigma,I_{m}}&(G(y),G(z))=\|G(y)\!-\!G(z)\|_{\infty,\gamma\!-\!\sigma,I_{m}}\!+\!\| DG(y)G(y)\!-\!DG(z)G(z) \|_{\infty,\gamma\!-\!\sigma\!-\!\alpha,I_{m}} \\
&+\interleave DG(y)G(y)-DG(z)G(z)\interleave_{\alpha,\gamma-\sigma-2\alpha,I_{m}}+\interleave R^{G(y)}-R^{G(z)} \interleave_{\alpha,\gamma-\sigma-\alpha,I_{m}}\\
&+\interleave R^{G(y)}-R^{G(z)}\interleave_{2\alpha,\gamma-\sigma-2\alpha,I_{m}}.
\end{align*}
\end{small}
By assumption $\textbf{(G)}$, it is easy to get
\begin{align*}
\|G(y)-G(z)\|_{\infty,\gamma-\sigma,I_{m}}\leq \mu \|y-z\|_{\infty,\gamma,I_{m}}.
\end{align*}
Similarly, 
we have
$$\|DG(y)G(y)-DG(z)G(z)\|_{\infty,\gamma-\sigma-\alpha}\leq C\mu \|y-z\|_{\infty,\gamma,I_{m}}(\|y,y^{\prime}\|_{W,2\alpha,\gamma,I_{m}}+1).$$
For any $s<t\in  I_{m}$,
\begin{align*}
\|DG(y_{t})G(y_{t})&-DG(z_{t})G(z_{t})-DG(y_{s})G(y_{s})+DG(z_{s})G(z_{s})\|_{\gamma-\sigma-2\alpha}\\
&\leq\|DG(y_{t})-DG(z_{t})-DG(y_{s})+DG(z_{s})\|_{\mathcal{L}(\mathcal{B}_{\gamma-2\alpha},\mathcal{B}_{\gamma-\sigma-2\alpha})}\\
&~~~~~~~~~~~~~~~~~~~~~~~~~~~~\times\|G(y_{t})+G(z_{t})+G(y_{s})+G(z_{s})\|_{\gamma-2\alpha}\\
&+\|DG(y_{t})+DG(z_{t})-DG(y_{s})-DG(z_{s})\|_{\mathcal{L}(\mathcal{B}_{\gamma-2\alpha},\mathcal{B}_{\gamma-\sigma-2\alpha})}\\
&~~~~~~~~~~~~~~~~~~~~~~~~~~~~\times\|G(y_{t})-G(z_{t})+G(y_{s})-G(z_{s})\|_{\gamma-2\alpha}\\
&+\|DG(y_{t})-DG(z_{t})+DG(y_{s})-DG(z_{s})\|_{\mathcal{L}(\mathcal{B}_{\gamma-2\alpha},\mathcal{B}_{\gamma-\sigma-2\alpha})}\\
&~~~~~~~~~~~~~~~~~~~~~~~~~~~~\times\|G(y_{t})+G(z_{t})-G(y_{s})-G(z_{s})\|_{\gamma-2\alpha}\\
&+\|DG(y_{t})+DG(z_{t})+DG(y_{s})+DG(z_{s})\|_{\mathcal{L}(\mathcal{B}_{\gamma-2\alpha},\mathcal{B}_{\gamma-\sigma-2\alpha})}\\
&~~~~~~~~~~~~~~~~~~~~~~~~~~~~\times\|G(y_{t})-G(z_{t})-G(y_{s})+G(z_{s})\|_{\gamma-2\alpha}\\
&:=J_{1}+J_{2}+J_{3}+J_{4}.
\end{align*}
We deal with these four terms  separately. For the first one, using the Mean Value Theorem twice we have that
\begin{align}\label{5.3}
&\|DG(y_{t})-DG(z_{t})-DG(y_{s})+DG(z_{s})\|_{\mathcal{L}(\mathcal{B}_{\gamma-2\alpha},\mathcal{B}_{\gamma-\sigma-2\alpha})}\nonumber\\
&\leq\|D^{2}G\|_{\mathcal{L}(\mathcal{B}^{2}_{\gamma-2\alpha},\mathcal{B}_{\gamma-\sigma-2\alpha})}\|y_{s,t}-z_{s,t}\|_{\gamma-2\alpha}+\|D^{3}G\|_{\mathcal{L}(\mathcal{B}^{3}_{\gamma-2\alpha},\mathcal{B}_{\gamma-\sigma-2\alpha})}\nonumber\\
&~~~~~~~~~~~~~~~~~~~~~~~~~~~~~~~~~~~~~~~~~~~~~~~~~~~\times\|y-z\|_{\infty,\gamma-2\alpha,I_{m}}\|y_{s,t}\|_{\gamma-2\alpha}\nonumber\\
&\leq C\mu\left(\|y_{s,t}-z_{s,t}\|_{\gamma-2\alpha}+\|y-z\|_{\infty,\gamma-2\alpha,I_{m}}\|y_{s,t}\|_{\gamma-2\alpha}\right).
\end{align}
Invoking identities $y_{s,t}=y^{\prime}_{s}W_{s,t}+R^{y}_{s,t}$ and $z_{s,t}=z^{\prime}_{s}\tilde{W}_{s,t}+R^{z}_{s,t}$, we have
\begin{align}
\|y_{s,t}-z_{s,t}\|_{\gamma-2\alpha}&=\|y_{s}^{\prime}W_{s,t}+R^{y}_{s,t}-z_{s}^{\prime}\tilde{W}_{s,t}-R^{z}_{s,t}\|_{\gamma-2\alpha}\nonumber\\
&\leq \|y_{s}^{\prime}(W_{s,t}-\tilde{W}_{s,t})\|_{\gamma-2\alpha}+\|(y_{s}^{\prime}-z_{s}^{\prime})\tilde{W}_{s,t}\|_{\gamma-2\alpha}+\|R^{y}_{s,t}-R^{y}_{s,t}\|_{\gamma-2\alpha}\nonumber\\
&\leq \Bigg(\|y^{\prime}_{s}\|_{\gamma-2\alpha}\interleave W-\tilde{W}\interleave_{\alpha,I_{m}}  +\|R^{y}-R^{z}\|_{\alpha,\gamma-2\alpha,I_{m}}\nonumber\\
&~~~~~+\|y^{\prime}_{s}-z^{\prime}_{s}\|_{\gamma-2\alpha}\interleave \tilde{W}\interleave_{\alpha,I_{m}}\Bigg)(t-s)^{\alpha}\nonumber\\
&\leq C\Bigg(\|y^{\prime}\|_{\infty,\gamma-\alpha,I_{m}}\interleave W-\tilde{W}\interleave_{\alpha,I_{m}}  +\|R^{y}-R^{z}\|_{\alpha,\gamma-\alpha,I_{m}}\nonumber\\
&~~~~~+\|y^{\prime}-z^{\prime}\|_{\infty,\gamma-\alpha,I_{m}}\interleave \tilde{W}\interleave_{\alpha,I_{m}}\Bigg)(t-s)^{\alpha}.
\end{align}
Thus,
\begin{align}\label{5.5}
\|y-z\|_{\alpha,\gamma-2\alpha,I_{m}}&\leq C \Bigg(d_{2\alpha,\gamma,I_{m}}(y,z)(\interleave \tilde{W}\interleave_{\alpha,I_{m}}+1)\nonumber\\
&~~~~~+\|y,y^{\prime}\|_{X,2\alpha,\gamma,I_{m}}d_{\alpha,I_{m}}(\boldsymbol{W},\tilde{\boldsymbol{W}})\Bigg).
\end{align}
Similarly,
\begin{align}\label{B.4}
\|y\|_{\alpha,\gamma-2\alpha}\leq C\|y,y^{\prime}\|_{X,2\alpha,\gamma,I_{m}}(1+\interleave W\interleave_{\alpha,I_{m}} ) .
\end{align}
Combining \eqref{5.3}, \eqref{5.5}, \eqref{B.4}, we  get
\begin{align}\label{5.7}
&\|DG(y_{t})-DG(z_{t})-DG(y_{s})+DG(z_{s})\|_{\mathcal{L}(\mathcal{B}_{\gamma-2\alpha},\mathcal{B}_{\gamma-\sigma-2\alpha})}\nonumber\\
&\leq C\mu\Bigg(d_{2\alpha,\gamma,I_{m}}(y,z)\left(\interleave \tilde{W}\interleave_{\alpha,I_{m}} +1+\|y,y^{\prime}\|_{W,2\alpha,\gamma,I_{m}}(\interleave W\interleave_{\alpha,I_{m}} +1)\right)\nonumber\\
&~~~~~~~~~~~~~~~~~~~~~~~~~~~~~~+\|y,y^{\prime}\|_{W,2\alpha,\gamma,I_{m}}d_{\alpha,I_{m}}(\boldsymbol{W},\tilde{\boldsymbol{W}})\Bigg)(t-s)^{\alpha}.
\end{align}
Then, by \eqref{5.7} we obtain
\begin{align}
J_{1}&\leq C\mu(\|y,y^{\prime}\|_{W,2\alpha,\gamma,I_{m}}+\|z,z^{\prime}\|_{\tilde{W},2\alpha,\gamma,I_{m}})\nonumber\\
&\times\Bigg(d_{2\alpha,\gamma,I_{m}}(y,z)\left(\interleave \tilde{W}\interleave_{\alpha,I_{m}} +1+\|y,y^{\prime}\|_{W,2\alpha,\gamma,I_{m}}(\interleave W\interleave_{\alpha,I_{m}} +1)\right)\nonumber\\
&~~~~~~~~~~~~~~~~~~~~~~~~~~~~~~+\|y,y^{\prime}\|_{W,2\alpha,\gamma,I_{m}}d_{\alpha,I_{m}}(\boldsymbol{W},\tilde{\boldsymbol{W}})\Bigg)(t-s)^{\alpha}.
\end{align}
For the second term, using \eqref{B.4} for $y,z$  we get
\begin{align}
J_{2}\leq &C\mu d_{2\alpha,\gamma,I_{m}}(y,z)\left(\|y,y^{\prime}\|_{W,2\alpha,\gamma,I_{m}}+\|z,z^{\prime}\|_{\tilde{W},2\alpha,\gamma,I_{m}}\right)\nonumber\\
&~~~~~~~~~~~~~~~~\times(1+\interleave W\interleave_{\alpha,I_{m}}+\interleave \tilde{W}\interleave_{\alpha,I_{m}} )(t-s)^{\alpha}.
\end{align}
For the third term, we have
\begin{align}
J_{3}\leq C\mu d_{2\alpha,\gamma,I_{m}}(y,z)\left(\|y,y^{\prime}\|_{W,2\alpha,\gamma,I_{m}}+\|z,z^{\prime}\|_{\tilde{W},2\alpha,\gamma,I_{m}}\right)(t-s)^{\alpha}.
\end{align}
Finally, we can easily get the estimate of $J_{4}$ as follows
\begin{align}
J_{4}\leq C\mu d_{2\alpha,\gamma,I_{m}}(y,z)(t-s)^{\alpha}.
\end{align}
Combining the above estimates, we have
\begin{small}
\begin{align}
&\interleave DG(y)G(y)-DG(z)G(z)\interleave_{\alpha,\gamma-\sigma-2\alpha}\nonumber\\
&~~\leq C\mu d_{2\alpha,\gamma,I_{m}}(y,z)\left(\|y,y^{\prime}\|_{W,2\alpha,\gamma,I_{m}}+\|z,z^{\prime}\|_{\tilde{W},2\alpha,\gamma,I_{m}}+1\right)^{2}\nonumber\\
&~~~~~~~~~~~~~~\times(1+\interleave W\interleave_{\alpha,I_{m}}+\interleave \tilde{W}\interleave_{\alpha,I_{m}} )\nonumber\\
&~~~~+C\mu\left(\|y,y^{\prime}\|_{W,2\alpha,\gamma,I_{m}}+\|z,z^{\prime}\|_{\tilde{W},2\alpha,\gamma,I_{m}}+1\right)^{2}d_{\alpha,I_{m}}(\boldsymbol{W},\tilde{\boldsymbol{W}}).
\end{align}
\end{small}
Finally, we deal with the remainder term. For any controlled rough path $(y,y^\prime)\in{D}^{2\alpha}_{W,\gamma}$,  Lemma \ref{lemma A.3} shows that $(G(y),DG(y)G(y))\in{D}^{2\alpha}_{W,\gamma-\sigma}$. Thus,
 \begin{align*}
	 R^{G(y)}_{s,t}&=G(y_{t})-G(y_{s})-DG(y_{s})G(y_{s})W_{s,t}\nonumber\\
	 =&\int_{0}^{1}DG(y_{s}+\eta (y_{t}-y_{s}))(y_{t}-y_{s})d\eta-DG(y_{s})G(y_{s})W_{s,t}\nonumber\\
	 =&\int_{0}^{1}DG(y_{s}\!+\!\eta (y_{t}\!-\!y_{s}))R^{y}_{s,t}d\eta\!+\!\int_{0}^{1}DG(y_{s}+\eta (y_{t}-y_{s}))-DG(y_{s})d\eta G(y_{s})W_{s,t}.
	 \end{align*}
 For $\theta\in\{\alpha,2\alpha\},s<t\in I_{m}$, using the above inequality for $y,z$ we have
\begin{align}
&\|R^{G(y)}_{s,t}-R^{G(z)}_{s,t}\|_{\gamma-\theta-\sigma}\leq \Bigg\|\int_{0}^{1}DG(y_{s}+\eta y_{s,t})R^{y}_{s,t}d\eta-\int_{0}^{1}DG(z_{s}+\eta z_{s,t})R^{z}_{s,t}d\eta\Bigg\|_{\gamma-\sigma-\theta}\nonumber\\
&~~~~~~~~~~~~~~~~~~~~~~~~~~~~~~~~~~+\Bigg\|\int_{0}^{1}DG(y_{s}+\eta y_{s,t})-DG(y_{s})d\eta G(y_{s})W_{s,t}\nonumber\\
&~~~~~~~~~~~~~~~~~~~~~~~~~~~~~~~~~~-\int_{0}^{1}DG(z_{s}+\eta z_{s,t})+DG(z_{s})d\eta G(z_{s})\tilde{W}_{s,t}\Bigg\|_{\gamma-\theta-\sigma}\nonumber\\
&~~~~~~~~~~~~~~~~~~~~~~~~~~~~~\leq \Bigg\|\int_{0}^{1}DG(y_{s}+\eta y_{s,t})(R^{y}_{s,t}-R^{z}_{s,t})d\eta\Bigg\|_{\gamma-\theta-\sigma}\nonumber\\
&~~~~~~~~~~~~~~~~~~~~~~~~~~~~~~~~~+\Bigg\|\int_{0}^{1}(DG(y_{s}+\eta y_{s,t})-DG(z_{s}+\eta z_{s,t}))R^{z}_{s,t}d\eta\Bigg\|_{\gamma-\theta-\sigma}\nonumber\\
&~~~~~~~~~~~~~~~~~~~~~~~~~~~~~~~~~+\Bigg\|\int_{0}^{1}DG(y_{s}+\eta y_{s,t})-DG(y_{s})d\eta (G(y_{s})W_{s,t}-G(z_{s})\tilde{W}_{s,t})\Bigg\|_{\gamma-\theta-\sigma}\nonumber\\
&~~~~~~~~~~~~~~~~~~~~~~~~~~~~+\Bigg\|\int_{0}^{1}DG(y_{s}\!+\!\eta y_{s,t})\!-\!DG(y_{s})\!-\!DG(z_{s}\!+\!\eta z_{s,t})\!+\!DG(z_{s})d\eta G(z_{s})\tilde{W}_{s,t}\Bigg\|_{\gamma-\theta-\sigma}\nonumber\\
&~~~~~~~~~~~~~~~~~~~~~~~~~~~~
:=\uppercase\expandafter{\romannumeral1}+\uppercase\expandafter{\romannumeral2}+\uppercase\expandafter{\romannumeral3}+\uppercase\expandafter{\romannumeral4}.
\end{align}
For the first term, we have
\begin{align}
\uppercase\expandafter{\romannumeral1}&\leq \left\|DG\right\|_{\mathcal{L}(\mathcal{B}_{\gamma-\theta},\mathcal{B}_{\gamma-\theta-\sigma})}\|R^{y}-R^{z}\|_{\theta,\gamma-\theta}(t-s)^{\theta}\nonumber\\
&\leq  \mu d_{2\alpha,\gamma,I_{m}}(y,z)(t-s)^{\theta}.
\end{align}
For the second term, we obtain
\begin{align}
\uppercase\expandafter{\romannumeral2}&\leq \|D^{2}G\|_{\mathcal{L}(\mathcal{B}^{2}_{\gamma-\theta},\mathcal{B}_{\gamma-\theta-\sigma})}\|y-z\|_{\infty,\gamma-\theta,I_{m}}\|R^{z}\|_{\theta,\gamma-\theta,I_{m}}(t-s)^{\theta}\nonumber\\
&\leq C\mu d_{2\alpha,\gamma,I_{m}}(y,z)\|z,z^{\prime}\|_{\tilde{W},2\alpha,\gamma,I_{m}}(t-s)^{\theta}.
\end{align}
Firstly, as $\theta=\alpha$,  for the third term  we have
\begin{small}
\begin{align}
\uppercase\expandafter{\romannumeral3}&\leq C\left\|DG\right\|_{\mathcal{L}(\mathcal{B}_{\gamma-\alpha},\mathcal{B}_{\gamma-\alpha-\sigma})}\left\|G(y_{s})W_{s,t}-G(z_{s})\tilde{W}_{s,t}\right\|_{\gamma-\alpha}\nonumber\\
\leq& C\mu\left(\!\|y,y^{\prime}\|_{W,2\alpha,\gamma,I_{m}}d_{\alpha,I_{m}}(\boldsymbol{W},\tilde{\boldsymbol{W}})+d_{2\alpha,\gamma,I_{m}}(y,z)\interleave \tilde{W}\interleave_{\alpha,I_{m}}\right)(t-s)^{\alpha}.
\end{align}
\end{small}
Secondly, for  $\theta=2\alpha$,  by \eqref{B.4} we obtain
\begin{align}
\uppercase\expandafter{\romannumeral3}&\leq \left\|D^{2}G\right\|_{\mathcal{L}(\mathcal{B}^{2}_{\gamma-2\alpha},\mathcal{B}_{\gamma-2\alpha-\sigma})}\|y_{s,t}\|_{\gamma-2\alpha}\left\|G(y_{s})W_{s,t}-G(z_{s})\tilde{W}_{s,t}\right\|_{\gamma-2\alpha}\nonumber\\
&\leq C\mu\|y,y^{\prime}\|_{W,2\alpha,\gamma,I_{m}}(1+\interleave W\interleave_{\alpha,I_{m}})\bigg(\|y,y^{\prime}\|_{W,2\alpha,\gamma,I_{m}}d_{\alpha,I_{m}}(\boldsymbol{W},\tilde{\boldsymbol{W}})\nonumber\\
&~~~~~~~+d_{2\alpha,\gamma,I_{m}}(y,z)\interleave \tilde{W}\interleave_{\alpha,I_{m}}\bigg)(t-s)^{2\alpha}\nonumber\\
&\leq C\mu \Bigg[\|y,y^{\prime}\|_{W,2\alpha,\gamma,I_{m}}\left(1+\interleave W\interleave_{\alpha,I_{m}} \right)\interleave \tilde{W}\interleave_{\alpha,I_{m}}d_{2\alpha,\gamma,I_{m}}(y,z)\nonumber\\
&+\|y,y^{\prime}\|^{2}_{W,2\alpha,\gamma,I_{m}}\left(1+\interleave W\interleave_{\alpha,I_{m}} \right)d_{\alpha,I_{m}}(\boldsymbol{W},\tilde{\boldsymbol{W}})\Bigg](t-s)^{2\alpha}.
\end{align}
Similarly,  we can deal with $\uppercase\expandafter{\romannumeral4}$. For $\theta=\alpha$, we have
\begin{align}
\uppercase\expandafter{\romannumeral4}&\leq C\left\|D^{2}G\right\|_{\mathcal{L}(\mathcal{B}^{2}_{\gamma-\alpha},\mathcal{B}_{\gamma-\alpha-\sigma})}\|y-z\|_{\infty,\gamma-\alpha,,I_{m}}\left\|G(z_{s})\tilde{W}_{s,t}\right\|_{\gamma-\alpha}\nonumber\\
&\leq C\mu \|y-z\|_{\infty,\gamma,I_{m}}\|z,z^{\prime}\|_{\tilde{W},2\alpha,\gamma,I_{m}}\interleave \tilde{W}\interleave_{\alpha,,I_{m}}(t-s)^{\alpha}\nonumber\\
&\leq C\mu d_{2\alpha,\gamma,I_{m}}(y,z)\|z,z^{\prime}\|_{\tilde{W},2\alpha,\gamma,I_{m}}\interleave \tilde{W}\interleave_{\alpha,,I_{m}}(t-s)^{\alpha}.
\end{align}
For $\theta=2\alpha$, by \eqref{5.5}-\eqref{B.4} we have
\begin{align}
\uppercase\expandafter{\romannumeral4}&\leq \Bigg[\|D^{2}G\|_{\mathcal{L}(\mathcal{B}^{2}_{\gamma-2\alpha},\mathcal{B}_{\gamma-2\alpha-\sigma})}\|y_{s,t}-z_{s,t}\|_{\gamma-2\alpha}+C\|D^{3}G\|_{\mathcal{L}(\mathcal{B}^{3}_{\gamma-2\alpha},\mathcal{B}_{\gamma-2\alpha-\sigma})}\bigg(\nonumber\\
&~~~~~~~~~~~~~\|y_{s,t}\|_{\gamma-2\alpha}\|y-z\|_{\infty,\gamma-2\alpha,I_{m}}\bigg)\Bigg]\|z,z^{\prime}\|_{\tilde{W},2\alpha,\gamma,I_{m}}\interleave \tilde{W}\interleave_{\alpha,I_{m}}(t-s)^{\alpha}\nonumber\\
&\leq C\mu \Bigg[\Bigg(d_{2\alpha,\gamma,I_{m}}(y,z)(\interleave \tilde{W}\interleave_{\alpha,I_{m}}+1)+\|y,y^{\prime}\|_{W,2\alpha,\gamma,I_{m}}d_{\alpha,I_{m}}(\boldsymbol{W},\tilde{\boldsymbol{W}})\Bigg)\nonumber\\
&+d_{2\alpha,\gamma,I_{m}}(y,z)\|y,y^{\prime}\|_{W,2\alpha,\gamma,I_{m}}(1+\interleave W\interleave_{\alpha,I_{m}})\Bigg]\|z,z^{\prime}\|_{\tilde{W},2\alpha,\gamma,I_{m}}\interleave \tilde{W}\interleave_{\alpha,I_{m}}(t-s)^{2\alpha}\nonumber\\
&\leq C\mu \Bigg[ d_{2\alpha,\gamma,I_{m}}(y,z)(\|z,z^{\prime}\|_{\tilde{W},2\alpha,\gamma,I_{m}}+\|y,y^{\prime}\|_{W,2\alpha,\gamma,I_{m}}+1)^{2}\nonumber\\
&~~~~~~~~~~~~~~~~\times(\interleave W\interleave_{\alpha,I_{m}}+\interleave \tilde{W}\interleave_{\alpha,I_{m}}+1)\nonumber\\
&+C\mu(\!\|z,z^{\prime}\|_{\tilde{W},2\alpha,\gamma,I_{m}}\!+\!\|y,y^{\prime}\|_{W,2\alpha,\gamma,I_{m}}\!+\!1)^{2}d_{\alpha,I_{m}}(\boldsymbol{W},\tilde{\boldsymbol{W}})\!\Bigg]\!(t\!-\!s)^{2\alpha}.
\end{align}
Then, by the above  results we get
\begin{align}
&\|R^{y}-R^{z}\|_{\alpha,\gamma-\alpha,I_{m}}+\|R^{y}-R^{z}\|_{2\alpha,\gamma-2\alpha,I_{m}}\nonumber\\
&\leq C\mu d_{2\alpha,\gamma,I_{m}}(y,z)(\|z,z^{\prime}\|_{\tilde{W},2\alpha,\gamma,I_{m}}+\|y,y^{\prime}\|_{W,2\alpha,\gamma,I_{m}}+1)^{2}\nonumber\\
&~~~~~~~~~~~~~~~~~\times(\interleave W\interleave_{\alpha,I_{m}}+\interleave \tilde{W}\interleave_{\alpha,I_{m}}+1)\nonumber\\
&~~+C\mu(\|z,z^{\prime}\|_{\tilde{W},2\alpha,\gamma,I_{m}}+\|y,y^{\prime}\|_{W,2\alpha,\gamma,I_{m}}+1)^{2}d_{\alpha,I_{m}}(\boldsymbol{W},\tilde{\boldsymbol{W}})\nonumber\\
&~~~~~~~~~~~~~~~~~\times(\interleave W\interleave_{\alpha,I_{m}}+\interleave \tilde{W}\interleave_{\alpha,I_{m}}+1).
\end{align}
Furthermore, we obtain
\begin{small}
 \begin{align*}
d_{2\alpha,\gamma-\sigma,I_{m}}(G(y),G(z))&\leq C\mu d_{2\alpha,\gamma,I_{m}}(y,z)(\|z,z^{\prime}\|_{\tilde{W},2\alpha,\gamma,I_{m}}+\|y,y^{\prime}\|_{W,2\alpha,\gamma,I_{m}}+1)^{2}\nonumber\\
&~~~~~~~~~~~~~~~~~\times(\interleave W\interleave_{\alpha,I_{m}}+\interleave \tilde{W}\interleave_{\alpha,I_{m}}+1)\nonumber\\
&~~+C\mu(\|z,z^{\prime}\|_{\tilde{W},2\alpha,\gamma,I_{m}}+\|y,y^{\prime}\|_{W,2\alpha,\gamma,I_{m}}+1)^{2}d_{\alpha,I_{m}}(\boldsymbol{W},\tilde{\boldsymbol{W}})\nonumber\\
&~~~~~~~~~~~~~~~~~\times(\interleave W\interleave_{\alpha,I_{m}}+\interleave \tilde{W}\interleave_{\alpha,I_{m}}+1).
 \end{align*}
 \end{small}
\end{proof}
\section{Proof of Theorem 5.2}\label{Appendix C}
\setcounter{equation}{0}

\renewcommand{\theequation}{C.\arabic{equation}}
\begin{proof}
By the definition \ref{definition 5.1}, we have that
\begin{align}
d_{2\alpha,0,I_{i}(T)}(S(\cdot)y_{0},S(\cdot)y_{0}^{\eta})&=\|S(\cdot)(y_{0}-y_{0}^{\eta})\|_{\infty,0,I_{i}(T)}+\|S(\cdot)(y_{0}-y_{0}^{\eta})\|_{\alpha,-\alpha,I_{i}(T)}\nonumber\\
&+ \|S(\cdot)(y_{0}-y_{0}^{\eta})\|_{2\alpha,-2\alpha,I_{i}(T)}.
\end{align}
Clearly,
\begin{align}\label{5.25}
\|S(\cdot)(y_{0}-y_{0}^{\eta})\|_{\infty,0,I_{i}(T)}&\leq Ce^{-\lambda T_{i-1}(\theta_{-T}\boldsymbol{W}(\omega))}\|y_{0}-y^{\eta}_{0}\|
\end{align}
In addition, for $\theta=\{\alpha,2\alpha\}$ we have
\begin{align}\label{5.26}
\interleave S(\cdot)(y_{0}-y_{0}^{\eta})\interleave_{\theta,-\theta}\leq Ce^{-\lambda T_{i-1}(\theta_{-T}\boldsymbol{W}(\omega))}\|y_{0}-y^{\eta}_{0}\|.
\end{align}
Then, \eqref{5.25} and \eqref{5.26} show that
\begin{align}\label{5.27}
 d_{2\alpha,0,I_{i}(T)}(S(\cdot)y_{0},S(\cdot)y_{0}^{\eta})\leq Ce^{-\lambda T_{i-1}(\theta_{-T}\boldsymbol{W}(\omega))}\|y_{0}-y^{\eta}_{0}\|.
\end{align}
For the drift term we need to compute
\begin{align}\label{5.28}
d_{2\alpha,0,I_{i}(T)}&\left(\int_{0}^{\cdot}S(\cdot-r)F(y_{r})dr,\int_{0}^{\cdot}S(\cdot-r)F(y^{\eta}_{r})dr\right)\nonumber\\
  &=\left\|\int_{0}^{\cdot}S(\cdot-r)\left(F(y_{r})-F(y^{\eta}_{r})\right)dr\right\|_{\infty,0,I_{i}(T)}\nonumber\\
  &~~+\left\|\int_{0}^{\cdot}S(\cdot-r)\left(F(y_{r})-F(y^{\eta}_{r})\right)dr\right\|_{\alpha,-\alpha,I_{i}(T)}\nonumber\\
  &~~+\left\|\int_{0}^{\cdot}S(\cdot-r)\left(F(y_{r})-F(y^{\eta}_{r})\right)dr\right\|_{2\alpha,-2\alpha,I_{i}(T)}.
\end{align}
The estimates of these three terms are similar to Lemma \ref{lemma A.2}  and Remark \ref{R-A.1}, we only use the  property of Lipshcitz  of  $F$  rather than linear growth.  We give directly its  result as follows
\begin{align}\label{5.39}
  d_{2\alpha,0,I_{i}(T)}&\left(\int_{0}^{\cdot}S(\cdot-r)F(y_{r})dr,\int_{0}^{\cdot}S(\cdot-r)F(y^{\eta}_{r})\right)dr\nonumber\\
  \leq& C\mu\sum_{m=1}^{i-1}e^{-\lambda\left(T_{i-1}(\theta_{-T}\boldsymbol{X})-T_{m}(\theta_{-T}\boldsymbol{X})\right)}\|y-y^{\eta}\|_{\infty,0,I_{m}(T)}\nonumber\\
&+C\mu\|y-y^{\eta}\|_{\infty,0,I_{i}(T)}.
\end{align}
Finally, denote  $T_{m}:=T_m(\theta_{-T}\boldsymbol{W}(\omega)), m\in \{1,\cdots,i\}$.   Similar to the drift term, for any $t\in I_{i}(T)$ we have
\begin{align}\label{5.40}
&\int_{0}^{t}S(t-r)G(y_{r})d\theta_{-T}\boldsymbol{W}_{r}(\omega)-\int_{0}^{t}S(t-r)G(y^{\eta}_{r})d\theta_{-T}\boldsymbol{W}^{\eta}_{r}(\omega)\nonumber\\
&~~~~=\sum_{m=1}^{i-1}\left(\int_{T_{m-1}}^{T_{m}}S(t-r)G(y_{r})d\theta_{-T}\boldsymbol{W}_{r}(\omega)-\int_{T_{m-1}}^{T_{m}}S(t-r)G(y^{\eta}_{r})d\theta_{-T}\boldsymbol{W}^{\eta}_{r}(\omega)\right)\nonumber\\
&~~~~~~~+\int_{T_{i-1}}^{t}S(t-r)G(y_{r})d\theta_{-T}\boldsymbol{W}_{r}(\omega)-\int_{T_{i-1}}^{t}S(t-r)G(y^{\eta}_{r})d\theta_{-T}\boldsymbol{W}^{\eta}_{r}(\omega).
\end{align}
Let $t^\prime=t-T_{i-1}$, $s^\prime=s-T_{i-1}$. For any  $m\in\{1,2,\cdots,i-1\}$, by \eqref{2.1*}, Lemma \ref{lemma A.3}, Lemma \ref{lemma 5.3} and Lemma \ref{lemma 5.4} for $\gamma=0$, we have that
\begin{small}
\begin{align}\label{5.41}
&\left\|\int_{T_{m-1}}^{T_{m}}S(t-r)G(y_{r})d\theta_{-T}\boldsymbol{W}_{r}(\omega)-\int_{T_{m-1}}^{T_{m}}S(t-r)G(y^{\eta}_{r})d\theta_{-T}\boldsymbol{W}^{\eta}_{r}(\omega)\right\|\nonumber\\
&~~~~=\Bigg\|\int_{0}^{T(\theta_{T_{m-1}-T}\boldsymbol{W}(\omega))}S(t-r-T_{m-1})G(y_{r+T_{m-1}})d\theta_{T_{m-1}-T}\boldsymbol{W}_{r}(\omega)\nonumber\\
&~~~~~~~~-\int_{0}^{T(\theta_{T_{m-1}-T}\boldsymbol{W}(\omega))}S(t-r-T_{m-1})G(y^{\eta}_{r+T_{m-1}})d\theta_{T_{m-1}-T}\boldsymbol{W}^{\eta}_{r}(\omega)\Bigg\|\nonumber\\
&~~~~=\Bigg\|S(t^{\prime}+T_{i-1}-T_{m})\bigg[\int_{0}^{T(\theta_{T_{m-1}-T}\boldsymbol{W}(\omega))}S(T(\theta_{T_{m-1}-T}\boldsymbol{W}(\omega))-r)G(y_{r+T_{m-1}})d\theta_{T_{m-1}-T}\boldsymbol{W}_{r}(\omega)\nonumber\\
&~~~~~~~~-\int_{0}^{T(\theta_{T_{m-1}-T}\boldsymbol{W}(\omega))}S(T(\theta_{T_{m-1}-T}\boldsymbol{W}(\omega))-r)G(y^{\eta}_{r+T_{m-1}})d\theta_{T_{m-1}-T}\boldsymbol{W}^{\eta}_{r}(\omega)\bigg]\Bigg\|\nonumber\\
&~~~~\leq Ce^{-\lambda(T_{i-1}-T_{m})}\Bigg\|\bigg[\int_{0}^{T(\theta_{T_{m-1}-T}\boldsymbol{W}(\omega))}S(T(\theta_{T_{m-1}-T}\boldsymbol{W}(\omega))-r)G(y_{r+T_{m-1}})d\theta_{T_{m-1}-T}\boldsymbol{W}_{r}(\omega)\nonumber\\
&~~~~~~~~-\int_{0}^{T(\theta_{T_{m-1}-T}\boldsymbol{W}(\omega))}S(T(\theta_{T_{m-1}-T}\boldsymbol{W}(\omega))-r)G(y^{\eta}_{r+T_{m-1}})d\theta_{T_{m-1}-T}\boldsymbol{W}^{\eta}_{r}(\omega)\bigg]\Bigg\|\nonumber\\
&~~~~\leq Ce^{-\lambda(T_{i-1}-T_{m})}\Bigg[\left\|S\left(T(\theta_{T_{m-1}-T}\boldsymbol{W}(\omega))\right)\left(G(y_{T_{m}})\theta_{-T}W_{T_{m-1},T_{m}}(\omega)-G(y^{\eta}_{T_{m}})\theta_{-T}W^{\eta}_{T_{m-1},T_{m}}(\omega)\right)\right\|\nonumber\\
&~~~~~~~~+\left\|S\left(T(\theta_{T_{m-1}-T}\boldsymbol{W}(\omega))\right)\left(DG(y_{T_{m}})G(y_{T_{m}})\theta_{-T}\mathbb{W}_{T_{m-1},T_{m}}(\omega)-DG(y^{\eta}_{T_{m}})G(y^{\eta}_{T_{m}})\theta_{-T}\mathbb{W}^{\eta}_{T_{m-1},T_{m}}(\omega)\right)\right\|\nonumber\\
&~~~~~~~~+\bigg(d_{\alpha,I_{m}(T)}(\theta_{-T}\boldsymbol{W}(\omega),\theta_{-T}\boldsymbol{W}^{\eta}(\omega))\|G(y),DG(y)G(y)\|_{W,2\alpha,-\sigma,I_{m}(T)}\nonumber\\
&~~~~~~~~+d_{2\alpha,-\sigma,I_{m}(T)}(G(y),G(y^{\eta}))d_{\alpha,I_{m}(T)}(\theta_{-T}\boldsymbol{W}^{\eta}(\omega),0)\bigg) |T(\theta_{T_{m-1}-T}\boldsymbol{W}(\omega))|^{\alpha-\sigma}\Bigg]\nonumber\\
&~~~~\leq Ce^{-\lambda(T_{i-1}-T_{m})}\Bigg[\|S(T(\theta_{T_{m-1}-T}\boldsymbol{W}(\omega)))\|_{\mathcal{L}(\mathcal{B}_{-\sigma},\mathcal{B})}\left\|G(y_{T_{m}})\theta_{-T}W_{T_{m-1},T_{m}}(\omega)-G(y^{\eta}_{T_{m}})\theta_{-T}W^{\eta}_{T_{m-1},T_{m}}(\omega)\right\|_{-\sigma}\nonumber\\
&~~~~~~~~+\Bigg(\left\|DG(y_{T_{m}})G(y_{T_{m}})\theta_{-T}\mathbb{W}_{T_{m-1},T_{m}}(\omega)-DG(y^{\eta}_{T_{m}})G(y^{\eta}_{T_{m}})\theta_{-T}\mathbb{W}^{\eta}_{T_{m-1},T_{m}}(\omega)\right\|_{-\sigma-\alpha}\nonumber\\
&~~~~~~~~\times\|S(T(\theta_{T_{m-1}-T}\boldsymbol{W}(\omega)))\|_{\mathcal{L}(\mathcal{B}_{-\sigma-\alpha},\mathcal{B})}\Bigg)\!+\!\bigg(\!d_{\alpha,I_{m}(T)}(\theta_{-T}\boldsymbol{W}(\omega),\theta_{-T}\boldsymbol{W}^{\eta}(\omega))\|G(y),DG(y)G(y)\|_{W,2\alpha,-\sigma,I_{m}(T)}\nonumber\\
&~~~~~~~~+d_{2\alpha,-\sigma,I_{m}(T)}(G(y),G(y^{\eta}))d_{\alpha,I_{m}(T)}(\theta_{-T}\boldsymbol{W}^{\eta}(\omega),0)\bigg)|T(\theta_{T_{m-1}-T}\boldsymbol{W}(\omega))|^{\alpha-\sigma}\Bigg]\nonumber\\
&~~~~\leq Ce^{-\lambda(T_{i-1}-T_{m})}\bigg(d_{\alpha,I_{m}(T)}(\theta_{-T}\boldsymbol{W}(\omega),\theta_{-T}\boldsymbol{W}^{\eta}(\omega))\|G(y),DG(y)G(y)\|_{X,2\alpha,-\sigma,I_{m}(T)}\nonumber\\
&~~~~~~~~+d_{2\alpha,-\sigma,I_{m}(T)}(G(y),G(y^{\eta}))d_{\alpha,I_{m}(T)}(\theta_{-T}\boldsymbol{W}^{\eta}(\omega),0))\bigg) |T(\theta_{T_{m-1}-T}\boldsymbol{W}(\omega))|^{\alpha-\sigma}\nonumber\\
&~~~~\leq C\mu e^{-\lambda(T_{i-1}-T_{m})}(\|y,y^{\prime}\|_{W,2\alpha,0,I_{m}(T)}+\|y^{\eta},(y^{\eta})^{\prime}\|_{W^\eta,2\alpha,0,I_{m}(T)}+1)^{2}\nonumber\\
&~~~~~~~~\times(1+\|\theta_{-T}\boldsymbol{W}(\omega)\|_{\alpha,I_{m}(T)}+\|\theta_{-T}\boldsymbol{W}^{\eta}(\omega)\|_{\alpha,I_{m}(T)})^{2}\nonumber\\
&~~~~~~~~\times\bigg(d_{\alpha,I_{m}(T)}(\theta_{-T}\boldsymbol{W}(\omega),\theta_{-T}\boldsymbol{W}^{\eta}(\omega))+d_{2\alpha,0,I_{m}(T)}(y,y^{\eta})\bigg).
\end{align}
\end{small}
Furthermore, we have
\begin{align}\label{5.42}
&\left\|\int_{T_{i-1}}^{t}S(t-r)G(y_{r})d\theta_{-T}\boldsymbol{W}_{r}(\omega)-\int_{T_{i-1}}^{t}S(t-r)G(y^{\eta}_{r})d\theta_{-T}\boldsymbol{W}^{\eta}_{r}(\omega)\right\|\nonumber\\
&~~~~\leq C\mu (\|y,y^{\prime}\|_{W,2\alpha,0,I_{i}(T)}+\|y^{\eta},(y^{\eta})^{\prime}\|_{W^\eta,2\alpha,0,I_{i}(T)}+1)^{2}\nonumber\\
&~~~~~~\times(1+\|\theta_{-T}\boldsymbol{W}(\omega)\|_{\alpha,I_{i}(T)}+\|\theta_{-T}\boldsymbol{W}^{\eta}(\omega)\|_{\alpha,I_{i}(T)})^{2}\nonumber\\
&~~~~~~\times\bigg(d_{\alpha,I_{i}(T)}(\theta_{-T}\boldsymbol{W}(\omega),\theta_{-T}\boldsymbol{W}^{\eta}(\omega))+d_{2\alpha,0,I_{i}(T)}(y,y^{\eta})\bigg).
\end{align}
Combining  \eqref{5.41} and \eqref{5.42}, we obtain
\begin{align}\label{5.43}
&\left\|\int_{0}^{\cdot}S(\cdot-r)G(y_{r})d\theta_{-T}\boldsymbol{W}_{r}(\omega)-\int_{0}^{\cdot}S(\cdot-r)G(y^{\eta}_{r})d\theta_{-T}\boldsymbol{W}^{\eta}_{r}(\omega)\right\|_{\infty,0,I_{i}(T)}\nonumber\\
&\leq \sum_{m=1}^{i-1}C\mu e^{-\lambda(T_{i-1}-T_{m})}(\|y,y^{\prime}\|_{W,2\alpha,0,I_{m}(T)}+\|y^{\eta},(y^{\eta})^{\prime}\|_{W^\eta,2\alpha,0,I_{m}(T)}+1)^{2}\nonumber\\
&~~\times(1+\|\theta_{-T}\boldsymbol{W}(\omega)\|_{\alpha,I_{m}(T)}+\|\theta_{-T}\boldsymbol{W}^{\eta}(\omega)\|_{\alpha,I_{m}(T)})^{2}\nonumber\\
&~~\times\bigg(d_{\alpha,I_{m}(T)}(\theta_{-T}\boldsymbol{W}(\omega),\theta_{-T}\boldsymbol{W}^{\eta}(\omega))+d_{2\alpha,0,I_{m}(T)}(y,y^{\eta})\bigg)\nonumber\\
&~~ +C\mu \Bigg[(\|y,y^{\prime}\|_{W,2\alpha,0,I_{i}(T)}+\|y^{\eta},(y^{\eta})^{\prime}\|_{W^\eta,2\alpha,0,I_{i}(T)}+1)^{2}\nonumber\\
&~~\times(1+\|\theta_{-T}\boldsymbol{W}(\omega)\|_{\alpha,I_{i}(T)}+\|\theta_{-T}\boldsymbol{W}^{\eta}(\omega)\|_{\alpha,I_{i}(T)})^{2}\nonumber\\
&~~\times\bigg(d_{\alpha,I_{i}(T)}(\theta_{-T}\boldsymbol{W}(\omega),\theta_{-T}\boldsymbol{W}^{\eta}(\omega))+d_{2\alpha,0,I_{i}(T)}(y,y^{\eta})\bigg)\Bigg].
\end{align}
The estimation of the term $\left\|G(y)-G(y^{\eta})\right\|_{\infty,-\alpha}$ is similar to \eqref{A.8} , and using Lemma \ref{lemma 5.4}, we have
\begin{align}\label{5.44}
&\left\|G(y)\!-\!G(y^{\eta})\right\|_{\infty,-\alpha}\!\leq \!C\!(\!1\!+\!(T_{i}(\theta_{-T}\boldsymbol{W}(\omega))\!-\!T_{i-1}(\theta_{-T}\boldsymbol{W}(\omega))^{\alpha-\sigma}))d_{2\alpha,-\sigma,I_{i}(T)}(G(y),G(y^{\eta}))\nonumber\\
&~~\leq C\mu (\|y,y^{\prime}\|_{W,2\alpha,0,I_{i}(T)}+\|y^{\eta},(y^{\eta})^{\prime}\|_{W^\eta,2\alpha,0,I_{i}(T)}+1)^{2}\nonumber\\
&~~~~~~\times(1+\|\theta_{-T}\boldsymbol{W}(\omega)\|_{\alpha,I_{i}(T)}+\|\theta_{-T}\boldsymbol{W}^{\eta}(\omega)\|_{\alpha,I_{i}(T)})^{2}\nonumber\\
&~~\times\bigg(d_{\alpha,I_{i}(T)}(\theta_{-T}\boldsymbol{W}(\omega),\theta_{-T}\boldsymbol{W}^{\eta}(\omega))+d_{2\alpha,0,I_{i}(T)}(y,y^{\eta})\bigg).
\end{align}
The estimation of  $\|G(y)-G(y^{\eta})\|_{\alpha,-2\alpha}$ is also similar to $\|G(y)\|_{\alpha,-2\alpha}$, namely \eqref{A.8}.  Thus, by Lemma \ref{lemma 5.4} we have that  
\begin{align}\label{5.46}
&\|G(y)-G(y^{\eta})\|_{\alpha,-2\alpha}\leq C\mu (\|y,y^{\prime}\|_{W,2\alpha,0,I_{i}(T)}+\|y^{\eta},(y^{\eta})^{\prime}\|_{W^\eta,2\alpha,0,I_{i}(T)}+1)^{2}\nonumber\\
&~~~~\times(1+\|\theta_{-T}\boldsymbol{W}(\omega)\|_{\alpha,I_{i}(T)}+\|\theta_{-T}\boldsymbol{W}^{\eta}(\omega)\|_{\alpha,I_{i}(T)})^{2}\nonumber\\
&~~~~\times\bigg(d_{\alpha,I_{i}(T)}(\theta_{-T}\boldsymbol{W}(\omega),\theta_{-T}\boldsymbol{W}^{\eta}(\omega))+d_{2\alpha,0,I_{i}(T)}(y,y^{\eta})\bigg).
\end{align}
Finally, we consider the H\"{o}lder seminorm of the remainder terms of the stochastic convolution $Z_{t}=\int_{0}^{t}S(t-r)G(y_{r})d\theta_{-T}\boldsymbol{W}_{r}(\omega)$ and  $Z^{\eta}_{t}=\int_{0}^{t}S(t-r)G(y^{\eta}_{r})d\theta_{-T}\boldsymbol{W}^{\eta}_{r}(\omega)$, namely $\left\|R^{Z}-R^{Z^{\eta}}\right\|_{\theta,-\theta}$. For any $s<t\in I_{i}(T)$, we have
\begin{align}\label{5.47}
R^{Z}_{s,t}-R_{s,t}^{Z^{\eta}}&=\int_{s}^{t}S(t-r)G(y_{r})d\theta_{-T}\boldsymbol{W}_{r}(\omega)-\int_{s}^{t}S(t-r)G(y^{\eta}_{r})d\theta_{-T}\boldsymbol{W}^{\eta}_{r}(\omega)\nonumber\\
&-\!S(t\!-\!s)\!(G(y_{s})\theta_{-T}W_{s,t}(\omega)\!+\!DG(y_{s})G(y_{s})\theta_{-T}\mathbb{W}_{s,t}(\omega)\!\nonumber\\
&-\!G(y^{\eta}_{s})\theta_{-T}W^{\eta}_{s,t}(\omega)\!-\!DG(y^{\eta}_{s})G(y^{\eta}_{s})\theta_{-T}\mathbb{W}^{\eta}_{s,t}(\omega))\nonumber\\
&+\left(S(t-s)-Id\right)(G(y_{s})\theta_{-T}W_{s,t}(\omega)-G(y^{\eta}_{s})\theta_{-T}W_{s,t}^{\eta}(\omega))\nonumber\\
&+S(t-s)(D(y_{s})G(y_{s})\theta_{-T}\mathbb{W}_{s,t}(\omega)-DG(y^{\eta}_{s})G(y^{\eta}_{s})\theta_{-T}\mathbb{W}_{s,t}^{\eta}(\omega))\nonumber\\
&+\left(S(t-s)-Id\right)\bigg(\int_{0}^{s}S(s-r)G(y_{r})d\theta_{-T}\boldsymbol{W}_{r}(\omega)\nonumber\\
&~~~~~~~~~~~~~~~-\int_{0}^{s}S(s-r)G(y^{\eta}_{r})d\theta_{-T}\boldsymbol{W}^{\eta}_{r}(\omega)\bigg).
\end{align}
By Lemma \ref{lemma 5.3}, \ref{lemma A.3},  \ref{lemma 5.4}, we obtain
\begin{small}
\begin{align}\label{5.48}
\bigg\|&\int_{s}^{t}S(t-r)G(y_{r})d\theta_{-T}\boldsymbol{W}_{r}(\omega)-\int_{s}^{t}S(t-r)G(y^{\eta}_{r})d\theta_{-T}\boldsymbol{W}^{\eta}_{r}(\omega)\nonumber\\
&~~~~-S(t\!-\!s)(G(y_{s})\theta_{-T}W_{s,t}(\omega)\!+\!DG(y_{s})G(y_{s})\theta_{-T}\mathbb{W}_{s,t}(\omega)\!\nonumber\\
&~~~~-\!G(y^{\eta}_{s})\theta_{-T}W^{\eta}_{s,t}(\omega)\!-\!DG(y^{\eta}_{s})G(y^{\eta}_{s})\theta_{-T}\mathbb{W}^{\eta}_{s,t}(\omega))\bigg\|_{-\theta}\nonumber\\
&~~\leq \bigg(d_{\alpha,I_{i}(T)}(\theta_{-T}\boldsymbol{W}(\omega),\theta_{-T}\boldsymbol{W}^{\eta}(\omega))\|G(y),DG(y)G(y)\|_{W,2\alpha,-\sigma,I_{i}(T)}\nonumber\\
&~~~~~~~~~~~~~~~~~~~~~~~~~~~~~~~~~~~~~~~~~+d_{2\alpha,-\sigma,I_{i}(T)}(G(y),G(y^{\eta})\|\theta_{-T}\boldsymbol{W}^{\eta}(\omega)\|_{\alpha,I_{i}(T)}\bigg)(t-s)^{\alpha+\theta-\sigma}\nonumber\\
&~~\leq C\mu (\|y,y^{\prime}\|_{W,2\alpha,0,I_{i}(T)}+\|y^{\eta},(y^{\eta})^{\prime}\|_{W^\eta,2\alpha,0,I_{i}(T)}+1)^{2}\nonumber\\
&~~~~\times(1+\|\theta_{-T}\boldsymbol{W}(\omega)\|_{\alpha,I_{i}(T)}+\|\theta_{-T}\boldsymbol{W}^{\eta}(\omega)\|_{\alpha,I_{i}(T)})^{2}\nonumber\\
&~~~~\times\bigg(d_{\alpha,I_{i}(T)}(\theta_{-T}\boldsymbol{W}(\omega),\theta_{-T}\boldsymbol{W}^{\eta}(\omega))+d_{2\alpha,0,I_{i}(T)}(y,y^{\eta})\bigg)(t-s)^{\alpha+\theta-\sigma}.
\end{align}
\end{small}
By \eqref{2.2*}, Lemma \ref{lemma A.3}, \ref{lemma 5.4},  we have
\begin{align}\label{5.49}
  &\|\left(S(t-s)-Id\right)(G(y_{s})\theta_{-T}W_{s,t}(\omega)-G(y^{\eta}_{s})\theta_{-T}W_{s,t}^{\eta}(\omega))\|_{-\theta}\nonumber\\
  &\leq \|S(t-s)-Id\|_{\mathcal{L}(\mathcal{B}_{-\sigma},\mathcal{B}_{-\theta})}\|G(y_{s})\theta_{-T}W_{s,t}(\omega)-G(y^{\eta}_{s})\theta_{-T}W_{s,t}^{\eta})(\omega)\|_{-\sigma}\nonumber\\
  &\leq C|t-s|^{\theta+\alpha-\sigma}\bigg[\|G(y),DG(y)G(y)\|_{W,2\alpha,-\sigma,I_{i}(T)}d_{\alpha,I_{i}(T)}(\theta_{-T}\boldsymbol{W}(\omega),\theta_{-T}\boldsymbol{W}^{\eta}(\omega))\nonumber\\
  &~~~~~~~~~~~~~~~~~~~~~~~~~~~+d_{2\alpha,-\sigma,I_{i}(T)}d(G(y),G(y^{\eta}))d_{\alpha,I_{i}(T)} (\theta_{-T}\boldsymbol{W}^{\eta}(\omega),0)\bigg]\nonumber\\
  &\leq C\mu (\|y,y^{\prime}\|_{W,2\alpha,0,I_{i}(T)}+\|y^{\eta},(y^{\eta})^{\prime}\|_{W^\eta,2\alpha,0,I_{i}(T)}+1)^{2}\nonumber\\
  &~~~~\times(1+\|\theta_{-T}\boldsymbol{W}(\omega)\|_{\alpha,I_{i}(T)}+\|\theta_{-T}\boldsymbol{W}^{\eta}(\omega)\|_{\alpha,I_{i}(T)})^{2}\nonumber\\
&~~~~\times\bigg(d_{\alpha,I_{i}(T)}(\theta_{-T}\boldsymbol{W}(\omega),\theta_{-T}\boldsymbol{W}^{\eta}(\omega))+d_{2\alpha,0,I_{i}(T)}(y,y^{\eta})\bigg)(t-s)^{\alpha+\theta-\sigma}.
\end{align}
Using \eqref{2.1*}, Lemma \ref{lemma A.3}, \ref{lemma 5.4},   we get
\begin{align}\label{5.50}
 &\|S(t-s)(D(y_{s})G(y_{s})\theta_{-T}\mathbb{W}_{s,t}(\omega)-DG(y^{\eta}_{s})G(y^{\eta}_{s})\theta_{-T}\mathbb{W}_{s,t}^{\eta}(\omega))\|_{-\theta}\nonumber\\
 &~~~~\leq C(t-s)^{\theta-2\alpha}\|DG(y_s)G(y_{s})\theta_{-T}\mathbb{W}_{s,t}(\omega)-DG(y_s^\eta)G(y^{\eta}_{s})\theta_{-T}\mathbb{W}_{s,t}^{\eta}(\omega)\|_{-2\alpha}\nonumber\\
 &~~~~\leq C(t-s)^{\theta}\biggr(\|DG(y_s)G(y_s)\|_{-2\alpha}d_{\alpha,I_{i}(T)}(\theta_{-T}\boldsymbol{W}(\omega),\theta_{-T}\boldsymbol{W}^{\eta}(\omega))\nonumber\\
 &~~~~~~~~~~+d_{\alpha,I_{i}(T)} (\theta_{-T}\boldsymbol{W}^{\eta}(\omega),0)\|DG(y_s)G(y_s)-DG(y^{\eta}_s)G(y^\eta_s)\|_{-2\alpha}\biggr)\nonumber\\
 &~~~~\leq C\mu (\|y,y^{\prime}\|_{W,2\alpha,0,I_{i}(T)}+\|y^{\eta},(y^{\eta})^{\prime}\|_{W^\eta,2\alpha,0,I_{i}(T)}+1)^{2}\nonumber\\
 &~~~~~~~~~~\times(1+\|\theta_{-T}\boldsymbol{W}(\omega)\|_{\alpha,I_{i}(T)}+\|\theta_{-T}\boldsymbol{W}^{\eta}(\omega)\|_{\alpha,I_{i}(T)})^{2}\nonumber\\
&~~~~~~~~~~\times \bigg(d_{\alpha,I_{i}(T)}(\theta_{-T}\boldsymbol{W}(\omega),\theta_{-T}\boldsymbol{W}^{\eta}(\omega))+d_{2\alpha,0,I_{i}(T)}(y,y^{\eta})\bigg)(t-s)^{\theta}.
\end{align}
For the last term \eqref{5.47}, by \eqref{5.41}, \eqref{5.43}, we obtain
\begin{align}\label{5.51}
&\left\|\left(S(t-s)-Id\right)\left(\int_{0}^{s}S(s-r)G(y_{r})d\theta_{-T}\boldsymbol{W}_{r}(\omega)-\int_{0}^{s}S(s-r)G(y^{\eta}_{r})d\theta_{-T}\boldsymbol{W}^{\eta}_{r}(\omega)\right)\right\|_{-\theta}\nonumber\\
&\leq (t-s)^{\theta}\Bigg(\sum_{m=1}^{i-1}C\mu e^{-\lambda(T_{i-1}-T_{m})}(\|y,y^{\prime}\|_{W,2\alpha,0,I_{m}(T)}+\|y^{\eta},(y^{\eta})^{\prime}\|_{W^\eta,2\alpha,0,I_{m}(T)}+1)^{2}\nonumber\\
&~~~~\times(1+\|\theta_{-T}\boldsymbol{W}(\omega)\|_{\alpha,I_{m}(T)}+\|\theta_{-T}\boldsymbol{W}^{\eta}(\omega)\|_{\alpha,I_{m}(T)})^{2}\nonumber\\
&~~~~\times\bigg(d_{\alpha,I_{m}(T)}(\theta_{-T}\boldsymbol{W}(\omega),\theta_{-T}\boldsymbol{W}^{\eta}(\omega))+d_{2\alpha,0,I_{m}(T)}(y,y^{\eta})\bigg)\nonumber\\
& ~~~~+C\mu\bigg[ (\|y,y^{\prime}\|_{W,2\alpha,0,I_{i}(T)}+\|y^{\eta},(y^{\eta})^{\prime}\|_{W^\eta,2\alpha,0,I_{i}(T)}+1)^{2}\nonumber\\
&~~~~\times(1+\|\theta_{-T}\boldsymbol{W}(\omega)\|_{\alpha,I_{i}(T)}+\|\theta_{-T}\boldsymbol{W}^{\eta}(\omega)\|_{\alpha,I_{i}(T)})^{2}\nonumber\\
&~~~~\times\bigg(d_{\alpha,I_{i}(T)}(\theta_{-T}\boldsymbol{W}(\omega),\theta_{-T}\boldsymbol{W}^{\eta}(\omega))+d_{2\alpha,0,I_{i}(T)}(y,y^{\eta})\bigg)\bigg]\Bigg).
\end{align}
Combining  \eqref{5.47}-\eqref{5.51}, we have
\begin{small}
\begin{align}\label{5.52}
  &\|R^{y}-R^{y^{\eta}}\|_{\theta,-\theta,I_{i}(T)}\leq \sum_{m=1}^{i-1}C\mu\bigg[ e^{-\lambda(T_{i-1}-T_{m})}(\|y,y^{\prime}\|_{W,2\alpha,0,I_{m}(T)}+\|y^{\eta},(y^{\eta})^{\prime}\|_{W^\eta,2\alpha,0,I_{m}(T)}+1)^{2}\nonumber\\
&~~~~~~~~~~~~~~~\times(1+\|\theta_{-T}\boldsymbol{W}(\omega)\|_{\alpha,I_{m}(T)}+\|\theta_{-T}\boldsymbol{W}^{\eta}(\omega)\|_{\alpha,I_{m}(T)})^{2}\nonumber\\
&~~~~~~~~~~~~~~~\times\bigg(d_{\alpha,I_{m}(T)}(\theta_{-T}\boldsymbol{W}(\omega),\theta_{-T}\boldsymbol{W}^{\eta}(\omega))+d_{2\alpha,0,I_{m}(T)}(y,y^{\eta})\bigg)\bigg]\nonumber\\
&~~~~~~~~~~~ ~~~~+C\mu\bigg[ (\|y,y^{\prime}\|_{W,2\alpha,0,I_{i}(T)}+\|y^{\eta},(y^{\eta})^{\prime}\|_{W^\eta,2\alpha,0,I_{i}(T)}+1)^{2}\nonumber\\
&~~~~~~~~~~~~~~~\times(1+\|\theta_{-T}\boldsymbol{W}(\omega)\|_{\alpha,I_{i}(T)}+\|\theta_{-T}\boldsymbol{W}^{\eta}(\omega)\|_{\alpha,I_{i}(T)})^{2}\nonumber\\
&~~~~~~~~~~~~~~~\times\bigg(d_{\alpha,I_{i}(T)}(\theta_{-T}\boldsymbol{W}(\omega),\theta_{-T}\boldsymbol{W}^{\eta}(\omega))+d_{2\alpha,0,I_{i}(T)}(y,y^{\eta})\bigg)\bigg].
\end{align}
\end{small}
Thus,
\begin{align}\label{5.53}
 &d_{2\alpha,0,I_{i}(T)}\left(\int_{0}^{\cdot}S(\cdot-r)G(y_{r})d\theta_{-T}\boldsymbol{W}_{r}(\omega),\int_{0}^{\cdot}S(\cdot-r)G(y^{\eta}_{r})d\theta_{-T}\boldsymbol{W}^{\eta}_{r}(\omega)\right)\nonumber\\
  &~~~~\leq \sum_{m=1}^{i-1}C\mu\bigg[ e^{-\lambda(T_{i-1}-T_{m})}(\|y,y^{\prime}\|_{W,2\alpha,0,I_{m}(T)}+\|y^{\delta},(y^{\eta})^{\prime}\|_{W^\eta,2\alpha,0,I_{m}(T)}+1)^{2}\nonumber\\
&~~~~~~\times(1+\|\theta_{-T}\boldsymbol{W}(\omega)\|_{\alpha,I_{m}(T)}+\|\theta_{-T}\boldsymbol{W}^{\eta}(\omega)\|_{\alpha,I_{m}(T)})^{2}\nonumber\\
&~~~~~~\times\bigg(d_{\alpha,I_{m}(T)}(\theta_{-T}\boldsymbol{W}(\omega),\theta_{-T}\boldsymbol{W}^{\eta}(\omega))+d_{2\alpha,0,I_{m}(T)}(y,y^{\eta})\bigg)\bigg]\nonumber\\
& ~~~~~~+C\mu\bigg[ (\|y,y^{\prime}\|_{W,2\alpha,0,I_{i}(T)}+\|y^{\eta},(y^{\eta})^{\prime}\|_{W^\eta,2\alpha,0,I_{i}(T)}+1)^{2}\nonumber\\
&~~~~~~\times(1+\|\theta_{-T}\boldsymbol{W}(\omega)\|_{\alpha,I_{i}(T)}+\|\theta_{-T}\boldsymbol{W}^{\eta}(\omega)\|_{\alpha,I_{i}(T)})^{2}\nonumber\\
&~~~~~~\times\bigg(d_{\alpha,I_{i}(T)}(\theta_{-T}\boldsymbol{W}(\omega),\theta_{-T}\boldsymbol{W}^{\eta}(\omega))+d_{2\alpha,0,I_{i}(T)}(y,y^{\eta})\bigg)\bigg].
\end{align}
Furthermore, in terms of \eqref{5.27}, \eqref{5.39}, \eqref{5.53}, we have
\begin{align}\label{5.54}
&d_{2\alpha,0,I_{i}(T)}(y,y^{\eta})\leq \sum_{m=1}^{i-1}C\mu\bigg[ e^{-\lambda(T_{i-1}-T_{m})}(\|y,y^{\prime}\|_{W,2\alpha,0,I_{m}(T)}+\|y^{\eta},(y^{\eta})^{\prime}\|_{W^\eta,2\alpha,0,I_{m}(T)}+1)^{2}\nonumber\\
&~~~~~~\times(1+\|\theta_{-T}\boldsymbol{W}(\omega)\|_{\alpha,I_{m}(T)}+\|\theta_{-T}\boldsymbol{W}^{\eta}(\omega)\|_{\alpha,I_{m}(T)})^{2}\nonumber\\
&~~~~~~\times\bigg(d_{\alpha,I_{m}(T)}(\theta_{-T}\boldsymbol{W}(\omega),\theta_{-T}\boldsymbol{W}^{\eta}(\omega))+d_{2\alpha,0,I_{m}(T)}(y,y^{\eta})\bigg)\bigg]\nonumber\\
&~~~~~~ +C\mu\bigg[ (\|y,y^{\prime}\|_{W,2\alpha,0,I_{i}(T)}+\|y^{\eta},(y^{\delta})^{\prime}\|_{W^\eta,2\alpha,0,I_{i}(T)}+1)^{2}\nonumber\\
&~~~~~~\times(1+\|\theta_{-T}\boldsymbol{W}(\omega)\|_{\alpha,I_{i}(T)}+\|\theta_{-T}\boldsymbol{W}^{\eta}(\omega)\|_{\alpha,I_{i}(T)})^{2}\nonumber\\
&~~~~~~\times\bigg(d_{\alpha,I_{i}(T)}(\theta_{-T}\boldsymbol{W}(\omega),\theta_{-T}\boldsymbol{W}^{\eta}(\omega))+d_{2\alpha,0,I_{i}(T)}(y,y^{\eta})\bigg)\bigg]\nonumber\\
&~~~~~~+Ce^{-\lambda T_{i-1}(\theta_{-T}\boldsymbol{W}(\omega))}\|y_{0}-y^{\eta}_{0}\|.
\end{align}
For sufficiently small $\eta>0$ , due to the boundedness  of  $\|y,y^{\prime}\|_{W,2\alpha,0,I_{m}(T)}$ and $\|y^{\eta},(y^{\eta})^{\prime}\|_{W^{\eta},2\alpha,0,I_{m}(T)}$ for any $m$. Then \eqref{5.54} can be reduced  to as follows
\begin{align}\label{5.55}
&d_{2\alpha,0,I_{i}(T)}(y,y^{\eta})\leq \sum_{m=1}^{i-1}C\mu e^{-\lambda(T_{i-1}-T_{m})}\bigg(d_{\alpha,I_{m}(T)}(\theta_{-T}\boldsymbol{W}(\omega),\theta_{-T}\boldsymbol{W}^{\eta}(\omega))+d_{2\alpha,0,I_{m}(T)}(y,y^{\eta})\bigg)\nonumber\\
& +\!C\mu \bigg(\!d_{\alpha,I_{i}(T)}(\theta_{-T}\boldsymbol{W}(\omega),\theta_{-T}\boldsymbol{W}^{\eta}(\omega))\!+\!d_{2\alpha,0,I_{i}(T)}(y,y^{\eta})\!\bigg)\!+\!Ce^{-\!\lambda \!T_{i-1}}\|y_{0}\!-\!y^{\eta}_{0}\|.
\end{align}
Choose $\mu$ sufficiently small  such that
\begin{align*}\label{5.56}
d_{2\alpha,0,I_{i}(T)}(y,y^{\eta})&\leq \sum_{m=1}^{i-1}\frac{C\mu}{1-C\mu} e^{-\lambda(T_{i-1}-T_{m})}\bigg(d_{\alpha,I_{m}(T)}(\theta_{-T}\boldsymbol{W}(\omega),\theta_{-T}\boldsymbol{W}^{\eta}(\omega))+d_{2\alpha,0,I_{m}(T)}(y,y^{\eta})\bigg)\nonumber\\
& +\frac{C\mu}{1-C\mu} d_{\alpha,I_{i}(T)}(\theta_{-T}\boldsymbol{W}(\omega),\theta_{-T}\boldsymbol{W}^{\eta}(\omega))+\frac{C}{1-C\mu}e^{-\lambda T_{i-1}}\|y_{0}-y^{\eta}_{0}\|.
\end{align*}
In order to estimate the above inequality,  we use the property of stopping times
$$T_{i-1}(\theta_{-T}\boldsymbol{W}(\omega))-T_{m}(\theta_{-T}\boldsymbol{W}(\omega))=T_{i-1-m}(\theta_{T_{m}-T}\boldsymbol{W}(\omega))=-T_{-i+1+m}(\theta_{T_{i-1}-T}\boldsymbol{W}(\omega)),$$ and the discrete Gr\"{o}nwall inequality\cite{holte2009discrete}, i.e. for the nonnegative sequences $\{y_{n}\}$ and $\{g_{n}\}$  which  satisfy                                                                                                                                                                             \begin{align*}
y_{n}\leq C+\sum_{j=0}^{n-1}g_{j}y_{j},
\end{align*}
where $C>0$. Then
$$y_{n}\leq C\prod_{j=0}^{n-1}(1+g_{j}).$$
Then by the Corollary \ref{corollary 5.1*}, there exists $\eta(\epsilon)$ for any  $\epsilon>0$, such that  $\eta<\eta(\epsilon)$
$$d_{\alpha,I_{m}(T)}(\theta_{-T}\boldsymbol{W}(\omega),\theta_{-T}\boldsymbol{W}^{\eta}(\omega))<\epsilon,~m\in\mathbb{Z}$$
where the inequality is uniform for $m$. Furthermore, by the property of stopping times we have
\begin{align*}
\sum_{m=1}^{i-1}e^{-\lambda(T_{i-1}-T_{m})}&=\sum_{m=1}^{i-1}e^{\lambda T_{-i+1+m}(\theta_{T_{i-1}-T}\boldsymbol{W}(\omega))}\nonumber\\
&=\sum_{m=2-i}^{0}e^{\lambda T_{m}(\theta_{T_{i-1}-T}\boldsymbol{W}(\omega))}\nonumber\\
&\leq \sum_{m=-\infty}^{0}e^{m\lambda(d-\epsilon)}\leq C.
\end{align*}
Finally, we claim that $\prod_{m=1}^{i-1}(1+\frac{C\mu}{1-C\mu}e^{-m\lambda(d-\epsilon)})$  is uniformly finite for $i$. Then we can obatin
\begin{align*}
 d_{2\alpha,0,I_{i}(T)}(y,y^{\eta})\rightarrow 0, \quad \eta\rightarrow 0.
\end{align*}
\end{proof}
\subsection*{Acknowledgements}

{
	The authors would like to thank Prof. B. Schmalfuss for many useful discussions and comments.
}
\bibliographystyle{abbrv}

\end{document}